\documentclass{amsart}

\usepackage{amsfonts, amssymb, amsmath, amsthm}                               
\usepackage{setspace}                                                         
\usepackage{geometry}                                                         
\usepackage{color}                                                            
\usepackage{graphicx}                                                         
\usepackage{slashed}                                                          

\newtheorem{theorem}{Theorem} 	      	      	                              
\newtheorem{corollary}[theorem]{Corollary}     	      	      	      	      
\newtheorem{lemma}[theorem]{Lemma}     	       	      	      	      	      
\newtheorem{proposition}[theorem]{Proposition} 	      	      	      	      
\newtheorem{definition}[theorem]{Definition} 	      	      	                
\newtheorem{problem}[theorem]{Problem} 	      	      	                      
\newtheorem{remark}[theorem]{Remark}                                          
\newtheorem{example}[theorem]{Example}                                        

\numberwithin{equation}{section}                                              
\numberwithin{theorem}{section}                                               
\numberwithin{figure}{section}                                                

\newcommand{\mf}[1]{\mathfrak{#1}}                                            
\newcommand{\mc}[1]{\mathcal{#1}}                                             

\newcommand{\N}{\mathbb{N}}                                                   
\newcommand{\R}{\mathbb{R}}                                                   
\newcommand{\Sph}{\mathbb{S}}                                                 

\newcommand{\nasla}{\slashed{\nabla}}                                         

\allowdisplaybreaks

\def\comp{1}

\begin{document}


\title[Observability on Time-Dependent Domains]{On Carleman and Observability Estimates for Wave Equations on Time-Dependent Domains}

\author{Arick Shao}
\address{School of Mathematical Sciences\\
Queen Mary University of London\\
London E1 4NS\\ United Kingdom}
\email{a.shao@qmul.ac.uk}

\begin{abstract}
We establish new Carleman estimates for the wave equation, which we then apply to derive novel observability inequalities for a general class of linear wave equations.
The main features of these inequalities are that (a) they apply to a fully general class of \emph{time-dependent domains}, with \emph{timelike moving boundaries}, (b) they apply to linear wave equations in \emph{any spatial dimension} and with \emph{general time-dependent lower-order coefficients}, and (c) they allow for \emph{smaller time-dependent regions of observations} than allowed from existing Carleman estimate methods.
As a standard application, we establish exact controllability for general linear waves, again in the setting of time-dependent domains and regions of control.
\end{abstract}


\setcounter{tocdepth}{2}

\maketitle



\section{Introduction} \label{sec.intro}

In this article, we establish new Carleman estimates for the wave equation using a geometric approach.
The main objective is to apply these estimates in order to derive novel observability inequalities for general linear wave equations, with the following features:
\begin{enumerate}
\item[(I)] The estimates apply to a general class of \emph{time-dependent domains}, with \emph{moving boundaries}.

\item[(II)] The estimates apply to wave equations in \emph{any spatial dimension}.

\item[(III)] The estimates apply to general linear waves with \emph{time-dependent lower-order coefficients}.

\item[(IV)] The estimates apply for a wide variety of \emph{time-dependent observation regions} that are smaller than those in standard Carleman-based observability inequalities.
\end{enumerate}
As a standard application of these observability estimates, we establish the exact controllability of linear waves on the same general class of time-dependent domains.
Again, the region of control is allowed to be time-dependent proper subsets of regions found in classical results.

While the present paper deals only with wave equations on $\R^{1+n}$ (that is, on Minkowski spacetime), the main Carleman estimates are proved using ideas and intuitions from Lorentzian geometry.
As a result, the techniques presented here will also form the foundations for studying analogous questions for geometric wave equations on Lorentzian manifolds.

\subsection{Background} \label{sec.intro_bg}

In the context of evolutionary differential equations, the question of \emph{exact controllability} is concerned with whether one can drive its solutions from any prescribed initial state to any desired final state at a later time, under the constraint that only some limited parameters in the system---the \emph{controls}---can be set.
If solutions of the differential equation represents a physical system, then the above asks whether one can completely govern the system through its controls.

\subsubsection{Controllability of Waves}

To be more specific, let us consider a general linear wave equation
\begin{equation}
\label{eq.intro_wave} - \partial_{tt}^2 \phi + \Delta_x \phi + \mc{X}^t \partial_t \phi + \mc{X}^x \cdot \nabla_x \phi + V \phi = 0 \text{,}
\end{equation}
on a bounded (spatial) domain $\Omega \subseteq \R^n$.
In this setting, the initial state and the final state each correspond to a pair of functions, representing the values of $\phi$ and $\partial_t \phi$ at given times.
There are many possible choices for the control, with some common examples being the following:
\begin{itemize}
\item \emph{Interior controllability}: An additional forcing term $F$ on the right-hand side of \eqref{eq.intro_wave}.

\item \emph{Dirichlet boundary controllability}: Part of the Dirichlet boundary data for \eqref{eq.intro_wave}.

\item \emph{Neumann boundary controllability}: Part of the Neumann boundary data for \eqref{eq.intro_wave}.
\end{itemize}
In general, the methods for attacking all of the above cases are fairly similar.
Thus, for conciseness, we focus our attention in this paper solely on Dirichlet boundary controllability:

\begin{problem} \label{prb.intro_control}
Let $\Omega$ be a bounded open subset of $\R^n$.
Fix an initial time $\tau_-$ and a final time $\tau_+$, as well as a subset $\Gamma \subseteq [ \tau_-, \tau_+ ] \times \partial \Omega$.
Given any initial and final data $( \phi_0^\pm, \phi_1^\pm ) \in L^2 ( \Omega ) \times H^{-1} ( \Omega )$, can one find some Dirichlet boundary control $\phi_b \in L^2 ( \R \times \partial \Omega )$, supported in $\Gamma$, such that the solution $\phi$ of \eqref{eq.intro_wave}, with initial and Dirichlet boundary data
\[
( \phi, \partial_t \phi ) |_{ t = \tau_- } = ( \phi_0^-, \phi_1^- ) \text{,} \qquad \phi |_{ \R \times \partial \Omega } = \phi_b \text{,}
\]
also achieves the final state
\[
( \phi, \partial_t \phi ) |_{ t = \tau_+ } = ( \phi_0^+, \phi_1^+ ) \text{?}
\]
\end{problem}

Controllability of wave equations, and Problem \ref{prb.intro_control} in particular, has been a topic of research for several decades.
Here, we give a brief, and by no means exhaustive, survey of some existing research.
For more complete discussions, see, for example, \cite{micu_zua:control_pde} and the references within.

Examples of early results include the works of Russell \cite{russe:control_wave_1, russe:control_wave_2}.
Modern treatments of controllability are derived from the \emph{Hilbert Uniqueness Method} (\emph{HUM}) of Lions \cite{lionj:control_hum, lionj:ctrlstab_hum}.
(This is also closely related to the more abstract functional analytic framework developed in \cite{dolec_russe:obs_control}.)

The main point is that by duality, controllability is equivalent to uniqueness properties of the adjoint differential equation.
More specifically, for the present setting, in order to establish an affirmative answer to Problem \ref{prb.intro_control}, the main hurdle is to prove the \emph{observability inequality}
\begin{equation}
\label{eq.intro_obs} \| \psi ( \tau_\pm ) \|_{ H^1 ( \Omega ) } + \| \partial_t \psi ( \tau_\pm ) \|_{ L^2 ( \Omega ) } \lesssim \| \partial_\nu \psi \|_{ L^2 ( \Gamma ) } \text{,}
\end{equation}
for solutions $\psi$ of the wave equation adjoint to \eqref{eq.intro_wave}.
Here, $\partial_\nu \psi$ denotes the Neumann boundary data for $\psi$, and $\Gamma$ is as described in Problem \ref{prb.intro_control}.
When the inequality \eqref{eq.intro_obs} holds, the HUM machinery then yields exact controllability, with Dirichlet boundary controls supported in $\Gamma$.

For wave equations, there is a fundamental obstruction preventing exact controllability that is due to finite speed of propagation.
Indeed, some minimum amount of time is required for information from the boundary controls to travel to all of $\Omega$.
In particular, this implies a lower bound on the timespan $\tau_+ - \tau_-$ (which depends on the region where the control is placed) required for any boundary controllability, and hence observability, result to hold.\footnote{This is in direct contrast to the heat and Schr\"odinger equations, which propagate at infinite speed. In these settings, controllability generally holds for arbitrarily small times (given correspondingly large controls).}

Many methods have been developed for proving the crucial observability estimate \eqref{eq.intro_obs}.
For simplicity, let us first restrict our attention to the free wave equation,
\begin{equation}
\label{eq.intro_wave_free} \Box \psi := - \partial_{tt} \psi + \Delta_x \psi = 0 \text{.}
\end{equation}

For wave equations in one spatial dimension, observability estimates have been proved for many PDEs using \emph{Fourier methods}.
In the context of \eqref{eq.intro_wave_free}, these results are often based on applications of Ingham's inequality \cite{ing:trig_ineq_series} and its generalizations to the Fourier series expansions of $\psi$.
Moreover, such methods are capable of retrieving the optimal timespan for control, as dictated by finite speed of propagation.
For detailed discussions of Fourier methods, the reader is referred to \cite{avdo_ivan:moments_control}.
We also remark that similar methods were recently used to prove \eqref{eq.intro_obs} in higher dimensions \cite{green_liu_mitko:harmonic_obs}.

Other proofs of \eqref{eq.intro_obs} used \emph{multiplier methods}, in which one integrates by parts an expression
\[
\int_{ [ \tau_-, \tau_+ ] \times \Omega } \Box \psi X \psi \text{,}
\]
where $X$ represents an appropriately chosen first-order operator.
Using this technique, Ho \cite{ho:obs_wave} showed that for a sufficiently large timespan $\tau_+ - \tau_-$, the observability estimate \eqref{eq.intro_obs} indeed holds for solutions $\psi$ of \eqref{eq.intro_wave_free}, for a control region $\Gamma$ of the form
\begin{equation}
\label{eq.intro_obs_region} \Gamma := ( \tau_-, \tau_+ ) \times \{ y \in \partial \Omega \mid ( y - x_0 ) \cdot \nu > 0 \} \text{,}
\end{equation}
where $x_0 \in \R^n$ is fixed, and where $\nu$ is the outward-pointing unit normal for $\Omega$.\footnote{Here, the timespan $\tau_+ - \tau_-$ depends on the choice of $x_0$.}
Note that \eqref{eq.intro_obs_region} contains all points $x \in \partial \Omega$ such that the ray in $\R^n$ emanating from $x_0$ and passing through $x$ is exiting $\Omega$ at $x$.
In particular, one needs not apply the control on all of $[ \tau_-, \tau_+ ] \times \partial \Omega$.

A modification of the above argument (see \cite{lionj:ctrlstab_hum}) resulted in a lower bound on the timespan,
\begin{equation}
\label{eq.intro_obs_timespan} \tau_+ - \tau_- > 2 \max_{ y \in \partial \Omega } | y - x_0 | \text{,}
\end{equation}
required for the inequality \eqref{eq.intro_obs} to hold, with $\Gamma$ again given by \eqref{eq.intro_obs_region}.
Furthermore, using ``rotated" multipliers, Osses \cite{osses:mult_control} obtained analogues of \eqref{eq.intro_obs} with different boundary regions $\Gamma$.

Multiplier techniques can also be used to treat some wave equations of the form \eqref{eq.intro_wave}, as well as some more general hyperbolic equations; see \cite{mira:hum_var_coeff}, for instance.
However, these results generally fail to recover lower bounds of the form \eqref{eq.intro_obs_timespan} for the required timespan.
For more detailed discussions on multiplier methods and its roles in observability estimates, see \cite{komo:control_stab}.

\eqref{eq.intro_obs} has also been established using \emph{microlocal methods}, yielding optimal results with regards to the requisite timespan and control region.
Of particular note is the seminal result of Bardos, Lebeau, and Rauch \cite{bard_leb_rauch:gcc}.
Consider now the boundary region
\begin{equation}
\label{eq.intro_obs_gcc} \Gamma := [ \tau_-, \tau_+ ] \times \Lambda \text{,} \qquad \Lambda \subseteq \partial \Omega \text{.}
\end{equation}
Then, \eqref{eq.intro_obs} holds for $\Gamma$ as in \eqref{eq.intro_obs_gcc} if and only if the \emph{geometric control condition} (\emph{GCC}) is satisfied: roughly, every null geodesic in $( \tau_-, \tau_+ ) \times \Omega$---with the condition that it is reflected whenever it hits the boundary $( \tau_-, \tau_+ ) \times \partial \Omega$---intersects some point of $\Gamma$; see \cite{bard_leb_rauch:gcc, burq_gera:control_gcc}.
Note that the necessity of the GCC follows from Gaussian beam constructions that propagate along these geodesics; see \cite{rals:gauss_beam}.

The result was generalized by Burq in \cite{burq:control_wave}, which reduced the regularity required.
A more modern and constructive proof of \eqref{eq.intro_obs} along similar directions was given in \cite{laur_leaut:obs_unif}.
Furthermore, \cite{lero_leb_terpo_trel:gcc_time} extended the above results to \emph{time-dependent} subsets $\Gamma$ of the boundary that satisfy the GCC; this further optimized the control regions required for \eqref{eq.intro_obs} to hold.

Microlocal methods can also be adapted to some wave equations \eqref{eq.intro_wave} with lower-order terms, with the caveat that the coefficients ($\mc{X}^t$, $\mc{X}^x$, $V$) are \emph{time-independent}, or at most \emph{analytic} in the time variable.\footnote{This requirement of time-analyticity is a consequence of the unique continuation results of \cite{hor:uc_interp, robb_zuil:uc_interp, tat:uc_hh}.}
Moreover, these methods also apply to (time-independent) geometric wave operators $- \partial_{tt} + \Delta_h$ on product manifolds of the form $[ - \tau_-, \tau_+ ] \times M$, where $( M, h )$ is a Riemannian manifold with or without boundary, and where $\Delta_h$ is the $h$-Laplace-Beltrami operator on $M$.

\emph{Carleman estimates} represent yet another family of tools that has been useful for deriving observability.
These are, roughly, weighted (spacetime) integral inequalities which contain an additional free parameter $\lambda > 0$.
Historically, Carleman estimates have been used extensively for establishing unique continuation properties of various PDEs, in particular with coefficients that fail to be analytic.
This theory, pioneered by Carleman \cite{carl:uc_strong}, Calder\'on \cite{cald:unique_cauchy}, and H\"ormander \cite{hor:lpdo4}, among others, has been developed over several decades; we refer the reader to \cite{tat:notes_uc} for a general survey.

In the context of Problem \ref{prb.intro_control}, Carleman estimates are advantageous due to their robustness, in that they are applicable to a wide range of settings.
For instance, in contrast to multiplier methods, they allow one to treat wave equations \eqref{eq.intro_wave} with arbitrary lower-order coefficients (with sufficient regularity), while still achieving the some control regions of the form \eqref{eq.intro_obs_region} as well as the lower bound \eqref{eq.intro_obs_timespan}.
In particular, by taking the free parameter $\lambda$ in these estimates to be as large as necessary, one can ``absorb" away any potentially dangerous contributions from lower-order terms.

For example, Carleman estimates were applied toward proving observability and controllability of wave equations (with lower-order terms) in \cite{lasie_trigg_zhang:wave_global_uc, zhang:obs_wave_pot, zhang:obs_wave_lower}.
Additional adaptions of the Carleman estimate method can be found in \cite{baudo_debuh_erv:carleman_wave, fu_yong_zhang:ctrl_semilinear}, among many others.
For further discussions, see \cite{furs_iman:ctrl_evol}.

On one hand, methods based on Carleman estimates lack the precision of microlocal methods and do not achieve the GCC in general.
However, Carleman methods apply to a wider class of settings, including wave equations with \emph{time-dependent} lower-order coefficients, without any assumption of analyticity in time.
Moreover, Carleman estimates have been extended to geometric wave equations, given additional assumptions on the (time-independent) geometry; see, e.g., \cite{duy_zhang_zua:obs_opt, fu_yong_zhang:ctrl_semilinear, trigg_yao:carleman_wave_uc}.

For this article, we are particularly concerned with lower-order coefficients $\mc{X}^t$, $\mc{X}^x$, $V$ that vary in both space and time, \emph{without any presumption of analyticity}.
As we wish for our results to be as widely applicable as possible, we resort to Carleman estimates methods here.

Finally, we remark that the Carleman methods applied in \cite{baudo_debuh_erv:carleman_wave, fu_yong_zhang:ctrl_semilinear, lasie_trigg_zhang:wave_global_uc, zhang:obs_wave_pot, zhang:obs_wave_lower} dealt only with the case in which $x_0 \not\in \bar{\Omega}$.
(This was in contrast to multiplier methods, which allows for all $x_0 \in \R^n$.)
In this article, we will also remove this restriction and consider cases with $x_0 \in \Omega$.

\subsubsection{Non-Static Domains}

All the works described thus far have dealt with wave equations on a time-independent cylindrical domain, $[ \tau_-, \tau_+ ] \times \Omega$, with time-independent boundary $[ \tau_-, \tau_+ ] \times \partial \Omega$.
However, one can also pose the analogue of Problem \ref{prb.intro_control} in settings in which the domain, and hence the boundary, are moving in time.
To be more specific, we consider a spacetime domain of the form
\begin{equation}
\label{eq.intro_gtc} \mc{U} := \bigcup_{ \tau \in \R } ( \{ \tau \} \times \Omega_\tau ) \text{,}
\end{equation}
where the $\Omega_\tau$'s are bounded open subsets of $\R^n$ that vary smoothly with respect to $\tau$.
For example, this could represent a physical system that is itself accelerating.

In comparison to static settings, there has been a relatively small amount of research on time-dependent, non-cylindrical domains.
Below, we briefly survey of the some existing literature.

An early study that predated the HUM is that of Bardos and Chen \cite{bard_chen:ctrlstab_wave3}, which established interior controllability for free waves on domains $\mc{U}$ that are expanding in time.
In particular, the result was proved by deriving energy decay bounds and adopting stabilization techniques.
Furthermore, using geometric methods, the above result was extended to geometric (free) wave equations in \cite{lu_li_chen_yao:ctrlstab_wave_moving}, under additional assumptions on the background geometry.

Using the HUM, Miranda \cite{mira:control_var_boundary} established Dirichlet boundary controllability for free waves on a class of time-dependent domains with the following features:
\begin{itemize}
\item The domains are self-similar (i.e., each $\Omega_\tau$ is of the form $k ( \tau ) \cdot \Omega_0$ for some $k (\tau) > 0$).

\item The domain becomes ``asymptotically cylindrical" (roughly, $k' (\tau)$ decays for large times).

\item On the other hand, $\mc{U}$ needs not be expanding nor contracting at all times.
\end{itemize}
The result is established by applying a change of variables to convert the problem into that of boundary controllability for a more general hyperbolic equation on a time-independent, cylindrical domain.
The transformed problem is treated in \cite{mira:hum_var_coeff} using multiplier methods and the observation that the PDE asymptotes to the standard wave equation at large times.

The problem in one spatial dimension has been further studied by several authors in recent years.
For example, \cite{cui_jiang_wang:control_wave_fec, sun_li_lu:control_wave_moving} (see references within for earlier works) used multiplier methods to study the case in which the boundary is given by two timelike lines $\ell_0$ and $\ell$ (with $\ell_0$ generally taken to be vertical).
More recently, \cite{sengou:obs_control_wave, sengou:obs_control_wave2} considered these problems using Fourier methods.
In these particular cases, the optimal timespan required for control was obtained in some of the above articles.
Analogous problems with more general boundaries (that are not lines) have been considered; one example is \cite{wang_he_li:control_wave_noncyl}, which does not achieve the optimal timespan.

One of the primary goals of the present article is to tackle this problem of \emph{Dirichlet boundary controllability} of \emph{wave equations} (with lower-order coefficients varying non-analytically in both space and time) on \emph{time-dependent domains} in full generality.
In particular, we make no assumptions on the shape of our domain $\mc{U}$, aside from its boundary being timelike; in particular, $\mc{U}$ needs not be expanding, contracting, or self-similar in time.
To the author's knowledge, the present paper provides the first results regarding controllability of waves in moving domains in this generality.

These results will be established using a novel global Carleman estimate that is entirely supported in the exterior of a null cone.
Furthermore, we obtain the best required timespan for controllability that have been achieved via multiplier and Carleman estimate methods.

\subsection{The Main Results} \label{sec.intro_results}

In this section, we state, at least roughly, the main results of this paper.
More specifically, we begin by describing the main observability inequalities, as well as how they improve upon existing results.
We then discuss the new Carleman estimate that is used to derive these observability estimates, as well as the main ideas involved in its derivation.

Our results will apply to spatially bounded but time-dependent domains of the form \eqref{eq.intro_gtc} which also have a smooth timelike boundary.
We will refer to these domains throughout the paper as \emph{generalized timelike cylinders}; see Definition \ref{def.gtc} for a precise definition of these domains.

\subsubsection{A Preliminary Estimate for Free Waves} \label{sec.intro_results_prelim}

We begin by first presenting a ``warm-up" estimate that applies only to the free wave equation on generalized timelike cylinders.
While this estimate is, in multiple ways, strictly weaker than our main results, it does allow us to first discuss the effects of time-dependent domains on observability apart from the other details of the main results.

The ``warm up" estimate can be roughly stated as follows:

\begin{theorem} \label{thm.intro_obs_prelim}
Let $\mc{U}$ be a generalized timelike cylinder of the form \eqref{eq.intro_gtc}, fix a point $x_0 \in \R^n$, and fix ``initial" and ``final" times $\tau_\pm \in \R$, with $\tau_- < \tau_+$.
Moreover, we assume that 
\begin{equation}
\label{eq.intro_obs_prelim_ass} \tau_+ - \tau_- > R_+ + R_- \text{,} \qquad R_\pm := \sup_{ y \in \partial \Omega_{ \tau_\pm } } | y - x_0 | \text{,}
\end{equation}
and we let $t_0 \in ( \tau_-, \tau_+ )$ be such that\footnote{Note that \eqref{eq.intro_obs_prelim_ass} implies that such a $t_0 \in ( \tau_-, \tau_+ )$ satisfying \eqref{eq.intro_obs_prelim_t0} exists.}
\begin{equation}
\label{eq.intro_obs_prelim_t0} \tau_+ - t_0 > R_+ \text{,} \qquad t_0 - \tau_- > R_- \text{.}
\end{equation}
Then, for any smooth solution $\phi$ of
\begin{equation}
\label{eq.intro_obs_prelim_wave} ( - \partial_{tt}^2 \phi + \Delta_x \phi ) |_{ \mc{U} } = 0 \text{,} \qquad \phi |_{ \partial \mc{U} } = 0 \text{,}
\end{equation}
we have the observability estimate
\begin{equation}
\label{eq.intro_obs_prelim} \int_{ \mc{U} \cap \{ t = \tau_\pm \} } [ ( \partial_t \phi )^2 + | \nabla_x \phi |^2 + \phi^2 ] \lesssim \int_{ \Gamma_\ast } | \mc{N} \phi |^2 \text{,}
\end{equation}
where $\mc{N}$ denotes the \emph{Minkowski} outer-pointing unit normal of $\mc{U}$, and where
\begin{equation}
\label{eq.intro_obs_prelim_region} \Gamma_\ast := \{ ( \tau, y ) \in \partial \mc{U} \mid \tau_- < \tau < \tau_+ \text{, } \mc{N} f_\ast |_{ ( \tau, y ) } > 0 \} \text{,} \qquad f_\ast := \frac{1}{4} [ | x - x_0 |^2 - ( t - t_0 )^2 ] \text{.}
\end{equation}
\end{theorem}

Theorem \ref{thm.intro_obs_prelim} implies, via the standard HUM, a corresponding exact controllability result for free waves.
A rough statement of this can be expressed as follows:

\begin{corollary} \label{thm.intro_control_prelim}
Assume the definitions and hyptheses of Theorem \ref{thm.intro_obs_prelim}.
Then, the free wave equation \eqref{eq.intro_wave_free} is exactly controllable, with initial and final data on $\mc{U} \cap \{ t = \tau_- \}$ and $\mc{U} \cap \{ t = \tau_+ \}$, respectively, and with Dirichlet boundary control supported in $\Gamma_\ast$ (all in the appropriate spaces).
\end{corollary}

\begin{remark}
See Theorem \ref{thm.control_hum} for a precise statement of exact controllability in this context.
\end{remark}

A precise version of Theorem \ref{thm.intro_obs_prelim}, using notations developed in the article, is found in Theorem \ref{thm.obs_prelim}.
In particular, despite its preliminary nature, Theorem \ref{thm.intro_obs_prelim} already achieves the properties (I) and (II) listed in the beginning of the introduction, at least in the setting of free waves.

The proof of Theorem \ref{thm.intro_obs_prelim} is based on Lorentzian geometric adaptations of the classical multiplier estimate found in \cite{lionj:ctrlstab_hum}.
The key points of this proof are as follows:
\begin{itemize}
\item The multiplier $( x - x_0 ) \cdot \nabla_x$ in \cite{lionj:ctrlstab_hum} is now replaced by the (Minkowski) gradient of $f_\ast$.

\item The proof uses the divergence theorem for Lorentzian manifolds; see \cite[Appendix B.2]{wald:gr}.
\end{itemize}
For further details, the reader is referred to the proof of Theorem \ref{thm.obs_prelim}.

We now discuss some of the main features of Theorem \ref{thm.intro_obs_prelim} and its relations to previous literature.
First, observe that when $\mc{U}$ is time-independent, so that $R_+ = R_-$, then \eqref{eq.intro_obs_prelim_ass} reduces to the standard bound \eqref{eq.intro_obs_timespan} on the timespan.
Moreover, \eqref{eq.intro_obs_prelim_ass} can be argued in terms of finite speed of propagation: information placed on the boundary at time $t = \tau_-$ needs time at most $R_-$ to travel to $x_0$, and this then needs time at most $R_+$ to travel back to the boundary at $t = \tau_+$.

Next, since the domain $\mc{U}$ may be changing in time, the outer unit normal $\mc{N}$ may have a nonzero $t$-component.
Here, we use $\mc{N}$ to denote the \emph{Minkowski}, rather than Euclidean, normal to $\mc{N}$.
More specifically, if $\bar{\mc{N}} := ( \bar{\nu}^t, \nu )$ denotes the Euclidean normal vector field to $\mc{U}$, then
\begin{equation}
\label{eq.intro_obs_prelim_normal} \mc{N} := ( \nu^t, \nu ) = ( - \bar{\nu}^t, \nu ) \text{.}
\end{equation}

Note that the control region $\Gamma_\ast$ from \eqref{eq.intro_obs_prelim_region} is simply the classical region \eqref{eq.intro_obs_region}, except that the condition $( y - x_0 ) \cdot \nu > 0$ from \eqref{eq.intro_obs_region} is now replaced by $\mc{N} f_\ast |_{ ( \tau, y ) } > 0$.
In particular, observe that when $\mc{U}$ is time-independent, \eqref{eq.intro_obs_prelim_region} reduces precisely to the standard region \eqref{eq.intro_obs_region}, since
\begin{equation}
\label{eq.intro_obs_prelim_cylinder} \mc{N} f_\ast = \frac{1}{4} \mc{N} ( | x - x_0 |^2 ) = \frac{1}{2} \sum_{ i, j = 1 }^n ( x^i - x_0^i ) \nu^j \partial_{ x_j } ( x^i - x_0^i ) = \frac{1}{2} ( x - x_0 ) \cdot \nu \text{.}
\end{equation}

\begin{remark}
Whereas the standard observability estimates obtained from multipliers or Carleman estimates can be thought to be centered about the point $x_0$ in space, in Theorem \ref{thm.intro_obs_prelim}, one would view the estimate as being centered about the event $( t_0, x_0 )$ in the spacetime.
In fact, the function $f_\ast$ in \eqref{eq.intro_obs_region} precisely measures the squared Minkowski distance to $( t_0, x_0 )$.
However, as was seen in \eqref{eq.intro_obs_prelim_cylinder}, the contribution of ``$t_0$" is not seen unless the domain $\mc{U}$ is changing in time.
\end{remark}

More generally, one can compute that
\begin{equation}
\label{eq.intro_obs_prelim_noncyl} \mc{N} f_\ast = \frac{1}{2} ( x - x_0 ) \cdot \nu - \frac{1}{2} ( t - t_0 ) \cdot \nu^t \text{.}
\end{equation}
By the definition of $\mc{N}$ and $\nu^t$ in \eqref{eq.intro_obs_prelim_normal}, we can observe the following:
\begin{itemize}
\item If $\nu^t > 0$ at a point of $\partial \mc{U}$, then $\mc{U}$ is expanding in time at that point.

\item If $\nu^t < 0$ at a point of $\partial \mc{U}$, then $\mc{U}$ is shrinking in time at that point.
\end{itemize}
Therefore, \eqref{eq.intro_obs_prelim_noncyl} implies the following general principle when $t > t_0$:
\begin{itemize}
\item Where $\mc{U}$ is expanding, $\mc{N} f_\ast$ is less positive than $( x - x_0 ) \cdot \nu$.
Thus, the region $\Gamma_\ast$ of observation is smaller than the standard region \eqref{eq.intro_obs_region} near points where $\mc{U}$ is expanding.

\item Where $\mc{U}$ is shrinking, $\mc{N} f_\ast$ is more positive than $( x - x_0 ) \cdot \nu$.
Thus, the region $\Gamma_\ast$ of observation is larger than the standard region \eqref{eq.intro_obs_region} near points where $\mc{U}$ is shrinking.
\end{itemize}
On the other hand, for times $t < t_0$, the above relations are reversed: $\Gamma_\ast$ becomes larger wherever $\mc{U}$ is expanding, while $\Gamma_\ast$ becomes smaller wherever $\mc{U}$ is shrinking.

\subsubsection{Observability Estimates}

From the ``warm-up" observability estimate, Theorem \ref{thm.intro_obs_prelim}, one can already see (from \eqref{eq.intro_obs_prelim_region}) the effect of the moving boundary on the region of observation.
However, Theorem \ref{thm.intro_obs_prelim} fails to achieve the features (III) and (IV) at the beginning of the introduction:
\begin{itemize}
\item Theorem \ref{thm.intro_obs_prelim} only applies to the free wave equation \eqref{eq.intro_wave_free}, and not to general linear wave equations of the form \eqref{eq.intro_wave}, that is, with lower-order terms.

\item The region of observation can be even significantly improved from \eqref{eq.intro_obs_prelim_region}.
\end{itemize}

As mentioned before, Carleman estimate methods have been successful in handling the first point above, without any additional assumption of independence or analyticity in time for the lower-order coefficients.
Here, we will similarly extend Theorem \ref{thm.intro_obs_prelim} to linear wave equations of the form \eqref{eq.intro_wave}---hence achieving property (III)---using our upcoming Carleman estimates.

With regards to the second point above, recall that the preliminary estimate of Theorem \ref{thm.intro_obs_prelim} is in principle centered about an event $( t_0, x_0 )$ of the spacetime.
Another particularly novel consequence of the main Carleman estimates of this article is that the region of observation can be further restricted to the exterior $\mc{D}_\ast$ of the null cone about $( t_0, x_0 )$.
In fact, the main Carleman estimate of this paper (see Theorem \ref{thm.carleman_est}) is itself supported entirely on this exterior $\mc{D}_\ast$.

With the above in mind, we now give a rough statement of our main observability estimates.
For simplicity of exposition, we avoid stating the most general cases here.

\begin{theorem} \label{thm.intro_obs_main}
Let $\mc{U}$ be a generalized timelike cylinder of the form \eqref{eq.intro_gtc}, fix a point $x_0 \in \R^n$, and fix ``initial" and ``final" times $\tau_\pm \in \R$, with $\tau_- < \tau_+$.
In addition:
\begin{itemize}
\item Assume the bound \eqref{eq.intro_obs_prelim_ass} holds, and let $t_0 \in ( \tau_-, \tau_+ )$ such that \eqref{eq.intro_obs_prelim_t0} holds.

\item Fix $V, \mc{X}^t \in C^\infty ( \bar{\mc{U}} )$, and let $\mc{X}^x \in C^\infty ( \bar{\mc{U}}; \R^n )$.

\item Let $\mc{N}$ denote the \emph{Minkowski} outer-pointing unit normal of $\mc{U}$, let $f_\ast$ be defined as in \eqref{eq.intro_obs_prelim_region}, and let $\Gamma_\dagger$ denote the following subset of $\partial \mc{U}$:
\begin{equation}
\label{eq.intro_obs_main_region} \Gamma_\dagger := \{ ( \tau, y ) \in \partial \mc{U} \mid f_\ast |_{ ( \tau, y ) } > 0 \text{, } \mc{N} f_\ast |_{ ( \tau, y ) } > 0 \} \text{.}
\end{equation}
Moreover, let $\mc{Y}_\dagger$ be any neighborhood of $\bar{\Gamma}_\dagger$ in $\partial \mc{U}$.
\end{itemize}
Then, for any smooth solution $\phi$ of
\begin{equation}
\label{eq.intro_obs_main_wave} ( - \partial_{tt}^2 \phi + \Delta_x \phi + \mc{X}^t \partial_t \phi + \mc{X}^x \cdot \nabla_x \phi + V \phi ) |_{ \mc{U} } = 0 \text{,} \qquad \phi |_{ \partial \mc{U} } = 0 \text{,}
\end{equation}
we have the observability estimate
\begin{equation}
\label{eq.intro_obs_main} \int_{ \mc{U} \cap \{ t = \tau_\pm \} } [ ( \partial_t \phi )^2 + | \nabla_x \phi |^2 + \phi^2 ] \lesssim \int_{ \mc{Y}_\dagger } | \mc{N} \phi |^2 \text{.}
\end{equation}
\end{theorem}

The most general observability results can be found in Theorems \ref{thm.obs_ext} and \ref{thm.obs_int}, handling the cases in which $( t_0, x_0 )$ lies outside and inside $\bar{\mc{U}}$, respectively.
A precise version of Theorem \ref{thm.intro_obs_main}, which is a corollary of Theorems \ref{thm.obs_ext} and \ref{thm.obs_int}, can be found in Theorem \ref{thm.obs_combo}.
Note in particular that in contrast to previous results in the literature using Carleman estimate methods \cite{baudo_debuh_erv:carleman_wave, fu_yong_zhang:ctrl_semilinear, lasie_trigg_zhang:wave_global_uc, zhang:obs_wave_pot, zhang:obs_wave_lower}, we do not require that the point $( t_0, x_0 )$ lies outside of $\bar{\mc{U}}$ in Theorem \ref{thm.intro_obs_main}.

Observe that the assumptions of Theorem \ref{thm.intro_obs_main}, captured in \eqref{eq.intro_gtc}, \eqref{eq.intro_obs_prelim_ass}, and \eqref{eq.intro_obs_prelim_t0}, are essentially the same as in the ``warm-up" Theorem \ref{thm.intro_obs_prelim}.
The first point of departure is that \emph{Theorem \ref{thm.intro_obs_main} applies to the Dirichlet boundary problem for general linear wave equations of the form \eqref{eq.intro_wave}}.

\begin{remark}
Theorem \ref{thm.intro_obs_main}, as stated, applies only to smooth solutions $\phi$ of \eqref{eq.intro_obs_main_wave}, with smooth coefficients $\mc{X}^x$, $\mc{X}^t$, $V$.
However, as is standard, the regularities of both the solutions and the coefficients can be significantly lowered by examining more precisely the integrability and differentiability conditions required in the proofs throughout the article.
Since regularity is not presently a concern, we avoid discussing these points in this article for simplicity.
\end{remark}

The second difference between Theorems \ref{thm.intro_obs_prelim} and \ref{thm.intro_obs_main} is in the observation region, in particular, between $\Gamma_\ast$ in \eqref{eq.intro_obs_prelim_region} and $\Gamma_\dagger$ in \eqref{eq.intro_obs_main_region}.
In particular, the condition $\tau_- < t < \tau_+$ within $\Gamma_\ast$ (specifying that $\Gamma_\ast$ lies within the initial and final times) is replaced by the condition $f_\ast > 0$ in $\Gamma_\dagger$ (specifying that $\Gamma_\dagger$ lies in the exterior of the null cone about $( t_0, x_0 )$).
In Lorentzian geometric terms, \emph{Theorem \ref{thm.intro_obs_main} improves upon existing results by further restricting the observation region to events that are not causally related to $( t_0, x_0 )$.}
See Figure \ref{fig.intro_obs_main} for examples of graphical depictions of $\Gamma_\dagger$.

\begin{remark}
In Minkowski geometry, the \emph{null cone} about an event $( t_0, x_0 )$ is the set of all points in $\R^{1+n}$ satisfying $f_\ast = 0$, or equivalently, the condition $| t - t_0 | = | x - x_0 |$.
\end{remark}

\begin{remark}
The assumptions \eqref{eq.intro_obs_prelim_ass} and \eqref{eq.intro_obs_prelim_t0} imply the region $\Gamma_\dagger$ in \eqref{eq.intro_obs_main_region} is a proper subset of $\Gamma_\ast$ in \eqref{eq.intro_obs_prelim_region}, so Theorem \ref{thm.intro_obs_main} represents an improvement in terms of the observation region.
Note, however, that the observation region in Theorem \ref{thm.intro_obs_prelim} is exactly $\Gamma_\ast$, whereas the observation region $\mc{Y}_\dagger$ in Theorem \ref{thm.intro_obs_main} must be strictly larger than $\Gamma_\dagger$ (i.e., any neighborhood of $\bar{\Gamma}_\dagger$).
\end{remark}

\begin{figure}[t]
\begin{align*}
\includegraphics[scale=0.5]{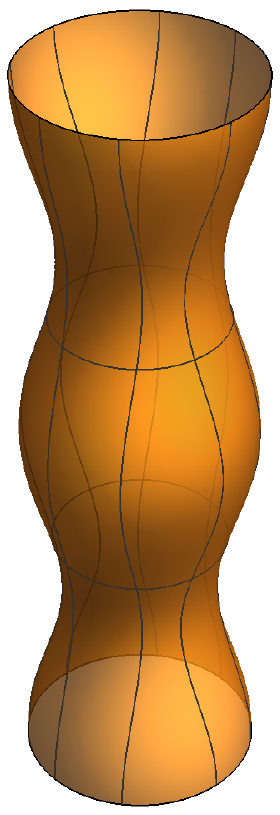} \qquad \includegraphics[scale=0.36]{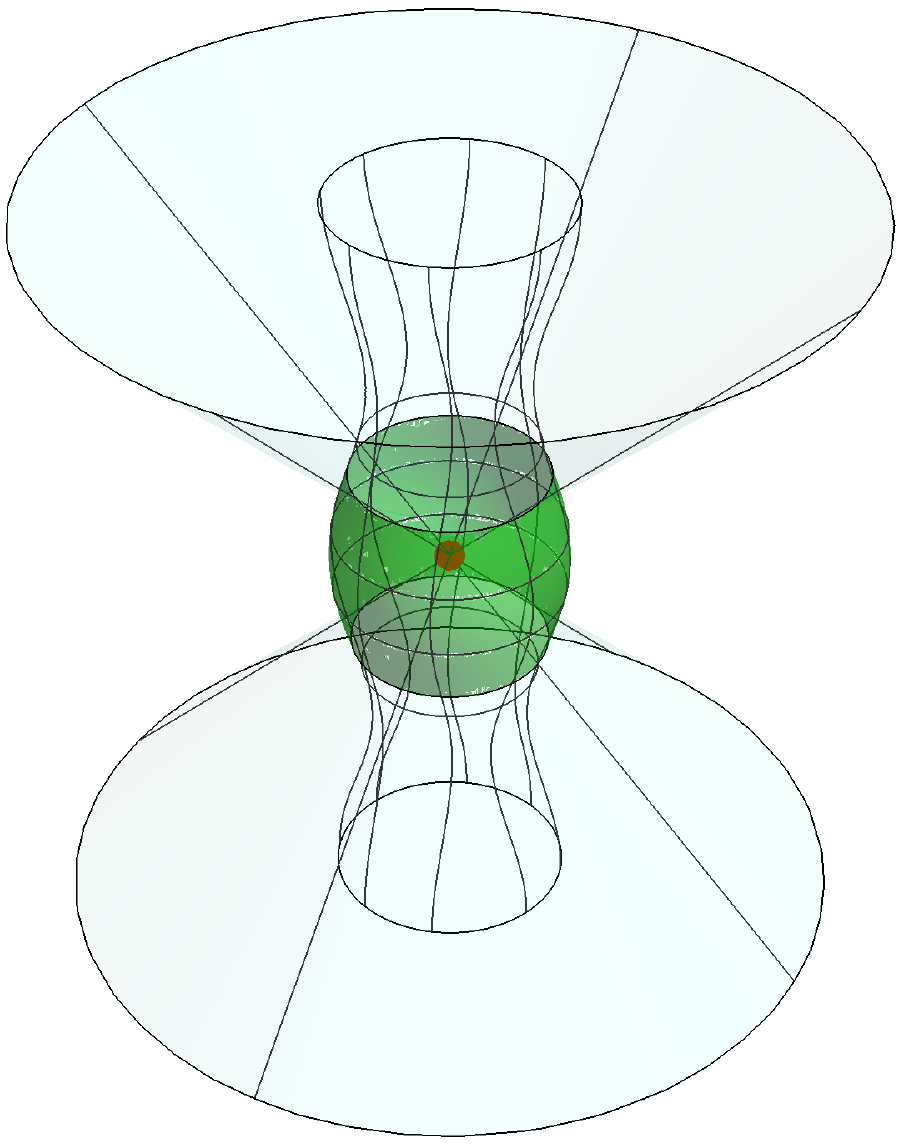} \qquad \includegraphics[scale=0.36]{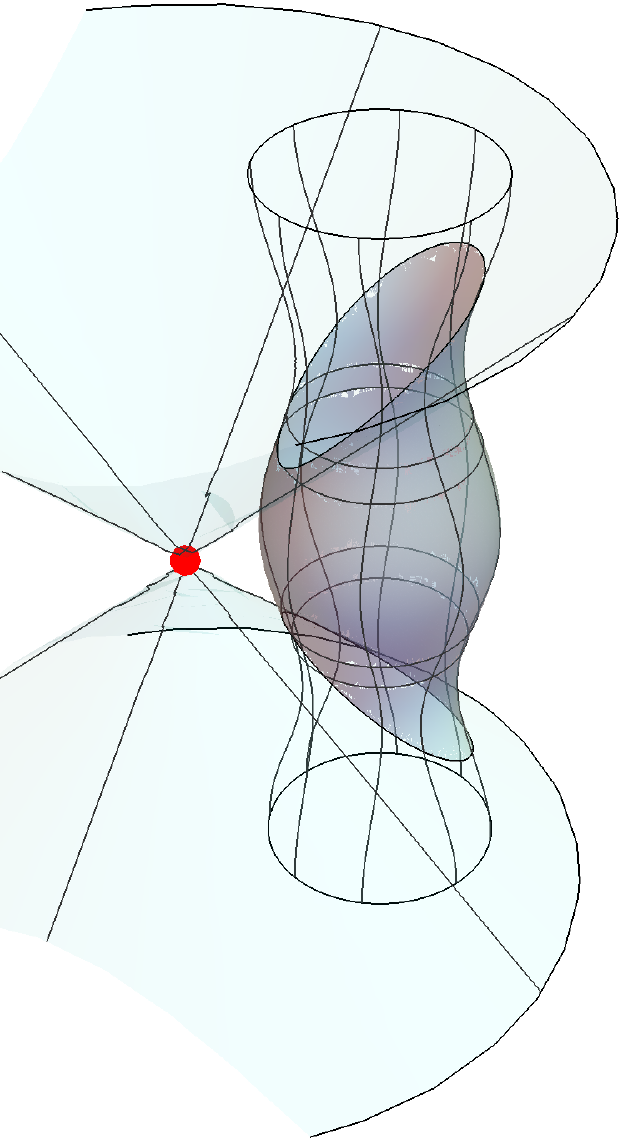} \qquad \includegraphics[scale=0.36]{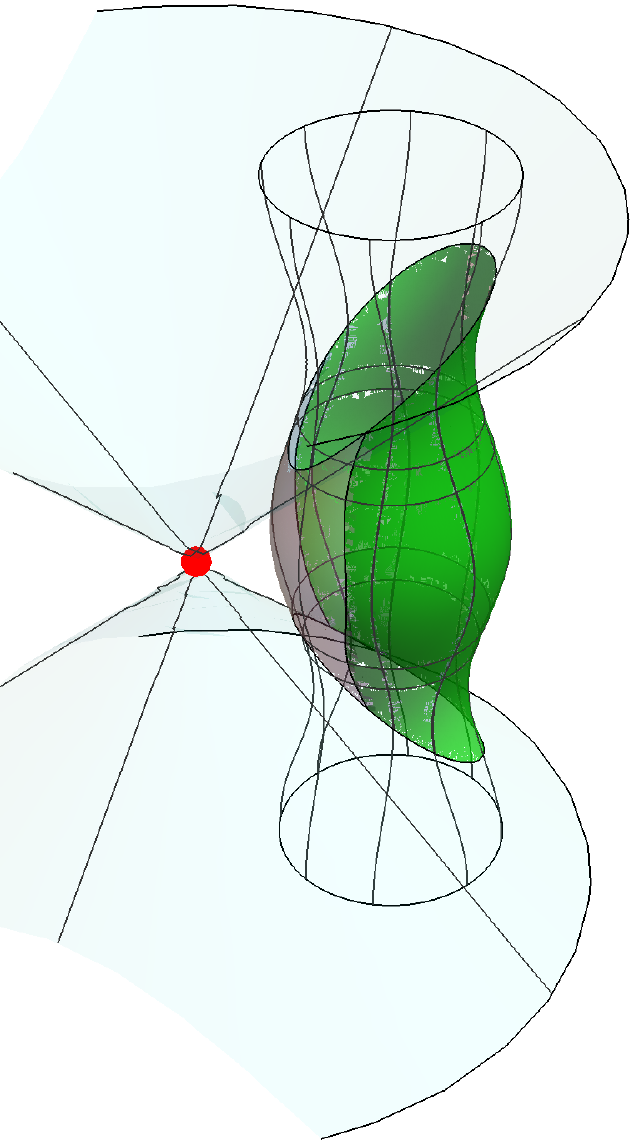}
\end{align*}
\caption{The diagrams give examples of regions considered in Theorem \ref{thm.intro_obs_main}.
The first image shows the boundary $\partial \mc{U}$ (in orange) of a generalized timelike cylinder $\mc{U}$.
The second image shows one case in which the point $( t_0, x_0 )$ (in red) lies within $\mc{U}$; here, $\Gamma_\dagger$ (drawn in green) is the full intersection of $\partial \mc{U}$ with the exterior $\mc{D}_\ast$ of the null cone about $( t_0, x_0 )$.
The last two images demonstrate another case with $( t_0, x_0 )$ outside of $\bar{\mc{U}}$: in the third image, the shaded piece (in light purple) is the full intersection of $\partial \mc{U}$ with the null cone exterior $\mc{D}_\ast$, while in the fourth image, the highlighted piece (in green) is the strictly smaller region $\Gamma_\dagger$, where $\mc{N} f_\ast > 0$.
All of the images were generated using \emph{Mathematica}.}
\label{fig.intro_obs_main}
\end{figure}

Compared to other results using multiplier and Carleman methods, Theorem \ref{thm.intro_obs_main} yields a strict improvement of the region of observation to time-independent subsets of $\partial \mc{U}$.
However, the regions $\Gamma_\dagger$ are in general weaker than the geometric control condition obtained from microlocal methods (under additional assumptions of time-independence or time-analyticity).
Thus, as far as the author is aware, the regions considered in Theorem \ref{thm.intro_obs_main} represent the best-known for non-analytic wave equations, as well as for wave equations on time-dependent domains.

\begin{remark} \label{rmk.intro_obs_main_rel}
In addition, for general time-dependent $\mc{U}$, the observation region $\Gamma_\dagger$ in \eqref{eq.intro_obs_main_region} can be interpreted as a relativistic modification of the standard condition $( x - x_0 ) \cdot \nu > 0$.
To explore this further, let us consider a point $( \tau, y ) \in \Gamma_\dagger$.
We can then perform a Lorentz boost centered about the event $( t_0, x_0 )$ to obtain a new inertial coordinate system $( t', x' )$ on $\R^{ 1 + n }$ such that
\[
t' ( \tau, y ) = t' ( t_0, x_0 ) = t ( t_0, x_0 ) = t_0 \text{.}
\]
In other words, $( \tau, y )$ and $( t_0, x_0 )$ are simultaneous with respect to $( t', x' )$.

Since the function $f_\ast$ is invariant with respect to such Lorentz boosts, then
\[
\mc{N} f_\ast |_{ ( \tau, y ) } = \frac{1}{4} \mc{N} [ | x' - x_0 |^2 ] |_{ ( \tau, y ) } = \frac{1}{2} [ ( x' - x_0 ) \cdot \nu' ] |_{ ( \tau, y ) } \text{,}
\]
where $\nu'$ denotes the spatial component of $\mc{N}$ in the boosted $( t', x' )$-coordinates.
In other words, the condition $\mc{N} f_\ast |_{ ( \tau, y ) } > 0$ is simply the standard condition $( x - x_0 ) \cdot \nu > 0$, but now with respect to a boosted inertial coordinate system in which $( t_0, x_0 )$ and $( \tau, y )$ are simultaneous.
\end{remark}

We also note the discrepancy between $\Gamma_\dagger$, which has the relativistic interpretation from Remark \ref{rmk.intro_obs_main_rel}, and the actual observation region $\mc{Y}_\dagger$ from \eqref{eq.intro_obs_main}.
In particular, $\mc{Y}_\dagger$ is strictly larger than $\Gamma_\dagger$, though by an arbitrarily small amount.
This difference arises from two technical points:
\begin{itemize}
\item First, in the main Carleman estimates, we work not with $f_\ast$, but rather with a perturbation $\bar{f}$; see Theorem \ref{thm.carleman_est} and \eqref{eq.carleman_weight} for details.\footnote{Note, however, that $\bar{f}$ can be made arbitrarily close to $f_\ast$ by adjusting the parameter $\varepsilon$ within.}
The reason is that the level sets of $f_\ast$ are only zero pseudoconvex with respect to the wave operator, and hence a Carleman estimate obtained using $f_\ast$ fails to control the full $H^1$-norm of the solution.\footnote{Earlier linear and nonlinear Carleman estimates obtained from $f_\ast$ were derived in \cite{alex_shao:uc_global, alex_shao:uc_nlwf}.}

\item Furthermore, for the ``interior" observability estimates, where $( t_0, x_0 ) \in \mc{U}$, the result is obtained by adding together two Carleman estimates centered about two nearby points.
This is needed since the weights in the Carleman estimate degenerate at the center point, hence the $H^1$-norm is not controlled near there.
For details, see Theorem \ref{thm.obs_int}.
\end{itemize}

\begin{remark}
In fact, when $\mc{U}$ is time-independent and $( t_0, x_0 ) \not\in \bar{\mc{U}}$ (that is, the situation treated by existing Carleman estimate methods in the literature), the observation region $\mc{Y}_\dagger$ in \eqref{eq.intro_obs_main} can be further improved to $\Gamma_\dagger$, which by \eqref{eq.intro_obs_prelim_cylinder} now has the form
\[
\Gamma_\dagger := \{ ( \tau, y ) \in \partial \mc{U} \mid f_\ast |_{ ( \tau, y ) } > 0 \text{, } ( y - x_0 ) \cdot \nu > 0 \} \text{,} \qquad \mc{N} := ( \nu^t, \nu ) \text{.}
\]
This is precisely the portion of the standard observation region \eqref{eq.intro_obs_region} that lies outside the null cone centered at $( t_0, x_0 )$.
See Theorem \ref{thm.obs_ext} and Corollary \ref{thm.obs_static} for details and precise computations.

On the other hand, when $( t_0, x_0 ) \in \mc{U}$, one does not expect that $\mc{Y}_\dagger$ can be replaced with $\Gamma_\dagger$, since in this case, $\Gamma_\dagger$ violates the geometric control condition.
\end{remark}

\begin{remark}
In the case $n = 1$, Theorem \ref{thm.intro_obs_main} recovers all the results in the existing literature, including the articles \cite{cui_jiang_wang:control_wave_fec, sengou:obs_control_wave, sengou:obs_control_wave2, sun_li_lu:control_wave_moving, wang_he_li:control_wave_noncyl} mentioned previously, up to the optimal observation time.\footnote{However, Theorem \ref{thm.intro_obs_main} does not treat the case in which $\tau_+ - \tau_-$ is exactly equal to the optimal time.}
In addition, Theorem \ref{thm.intro_obs_main} further extends previous results to general cases in which $\partial \mc{U}$ is given by two timelike curves, and to wave equations with arbitrary lower-order coefficients.
\end{remark}

Finally, we note that Theorem \ref{thm.intro_obs_main} again implies exact controllability:

\begin{corollary} \label{thm.intro_control_main}
Assume the definitions and hyptheses of Theorem \ref{thm.intro_obs_main}.
Then, the linear wave \eqref{eq.intro_wave} is exactly controllable, with initial and final data on $\mc{U} \cap \{ t = \tau_- \}$ and $\mc{U} \cap \{ t = \tau_+ \}$, respectively, and with Dirichlet boundary control supported in $\mc{Y}_\dagger$ (all in the appropriate spaces).
\end{corollary}

\begin{remark}
The initial and final data for the wave equation need not be placed on level sets of $t$.
In fact, with regards to Corollary \ref{thm.intro_control_main}, any pair of spacelike cross-sections of $\mc{U}$ lying in the past and future of the null cone about $( t_0, x_0 )$ would suffice; see the discussion in Section \ref{sec.app_control}.
This is an immediate consequence of (geometric) energy estimates for linear wave equations.
\end{remark}

\subsubsection{Carleman Estimates}

As mentioned before, our main observability results, in particular Theorem \ref{thm.intro_obs_main}, are obtained using a new Carleman estimate (derived in Section \ref{sec.carleman}), which is the foundation of this article.
Roughly speaking, this Carleman estimate has the following features:
\begin{itemize}
\item The estimate is supported in the exterior of a null cone about a point $( t_0, x_0 ) \in \R^{1+n}$.
In fact, this property implies, through the standard arguments, that the resulting observability inequalities, such as Theorem \ref{thm.intro_obs_main}, inherit this exterior support property.

\item The estimate is established through Lorentzian geometric techniques.
By taking advantage of methods that exploit the full \emph{spacetime} geometry, we obtain results that are equally applicable to domains with either fixed or moving boundaries.\footnote{In fact, most features of the Lorentzian approach can already be seen in the proof of the preliminary Theorem \ref{thm.intro_obs_prelim}.}
\end{itemize}
The precise Carleman estimate of this paper is stated in Theorem \ref{thm.carleman_est} and proved throughout Section \ref{sec.carleman}.
Below, we give a brief summary of some of the ideas behind its derivation.

The Lorentzian geometric perspective toward Carleman estimates for wave equations had its roots in other applications.
For example, in mathematical relativity, such estimates were used to establish rigidity properties for stationary black hole spacetimes; see, e.g., \cite{alex_io_kl:unique_bh, alex_io_kl:rigid_bh, chru_cost:unique_bh, io_kl:unique_bh, io_kl:unique_ip}.

Regarding the main Carleman estimates of this article, its ideas arose from the study of unique continuation properties of waves ``from infinity".
In \cite{alex_schl_shao:uc_inf}, the following result was shown using geometric Carleman estimates: if a linear wave \emph{vanished to infinite order} at a ``large enough" portion of both future and past ``null infinity" (formally, where the radiation field is manifested), then the wave must vanish locally near that portion of infinity.\footnote{See \cite{alex_schl_shao:uc_inf} for precise statements.}
Furthermore, this infinite-order vanishing is necessary, as one can construct solutions of \eqref{eq.intro_wave_free} that decay like $r^{-k}$ for any $k \in \N$.

Next, \cite{alex_shao:uc_global} showed that in $\R^{1+n}$, the infinite-order vanishing assumption can be removed provided the wave is \emph{globally} regular.\footnote{In particular, the counterexamples all blow up at some point in space.}
Technically speaking, the infinite-order vanishing condition arose from the use of cutoff functions in the standard argument proving unique continuation from Carleman estimates.
This was overcome in \cite{alex_shao:uc_global} using new \emph{global} Carleman estimates for which \emph{the associated weight itself vanishes on the null cone about the origin}.
In particular, this vanishing weight removed the need to use a cutoff function in the unique continuation argument.
In fact, a similar weight is also responsible for the exterior support property in the estimates of this paper.

Despite the global nature of the Carleman estimate in \cite{alex_shao:uc_global}, it is equally applicable to \emph{finite} spacetime regions $\mc{U} \subseteq \R^{1+n}$.
While this yields additional boundary terms on $\partial \mc{U}$, the vanishing Carleman weight again ensures that no boundary terms are found on the null cone itself.
A nonlinear variant of this estimate was applied in \cite{alex_shao:uc_nlwf} to study singularity formation for focusing subconformal nonlinear wave equations.
The key step with the Carleman estimate was to control the wave in the interior of a past time cone, based at a singular point, by its values on the cone itself.

Although this is closely related to boundary observability estimates, which also aim to control a wave in the interior of a cylindrical region $\mc{U}$ by its values on the boundary, the results of \cite{alex_shao:uc_global, alex_shao:uc_nlwf} unfortunately fail to imply such observability.
The main reason is that these Carleman estimates make heavy use of the hyperbolic function $f_\ast$ from \eqref{eq.intro_obs_prelim_region}, whose level sets fail to be strictly pseudoconvex.\footnote{Indeed, these level sets are only zero, or degenerately, pseudoconvex.}
As a result, these estimates only control the $L^2$-norm in the bulk, rather than the full $H^1$-norm that is required for observability and exact controllability.

In the existing literature on observability via Carleman estimates, this lack of pseudoconvexity was overcome by altering $f_\ast$.
More specifically, one instead considers the level sets of
\[
f_c = | x - x_0 |^2 - c^2 ( t - t_0 )^2 \text{,} \qquad 0 < c < 1 \text{,}
\]
which are now strictly pseudoconvex.
This property indeed allows for the recovery of the $H^1$-norm in the bulk.
However, this $f_c$ is now poorly adapted to the geometry, in particular the characteristics, of the wave operator.
In particular, one cannot recover, using this $f_c$, the aforementioned exterior support property of the Carleman estimates from \cite{alex_shao:uc_global, alex_shao:uc_nlwf}.

In this paper, we instead \emph{consider different alterations $\bar{f}$ of $f_\ast$}, with the following properties:
\begin{itemize}
\item $\bar{f}$ is well-adapted to the characteristics of the wave operator.
In particular, like for $f_\ast$, the level set $\{ \bar{f} = 0 \}$ corresponds to the null cone about $( t_0, x_0 )$.

\item The level sets of $\bar{f}$ are strictly pseudoconvex in the exterior of the null cone about $( t_0, x_0 )$.
\end{itemize}
In particular, by making use of the above $\bar{f}$, one can both recover the $H^1$-norm in the bulk and preserve the exterior support property in our Carleman estimates.

While these properties of $\bar{f}$ can be viewed as the consequences of extensive computations, they can also explained through the perspective of conformal geometry.
In this respect, the main idea in this paper is to consider instead a ``warped" metric $\bar{g}$ that slightly inflates the volumes of spatial spheres about the origin; see Definition \ref{def.wp_met}.
Moreover, this warped metric $\bar{g}$ can be shown to be conformally isometric to the usual Minkowski metric on $\R^{1+n}$; see Proposition \ref{thm.met_conf}.

A fairly direct computation (see Proposition \ref{thm.pseudoconvex_wp}) shows that the positive level sets of the function $f_\ast$, with $( t_0, x_0 ) = ( 0, 0 )$, are strictly pseudoconvex \emph{with respect to $\bar{g}$}.
Then:
\begin{itemize}
\item Since pseudoconvexity (with respect to geometric wave operators) is a conformally invariant property, the pullback $\bar{f}$ of $f_\ast$ through the aforementioned conformal isometry is pseudoconvex as well; this allows for the bulk $H^1$-bounds in our Carleman estimates.

\item Since the characteristics of geometric wave operators are also conformally invariant, this $\bar{f}$ is also as well-adapted to the characteristics of the (flat) wave operator as $f_\ast$ is.
As a result of this, we can still derive the external support property using this $\bar{f}$ in the place of $f_\ast$.
\end{itemize}
The above points provide the main ideas behind why the Carleman estimates in this paper work.

\begin{remark}
In fact, by conformal invariance, we can derive Carleman and observability estimates for wave equations with respect to other geometries that are conformally related to the Minkowski metric.
However, we will not pursue details of this in the present paper.
\end{remark}

Now, regarding the proof of the main Carleman estimate, the strategy proceeds as follows:
\begin{enumerate}
\item Prove a Carleman estimate for wave equations with respect to the ``warped" metric $\bar{g}$.

\item Use the above-mentioned conformal isometry and invariance in order to extract a corresponding Carleman estimate for the usual (flat) wave operator on $\R^{1+n}$.
\end{enumerate}
Steps (1) and (2) are proved in Sections \ref{sec.carleman_proof} and \ref{sec.carleman_warped}, respectively.
We remark that although this process still involves extensive computations, the benefit is that the main steps are relatively simple.

\begin{remark}
As previously mentioned, the existing literature on Carleman-based observability \cite{baudo_debuh_erv:carleman_wave, fu_yong_zhang:ctrl_semilinear, lasie_trigg_zhang:wave_global_uc, zhang:obs_wave_pot, zhang:obs_wave_lower} only considered settings with $( t_0, x_0 ) \not\in \bar{\mc{U}}$ (in the context of time-independent $\mc{U}$).
At a technical level, this was because one required uniform positivity of $| x - x_0 |$ in order to obtain a positive zero-order bulk term in the corresponding Carleman estimate.
An interesting feature of the exterior support property in our main Carleman estimate is that one no longer requires uniform positivity for $| x - x_0 |$ for a positive zero-order bulk.
As a result, our main Carleman estimate also holds when $( t_0, x_0 ) \in \mc{U}$ and hence is applicable toward ``interior" observability estimates.
\end{remark}

Finally, we remark that this warping of the metric, in particular the difference between the Minkowski and warped metrics, is primarily responsible for the discrepancy between the $\Gamma_\dagger$ and the actual observed region $\mc{Y}_\dagger$ in Theorem \ref{thm.intro_obs_main}.
Moreover, since this warping can be made arbitrarily small, the discrepancy between $\Gamma_\dagger$ and $\mc{Y}_\dagger$ can be made small as well.

\subsubsection{Other Directions}

A larger aim beyond the present paper is to study to similar controllability properties for geometric wave equations, in particular to settings with time-dependent geometry.
As far as the author is aware, there is no literature directly addressing controllability for waves on general Lorentzian manifolds.
While we do not discuss this here, our broader intention is to develop tools that can be robustly applied to geometric settings in future works.

Next, since the methods in this article do not require time-analyticity for its coefficients, they could also be used for treating nonlinear wave equations.
As already mentioned, earlier versions of the main Carleman estimate have been used in \cite{alex_shao:uc_nlwf}, in the context of studying singularity formation.

Finally, Carleman estimates have been widely applied toward solving inverse problems for wave equations, which themselves are closely connected to problems in tomography and seismology (see, e.g., \cite{stef_uhl:thermo_tomo, symes:seis_refl}).
In this respect, Carleman estimates have been applied both directly \cite{bukhg_klib:ip_carleman, iman_yama:global_det_hyp} or as part of an intermediate unique continuation argument \cite{beli_kury:ip_bc, kian:stab_timedep, isak:ip_pde, kacha_lass_kury:inv_bsp, kian:unique_timedep, kian_oks:ip_timedep}.
Another application of the estimates in this article (and their future geometric generalizations) is toward inverse problems for wave equations in settings with time-dependent domains and moving boundaries.

\subsection{Outline of the Paper}

The remainder of this paper will be organized as follows:
\begin{itemize}
\item In Section \ref{sec.prelim}, we define the objects and notations that we will use throughout the paper.
In particular, we describe the domains---the generalized timelike cylinders---that we will consider in our main results.
We also give a precise statement of the preliminary observability estimate given in Theorem \ref{thm.intro_obs_prelim}, and we give a short, simple proof of this estimate.

\item In Section \ref{sec.carleman}, we state and prove our main Carleman estimate, Theorem \ref{thm.carleman_est}.
In particular, within the proof, we introduce the ``warped" metric and establish its basic properties.

\item Section \ref{sec.obs} is dedicated to the the precise statements of the main observability inequalities of this paper, as well as to the proofs of these inequalities.

\item In Section \ref{sec.app}, we discuss some consequences of the main observability results.
In particular, Section \ref{sec.app_1d} explores the special case of one spatial dimension, while Section \ref{sec.app_control} deals with the exact controllability results that follow from our main inequalities.

\end{itemize}

\subsection{Acknowledgments}

The author wishes to thank Spyros Alexakis for numerous discussions, in particular for previous work related to an earlier form of the Carleman estimates in this paper.
The author also extends thanks to Lauri Oksanen, Matti Lassas, and Matthieu L\'eautaud for some helpful discussions, as well as to two anonymous referees for their comments and suggestions.
In addition, this work was partly supported by EPSRC grant EP/R011982/1.

\section{Preliminaries} \label{sec.prelim}

In this section, we set the notations and definitions that we will use throughout the paper.
Using these definitions, we will then give a precise formulation and a proof of the preliminary observability estimate for free waves that was stated in Theorem \ref{thm.intro_obs_prelim}.

\begin{definition} \label{def.ineq}
Throughout, we will adopt the following standard notations for inequalities:
\begin{itemize}
\item We write $A \lesssim B$ to mean that there is some universal constant $C > 0$ such that $A \leq C B$.
Moreover, $A \lesssim_{ a_1, \dots, a_m } B$ means that the above constant $C$ depends on $a_1, \dots, a_m$.

\item We write $A \ll B$ to mean that there is some sufficiently small and universal constant $c > 0$ such that $A \leq c B$.
Moreover, $A \ll_{ a_1, \dots, a_m } B$ means that this $c$ depends on $a_1, \dots, a_m$.
\end{itemize}
\end{definition}

\subsection{The Geometric Setting} \label{sec.carleman_geom}

Since the methods we use are Lorentz geometric in nature, we begin our discussions by setting the basic notations relating to Minkowski geometry---the background naturally associated to the wave operator.
In addition, we give a precise definition of \emph{generalized timelike cylinders}, the time-dependent domains on which our wave equations are set.

\subsubsection{Minkowski Geometry}

We begin by recalling the Minkowski metric on $\R^{1+n}$:

\begin{definition} \label{def.mink}
We define the following on $\R^{1+n}$:
\begin{itemize}
\item Let $t$ and $x = ( x^1, \dots, x^n )$ denote the usual Cartesian coordinates on $\R^{1+n}$, mapping to the first and the remaining $n$ components of $\R^{1+n}$, respectively.

\item Let $g$ denote the \emph{Minkowski metric} on $\R^{1+n}$:
\begin{equation}
\label{eq.mink_met} g := - dt^2 + d ( x^1 )^2 + \dots + d ( x^n )^2 \text{.}
\end{equation}
\end{itemize}
In particular, we refer to the manifold $( \R^{1+n}, g )$ as \emph{Minkowski spacetime}.
\end{definition}

\begin{definition} \label{def.mink_fct}
We also define the following standard functions on $\R^{1+n}$:
\begin{itemize}
\item Let $r := | x |$ denote the radial function.

\item Let $u$ and $v$ denote the null coordinates:
\begin{equation}
\label{eq.uv} u := \frac{1}{2} ( t - r ) \text{,} \qquad v := \frac{1}{2} ( t + r ) \text{.}
\end{equation}

\item Let $f$ denote the hyperbolic function,
\begin{equation}
\label{eq.f} f := - u v = \frac{1}{4} ( r^2 - t^2 ) \text{.}
\end{equation}
\end{itemize}
\end{definition}

\begin{remark}
The nonzero level sets of $f$ are one-sheeted ($f > 0$) and two-sheeted ($f < 0$) hyperboloids on $\R^{1+n}$, while $\{ f = 0 \}$ is the null cone with vertex at the origin.
\end{remark}

For our purposes, it would also be convenient to consider spacetime translations of the functions from Definitions \ref{def.mink_fct}.
Thus, we will also make use of the following notations:

\begin{definition} \label{def.mink_shift}
Fix a point $P \in \R^{1+n}$.
\begin{itemize}
\item We define the shifted time and spatial coordinates by
\begin{equation}
\label{eq.tx_shift} t_P := t - t (P) \text{,} \qquad x_P := x - x (P) \text{.}
\end{equation}

\item We define the shifted radial coordinate, null coordinates, and hyperbolic function by
\begin{equation}
\label{eq.ruvf_shift} r_P := | x_P | \text{,} \qquad u_P := \frac{1}{2} ( t_P - r_P ) \text{,} \qquad v_P := \frac{1}{2} ( t_P + r_P ) \text{,} \qquad f_P := - u_P v_P \text{.}
\end{equation}
\end{itemize}
\end{definition}

\begin{definition} \label{def.mink_coord}
Recall the following standard coordinate systems on $\R^{1+n}$:
\begin{itemize}
\item On $\{ r \neq 0 \}$, we recall the standard polar coordinates $( t, r, \omega )$ and null coordinates $( u, v, \omega )$, with $\omega$ being the angular coordinate taking values in $\Sph^{n - 1}$.
In addition, we let $\partial_t, \partial_r, \partial_u, \partial_v$ denote the coordinate vector fields with respect to these coordinate systems.

\item We can also shift the above coordinate systems by any $P \in \R^{1+n}$.
More specifically, on $\{ r_P \neq 0 \}$, we have the shifted polar coordinates $( t_P, r_P, \omega )$ and null coordinates $( u_P, v_P, \omega )$, with $\omega \in \Sph^{n-1}$ now denoting the angular value about the axis $\{ r_P = 0 \}$.
Again, we let $\partial_{ t_P }$, $\partial_{ r_P }$, $\partial_{ u_P }$, $\partial_{ v_P }$ denote the coordinate vector fields with respect to these coordinate systems.
\end{itemize}
\end{definition}

Recall that the Minkowski metric can be written in terms of polar and null coordinates as
\begin{equation}
\label{eq.mink_met_alt} g = - d t_P^2 + d r_P^2 + r_P^2 \mathring{\gamma} = - 4 du_P dv_P + r_P^2 \mathring{\gamma} \text{,} \qquad P \in \R^{1+n} \text{,}
\end{equation}
where $\mathring{\gamma}$ is the unit round metric on the level spheres of $( u_P, v_P )$.

\begin{definition} \label{def.mink_index}
For tensor fields on $\R^{1+n}$, we adopt the following index conventions:
\begin{itemize}
\item We use lowercase Greek letters (ranging from $0$ to $n$) for spacetime components in $\R^{1+n}$.

\item We use lowercase Latin letters (ranging from $1$ to $n - 1$) to denote angular components, corresponding to the $\omega \in \Sph^{n-1}$ in any of the coordinate systems in Definition \ref{def.mink_coord}.

\item As is standard, we will adopt the Einstein summation notation---repeated indices in both subscript and superscript refer to sums over all possible index values.

\item Unless stated otherwise, indices will be raised and lowered using $g$ and its metric dual.
\end{itemize}
\end{definition}

\begin{definition}
We adopt the following notations for covariant derivatives:
\begin{itemize}
\item Let $\nabla$ denote the Levi-Civita connection with respect to $g$.
In particular, for an appropriately defined function $A$, we write $\nabla A$ to mean the spacetime differential of $A$.

\item Let $\nabla^\sharp$ denote the gradient with respect to $g$; more specifically, for an appropriate scalar function $h$, we let $\nabla^\sharp h$ be the $g$-metric dual of $\nabla h$.\footnote{In other words, $\nabla^\sharp h$ is the vector field obtained by raising the index of $\nabla h$.}

\item Let $\Box$ denote the wave operator with respect to $g$:
\begin{equation}
\label{eq.mink_box} \Box := g^{ \alpha \beta } \nabla_{ \alpha \beta } \text{.}
\end{equation}

\item Let $\nasla$ denote the angular connections on the level spheres $\sigma$ of $(u, v)$, i.e., the Levi-Civita connections associated with the metrics $\slashed{g}$ induced by $g$ on the spheres $\sigma$.
Similarly, we use $\nasla{}^P$ to denote the angular connections on the level spheres of $( u_P, v_P )$.
\end{itemize}
\end{definition}

\begin{definition} \label{def.D}
Let $P \in \R^{1+n}$.
We define the domains $\mc{D}_P$ and $\mc{D}$ by
\begin{equation}
\label{eq.D} \mc{D}_P := \{ f_P > 0 \} \text{,} \qquad \mc{D} := \{ f > 0 \} \text{,}
\end{equation}
i.e., the exteriors of the null cones about $P$ and the origin, respectively.
\end{definition}

\begin{definition} \label{def.future_past}
Let $P \in \R^{1+n}$ and $A \subseteq \R^{1+n}$.
\begin{itemize}
\item Let $I^+ ( P )$ and $I^- ( P )$ denote the \emph{chronological future} and \emph{past}, respectively, of $P$:
\begin{equation}
\label{eq.I_point} I^+ ( P ) = \{ f_P < 0 \} \cap \{ t_P > 0 \} \text{,} \qquad I^- ( P ) = \{ f_P < 0 \} \cap \{ t_P < 0 \} \text{.}
\end{equation}

\item Let $I^+ ( A )$ and $I^- ( A )$ denote the \emph{chronological future} and \emph{past} of $A$:
\begin{equation}
\label{eq.I_subset} I^+ ( A ) = \bigcup_{ P \in A } I^+ ( P ) \text{,} \qquad I^- ( A ) = \bigcup_{ P \in A } I^- ( P ) \text{.}
\end{equation}
\end{itemize}
\end{definition}

\subsubsection{Geometric Timelike Cylinders}

Next, we define precisely the domains (in general, time-dependent and with moving boundaries) on which we will consider our wave equations.

\begin{definition} \label{def.gtc}
$\mc{U}$ is called a \emph{generalized timelike cylinder} (abbreviated \emph{GTC}) in $\R^{1+n}$ iff:
\begin{itemize}
\item $\mc{U} \subseteq \R^{1+n}$, and $\partial \mc{U}$ is a smooth timelike hypersurface of $\R^{1+n}$.

\item For any $\tau \in \R$, the set $\Omega_\tau := \{ y \in \R^n \mid ( \tau, y ) \in \mc{U} \}$ is nonempty, bounded, and open in $\R^n$.

\item There is a smooth future-directed timelike vector field $Z$ on $\bar{\mc{U}}$, with $Z |_{ \partial \mc{U} }$ tangent to $\partial \mc{U}$.
\end{itemize}
In the above, we refer to $Z$ as a \emph{generator} of $\mc{U}$.
\end{definition}

\begin{remark}
Note that if $\mc{U}$ is a GTC, then we have the following:
\begin{itemize}
\item With $\Omega_\tau$ as in Definition \ref{def.gtc}, we can write
\[
\mc{U} = \bigcup_{ \tau \in \R } \{ \tau \} \times \Omega_\tau \text{.}
\]

\item Since each $\Omega_\tau$ is open, then $\mc{U}$ is open in $\R^{1+n}$.

\item Any two $\bar{\Omega}_{ \tau_0 }$ and $\bar{\Omega}_{ \tau_1 }$ are diffeomorphic, as manifolds with boundaries.
For instance, one can identify $\bar{\Omega}_{ \tau_0 }$ with $\bar{\Omega}_{ \tau_1 }$ by flowing along the integral curves of any generator $Z$ of $\mc{U}$.
\end{itemize}
\end{remark}

\begin{remark}
The regularity of $\partial \mc{U}$ and the generator $Z$ in Definition \ref{def.gtc} could be lowered considerably without altering the main results of this paper.
However, since regularity is not a main focus here, we avoid exploring optimal regularities to avoid technical baggage.
\end{remark}

\begin{example} \label{ex.static_cyl}
Suppose $\Omega$ is bounded and open subset of $\R^n$, with a smooth boundary $\partial \Omega$.
Then, the static cylinder $\R \times \Omega$ is a GTC, and $\partial_t$ is a generator of $\R \times \Omega$.
\end{example}

Suppose $\mc{U}$ is a GTC.
Given a coordinate system $( y^1, \dots, y^n )$ on $\Omega_0$ (see Definition \ref{def.gtc}), we can lift the $y^k$'s to $\mc{U}$ by transporting them along the integral curves of a generator $Z$.
With respect to the coordinates $( t, y^1, \dots, y^n )$, we have that $\mc{U} \simeq \R \times \Omega_0$.
Thus, we can characterize GTCs as domains that can be reparametrized as static cylinders (in which the static direction is timelike).

Next, we describe geometrically the hypersurfaces on which one can impose Cauchy (i.e., initial or final) data for a linear wave equation on a GTC:

\begin{definition} \label{def.gtc_cs}
Let $\mc{U}$ be a GTC.
A subset $\mc{V} \subseteq \mc{U}$ is called a \emph{cross-section} of $\mc{U}$ iff:
\begin{itemize}
\item $\mc{V}$ is a smooth spacelike hypersurface of $\mc{U}$, and $\partial \mc{V} \subseteq \partial \mc{U}$.\footnote{Recall that $\mc{V}$ is called \emph{spacelike} iff the metric on $\mc{V}$ induced by $g$ is positive-definite.}

\item If $Z$ is a generator of $\mc{U}$, then every integral curve of $Z$ hits $\bar{\mc{V}}$ exactly once.
\end{itemize}
\end{definition}

Note that any two cross-sections $\mc{V}_1$ and $\mc{V}_2$ of a GTC must be diffeomorphic, since $\mc{V}_1$ can be identified with $\mc{V}_2$ by flowing along the integral curves of any generator.
Next, we develop some basic terminology for comparing two cross-sections of a GTC.

\begin{definition} \label{def.gtc_cs_comp}
Let $\mc{U}$ be a GTC, and let $\mc{V}_-, \mc{V}_+$ be two cross-sections of $\mc{U}$.
\begin{itemize}
\item We write $\mc{V}_- < \mc{V}_+$ iff $\mc{V}_+ \subseteq I^+ ( \mc{V}_- )$ (or equivalently, $\mc{V}_- \subseteq I^- ( \mc{V}_+ )$).

\item When $\mc{V}_- < \mc{V}_+$, we let $\mc{U} ( \mc{V}_-, \mc{V}_+ )$ denote the region of $\mc{U}$ between $\mc{V}_-$ and $\mc{V}_+$,
\begin{equation}
\label{eq.gtc_slab} \mc{U} ( \mc{V}_-, \mc{V}_+ ) := \mc{U} \cap I^+ ( \mc{V}_- ) \cap I^- ( \mc{V}_+ ) \text{,}
\end{equation}
and we let $\partial \mc{U} ( \mc{V}_-, \mc{V}_+ )$ denote the region of $\partial \mc{U}$ between $\mc{V}_-$ and $\mc{V}_+$,
\begin{equation}
\label{eq.gtc_bdry_slab} \partial \mc{U} ( \mc{V}_-, \mc{V}_+ ) := \partial \mc{U} \cap I^+ ( \mc{V}_- ) \cap I^- ( \mc{V}_+ ) \text{.}
\end{equation}
\end{itemize}
\end{definition}

The most basic examples of cross-sections are constructed from the level sets of $t$:

\begin{definition} \label{def.gtc_t}
Let $\mc{U}$ be a GTC.
\begin{itemize}
\item Given $\tau \in \R$, we define the cross-section
\begin{equation}
\label{eq.gtc_cs_t} \mc{U}_\tau := \mc{U} \cap \{ t = \tau \} \text{.}
\end{equation}

\item Moreover, given $\tau_-, \tau_+ \in \R$ with $\tau_- < \tau_+$, we define
\begin{align}
\label{eq.gtc_slab_t} \mc{U}_{ \tau_-, \tau_+ } &:= \mc{U} ( \mc{U}_{ \tau_- }, \mc{U}_{ \tau_+ } ) = \mc{U} \cap \{ \tau_- < t < \tau_+ \} \text{,} \\
\notag \partial \mc{U}_{ \tau_-, \tau_+ } &:= \partial \mc{U} ( \mc{U}_{ \tau_- }, \mc{U}_{ \tau_+ } ) = \partial \mc{U} \cap \{ \tau_- < t < \tau_+ \} \text{.}
\end{align}
\end{itemize}
\end{definition}

\subsection{Linear Wave Equations}

The next task is to give a precise description of the linear wave equations that we will consider throughout this paper:

\begin{problem} \label{prb.linear_wave}
Let $\mc{U}$ be a GTC in $\R^{1+n}$, and fix
\begin{equation}
\label{eq.XV} \mc{X} \in C^\infty ( \bar{\mc{U}}; \R^{1+n} ) \text{,} \qquad V \in C^\infty ( \bar{\mc{U}} ) \text{.}
\end{equation}
Solve the following linear wave equation on $\bar{\mc{U}}$:
\begin{equation}
\label{eq.linear_wave} \mc{P} \phi |_{ \mc{U} } := ( \Box \phi + \nabla_{ \mc{X} } \phi + V \phi ) |_{ \mc{U} } = 0 \text{.}
\end{equation}
\end{problem}

\begin{remark}
The assumption \eqref{eq.XV} that the coefficients $\mc{X}$ and $V$ in $\mc{P}$ are smooth is purely for convenience.
All the results in the paper still hold, as long as $\mc{X}$ and $V$ are sufficiently regular.
For instance, our main observability results remain true when $\mc{X}$ and $V$ are merely uniformly bounded.\footnote{The assumptions for $V$ could be further weakened (for instance, to $L^p$-integrability) with more refinements to the proofs in this article. However, we avoid doing this here in order to avoid additional technical details.}
\end{remark}

\subsubsection{Energy Estimates}

Next, we establish standard energy estimates, one global and one local, for solutions of Problem \ref{prb.linear_wave} with vanishing Dirichlet data.
Before stating and proving the energy estimates, however, we first require some notations regarding the sizes of tensorial objects:

\begin{definition} \label{def.tensor_norm}
Let $\mc{W}$ be an open subset of $\R^{1+n}$.
\begin{itemize}
\item For a scalar $h \in C^\infty ( \mc{W} )$, we define
\begin{equation}
\label{eq.tensor_norm_deriv} | \nabla_{ t, x } h |^2 = ( \partial_t h )^2 + ( \partial_{ x^1 } h )^2 + \dots + ( \partial_{ x^n } h )^2 \text{.}
\end{equation}

\item For a vector field $X \in C^\infty ( \mc{W}; \R^{1+n} )$, we define
\begin{equation}
\label{eq.tensor_norm_vf} | X^{ t, x } |^2 = ( X^t )^2 + ( X^{ x^1 } )^2 + \dots + ( X^{ x^n } )^2 \text{,}
\end{equation}
where $X^t, X^{ x^1 }, \dots, X^{ x^n }$ represent the components of $X$ in Cartesian coordinates.
\end{itemize}
\end{definition}

The following global energy estimate can be viewed as a geometric extension of the modified energy estimates of \cite{duy_zhang_zua:obs_opt, zua:ctrl_wave_1d}.
As was the case in \cite{duy_zhang_zua:obs_opt}, these modified estimates allow us to optimize our dependence on the potential in our main observability estimates.

\begin{proposition} \label{thm.energy_est}
Let $\mc{U}$ be a GTC, and let $\mc{V}_\pm$ be cross-sections of $\mc{U}$ with $\mc{V}_- < \mc{V}_+$.
Moreover, consider the linear wave equation of Problem \ref{prb.linear_wave}, and let
\begin{equation}
\label{eq.energy_est_M} \mc{M}_0 := 1 + \sup_{ \mc{U} ( \mc{V}_-, \mc{V}_+ ) } | V | \text{,} \qquad \mc{M}_1 := 1 + \sup_{ \mc{U} ( \mc{V}_-, \mc{V}_+ ) } | \mc{X}^{ t, x } | \text{,} \qquad \mc{T} := \sup_{ \mc{V}_+ } t - \inf_{ \mc{V}_- } t \text{.}
\end{equation}
Then, there exist constants $C, C' > 0$, depending on $\mc{U}$ and $\mc{V}_\pm$, such that
\begin{align}
\label{eq.energy_est} \int_{ \mc{V}_+ } ( | \nabla_{ t, x } \phi |^2 + \mc{M}_0 \cdot \phi^2 ) &\leq C e^{ C' ( \mc{M}_0^\frac{1}{2} + \mc{M}_1 ) \mc{T} } \int_{ \mc{V}_- } ( | \nabla_{ t, x } \phi |^2 + \mc{M}_0 \cdot \phi^2 ) \text{,} \\
\notag \int_{ \mc{V}_- } ( | \nabla_{ t, x } \phi |^2 + \mc{M}_0 \cdot \phi^2 ) &\leq C e^{ C' ( \mc{M}_0^\frac{1}{2} + \mc{M}_1 ) \mc{T} } \int_{ \mc{V}_+ } ( | \nabla_{ t, x } \phi |^2 + \mc{M}_0 \cdot \phi^2 ) \text{.}
\end{align}
for any solution $\phi \in C^2 ( \mc{U} ) \cap C^1 ( \bar{\mc{U}} )$ of \eqref{eq.linear_wave} that also satisfies $\phi |_{ \partial \mc{U} ( \mc{V}_-, \mc{V}_+ ) } = 0$.
\end{proposition}

\begin{proof}[Proof sketch.]
First, we fix some $\mf{t} \in C^\infty ( \bar{\mc{U}} )$ satisfying the following:
\begin{itemize}
\item $\nabla^\sharp \mf{t}$ is everywhere timelike and past-directed (i.e., $\mf{t}$ is a ``time coordinate").

\item There exist $\tau_\pm \in \R$ such that
\begin{equation}
\label{eql.energy_est_time} \mc{V}_\pm = \mc{U} \cap \{ \mf{t} = \tau_\pm \} \text{,} \qquad \tau_+ - \tau_- \simeq \mc{T} \text{.}
\end{equation}
\end{itemize}
Note $\mc{U} \cap \{ \mf{t} = \tau \}$ is a cross-section of $\mc{U}$ for any $\tau \in \R$.
We also define the modified energy
\begin{equation}
\label{eql.energy_est_energy} \mc{E} ( \tau ) = \int_{ \mc{U} \cap \{ \mf{t} = \tau \} } ( | \nabla_{ t, x } \phi |^2 + \mc{M}_0 \cdot \phi^2 ) \text{,} \qquad \tau \in \R \text{,}
\end{equation}
as well as the modified stress-energy tensor:
\begin{equation}
\label{eql.energy_est_emt} Q_{ \alpha \beta } := \nabla_\alpha \phi \nabla_\beta \phi - \frac{1}{2} g_{\alpha \beta} ( g^{\mu \nu} \nabla_\mu \phi \nabla_\nu \phi + \mc{M}_0 \cdot \phi^2 ) \text{.}
\end{equation}

Let $Z$ be a generator of $\mc{U}$; we then apply the usual multiplier method to $\phi$, with multiplier $Z$.
First, an application of the (Lorentzian) divergence theorem on $\mc{U} ( \mc{V}_-, \mc{V}_+ )$ yields
\begin{equation}
\label{eql.energy_est_0} \int_{ \mc{U} ( \mc{V}_-, \mc{V}_+ ) } \nabla^\alpha ( Q_{\alpha \beta} Z^\beta ) = \int_{ \mc{V}_+ } Q ( \mc{N}', Z ) + \int_{ \mc{V}_- } Q ( \mc{N}', Z ) + \int_{ \partial \mc{U} ( \mc{V}_-, \mc{V}_+ ) } Q ( \mc{N}, Z ) \text{,}
\end{equation}
where $\mc{N}$ and $\mc{N}'$ denote the outward and inward pointing unit normals to $\mc{U} ( \mc{V}_-, \mc{V}_+ )$, respectively.
Since $Z$ is future timelike, then standard computations yield that
\begin{equation}
\label{eql.energy_est_01} \mp \int_{ \mc{V}_\pm } Q ( \mc{N}', Z ) \simeq_{ \mc{U}, \mc{V}_\pm } \mc{E} ( \tau_\pm ) \text{.}
\end{equation}
Moreover, since $Z$ and $\mc{N}$ are orthogonal on $\partial \mc{U}$, and since $\phi$ vanishes on $\partial \mc{U} ( \mc{V}_-, \mc{V}_+ )$, then
\begin{equation}
\label{eql.energy_est_02} \int_{ \partial \mc{U} ( \mc{V}_-, \mc{V}_+ ) } Q ( \mc{N}, Z ) = \int_{ \partial \mc{U} ( \mc{V}_-, \mc{V}_+ ) } \mc{N} \phi Z \phi = 0 \text{.}
\end{equation}

For the bulk term in \eqref{eql.energy_est_0}, we expand and apply \eqref{eq.linear_wave} to estimate
\begin{align}
\label{eql.energy_est_03} \left| \int_{ \mc{U} ( \mc{V}_-, \mc{V}_+ ) } \nabla^\alpha ( Q_{\alpha \beta} Z^\beta ) \right| &\leq \int_{ \mc{U} ( \mc{V}_-, \mc{V}_+ ) } ( | \Box \phi | | Z \phi | + \mc{M}_0 | \phi | | Z \phi | + | Q_{\alpha \beta} \nabla^\alpha Z^\beta | ) \\
\notag &\lesssim_{ \mc{U}, \mc{V}_\pm } \int_{ \mc{U} ( \mc{V}_-, \mc{V}_+ ) } ( \mc{M}_1 | \nabla_{ t, x } \phi | | Z \phi | + \mc{M}_0 | \phi | | Z \phi | + | \nabla_{ t, x } \phi |^2 ) \text{.}
\end{align}
By \eqref{eql.energy_est_energy}, \eqref{eql.energy_est_03}, H\"older's inequality, and Fubini's theorem, we conclude that
\begin{equation}
\label{eql.energy_est_1} \left| \int_{ \mc{U} ( \mc{V}_-, \mc{V}_+ ) } \nabla^\alpha ( Q_{\alpha \beta} Z^\beta ) \right| \lesssim_{ \mc{U}, \mc{V}_\pm } ( \mc{M}_0^\frac{1}{2} + \mc{M}_1 ) \int_{ \tau_- }^{ \tau_+ } \mc{E} ( \tau ) d \tau \text{.}
\end{equation}
Combining \eqref{eql.energy_est_0}--\eqref{eql.energy_est_02} and \eqref{eql.energy_est_1}, we obtain the inequalities
\begin{equation}
\label{eql.energy_est_2} \mc{E} ( \tau_\pm ) \lesssim_{ \mc{U}, \mc{V}_\pm } \mc{E} ( \tau_\mp ) + ( \mc{M}_0^\frac{1}{2} + \mc{M}_1 ) \int_{ \tau_- }^{ \tau_+ } \mc{E} ( \tau ) d \tau \text{.}
\end{equation}

By varying $\mc{V}_\pm$ to be other level sets of $\mf{t}$, we see that \eqref{eql.energy_est_2} holds for arbitrary $\tau_\pm \in \R$.
The desired estimate \eqref{eq.energy_est} now follows from \eqref{eql.energy_est_2} and from Gronwall's inequality.
\end{proof}

We will also need a variant of Proposition \ref{thm.energy_est} that is localized to a null cone exterior.

\begin{proposition} \label{thm.energy_est_loc}
Fix $P \in \R^{1+n}$, let $\mc{U}$ be a GTC, and let $\mc{V}$ be a cross-section of $\mc{U}$ satisfying either $\mc{U}_{ t (P) } < \mc{V}$ or $\mc{V} < \mc{U}_{ t (P) }$.
Moreover, we consider Problem \ref{prb.linear_wave}, and we let
\begin{equation}
\label{eq.energy_est_loc_M} \mc{M}_0 := 1 + \sup_{ \mc{U} \cap \mc{D}_P } | V | \text{,} \qquad \mc{M}_1 := 1 + \sup_{ \mc{U} \cap \mc{D}_P } | \mc{X}^{ t, x } | \text{,} \qquad \mc{T} := \sup_{ \mc{V} } | t - t ( P ) |\text{.}
\end{equation}
Then, there exist $C, C' > 0$, depending on $\mc{U}$ and $\mc{V}$, such that
\begin{equation}
\label{eq.energy_est_loc} \int_{ \mc{V} \cap \mc{D}_P } ( | \nabla_{ t, x } \phi |^2 + \mc{M}_0 \cdot \phi^2 ) \leq C e^{ C' ( \mc{M}_0^\frac{1}{2} + \mc{M}_1 ) \mc{T} } \int_{ \mc{U}_{ t (P) } \cap \mc{D}_P } ( | \nabla_{ t, x } \phi |^2 + \mc{M}_0 \cdot \phi^2 ) \text{,}
\end{equation}
for any solution $\phi \in C^\infty ( \mc{U} ) \cap C^1 ( \bar{\mc{U}} )$ of \eqref{eq.linear_wave} that also satisfies $\phi |_{ \partial \mc{U} \cap \mc{D}_P } = 0$.
\end{proposition}

\begin{proof}[Proof sketch.]
Again, we fix a ``time coordinate" $\mf{t} \in C^\infty ( \bar{\mc{U}} )$ such that:
\begin{itemize}
\item $\nabla^\sharp \mf{t}$ is everywhere timelike and past-directed.

\item There exist $\tau_0, \tau_1 \in \R$ such that
\begin{equation}
\label{eql.energy_est_local_time} \mc{V} = \mc{U} \cap \{ \mf{t} = \tau_1 \} \text{,} \qquad \mc{U} \cap \{ \mf{t} = \tau_0 \} \text{,} \qquad | \tau_1 - \tau_0 | \simeq \mc{T} \text{.}
\end{equation}
\end{itemize}
Moreover, let $Q$ be defined as in \eqref{eql.energy_est_emt}, and define the localized energy,
\begin{equation}
\label{eql.eneryg_est_local_energy} \mc{H} ( \tau ) = \int_{ \mc{U} \cap \{ \mf{t} = \tau \} \cap \mc{D}_P } ( | \nabla_{ t, x } \phi |^2 + \mc{M}_0 \cdot \phi^2 ) \text{,} \qquad \tau \in \R \text{.}
\end{equation}
We only prove the case $\mc{U}_{ t (P) } < \mc{V}$ (that is, $\tau_0 < \tau_1$), as the remaining case $\mc{V} < \mc{U}_{ t (P) }$ can be derived analogously.
We also assume $\mc{V} \cap \mc{D}_P \neq \emptyset$, since the result is trivial otherwise.

Again, we let $Z$ be a generator of $\mc{U}$.
The argument proceeds like the proof of Proposition \ref{thm.energy_est}, except we integrate $\nabla^\alpha ( Q_{ \alpha \beta } Z^\beta )$ over $\mc{U} ( \mc{U}_{ t (P) }, \mc{V} ) \cap \mc{D}_P$.
Applying the divergence theorem yields
\begin{align}
\label{eql.energy_est_local_0} \int_{ \mc{U} ( \mc{U}_{ t (P) }, \mc{V} ) \cap \mc{D}_P } \nabla^\alpha ( Q_{\alpha \beta} Z^\beta ) &= \int_{ \mc{V} \cap \mc{D}_P } Q ( \mc{N}', Z ) + \int_{ \mc{U}_{ t (P) } \cap \mc{D}_P } Q ( \mc{N}', Z ) \\
\notag &\qquad + \int_{ \partial \mc{U} ( \mc{U}_{ t (P) }, \mc{V} ) \cap \mc{D}_P } Q ( \mc{N}, Z ) + \mc{I} \text{,}
\end{align}
where $\mc{N}$ and $\mc{N}'$ denotes the outward and inward pointing unit normals to $\mc{U} ( \mc{V}_-, \mc{V}_+ )$, respectively, and where the last term $\mc{I}$ denotes an integral over the null boundary $\mc{U} ( \mc{U}_{ t (P) }, \mc{V} ) \cap \partial \mc{D}_P$.

By computations similar to those in the proof of Proposition \ref{thm.energy_est}, we obtain
\begin{equation}
\label{eql.energy_est_local_01} - \int_{ \mc{V} \cap \mc{D}_P } Q ( \mc{N}', Z ) \simeq_{ \mc{U}, \mc{V} } \mc{H} ( \tau_1 ) \text{,} \qquad \int_{ \mc{U}_{ t (P) } \cap \mc{D}_P } Q ( \mc{N}', Z ) \simeq_{ \mc{U}, \mc{V} } \mc{H} ( \tau_0 ) \text{,}
\end{equation}
as well as
\begin{equation}
\label{eql.energy_est_local_02} \int_{ \partial \mc{U} ( \mc{U}_{ t (P) }, \mc{V} ) \cap \mc{D}_P } Q ( \mc{N}, Z ) = 0 \text{,}
\end{equation}
For the remaining null boundary integral, standard computations using $Q$ imply that $\mc{I} \geq 0$.\footnote{In particular, we use that $Q$ satisfies the so-called \emph{positive energy condition}.}

Next, using \eqref{eq.linear_wave}, we can bound
\begin{align}
\label{eql.energy_est_local_1} \left| \int_{ \mc{U} ( \mc{U}_{ t (P) }, \mc{V} ) \cap \mc{D}_P } \nabla^\alpha ( Q_{\alpha \beta} Z^\beta ) \right| &\leq \int_{ \mc{U} ( \mc{U}_{ t (P) }, \mc{V} ) \cap \mc{D}_P } ( | \Box \phi | | Z \phi | + \mc{M}_0 | \phi | | Z \phi | + | Q_{\alpha \beta} \nabla^\alpha Z^\beta | ) \\
\notag &\lesssim_{ \mc{U}, \mc{V} } ( \mc{M}_0^\frac{1}{2} + \mc{M}_1 ) \int_{ \tau_0 }^{ \tau_1 } \mc{H} ( \tau ) d \tau \text{.}
\end{align}
Combining \eqref{eql.energy_est_local_0}--\eqref{eql.energy_est_local_1} yields
\begin{equation}
\label{eql.energy_est_local_2} \mc{H} ( \tau_1 ) \lesssim_{ \mc{U}, \mc{V} } \mc{H} ( \tau_0 ) + ( \mc{M}_0^\frac{1}{2} + \mc{M}_1 ) \int_{ \tau_0 }^{ \tau_1 } \mc{H} ( \tau ) d \tau \text{.}
\end{equation}
Varying $\mc{V}$ yields that \eqref{eql.energy_est_local_2} holds for any $\tau_1 > 0$, and \eqref{eq.energy_est_loc} follows from Gronwall's inequality.
\end{proof}

\subsubsection{The Preliminary Result}

We conclude this section by proving the preliminary observability estimate for free waves described in Section \ref{sec.intro_results_prelim}.
More specifically, we establish the following:

\begin{theorem} \label{thm.obs_prelim}
Let $( \mc{U}, Z )$ be a GTC, and fix $x_0 \in \R^n$ and $\tau_\pm \in \R$, with $\tau_- < \tau_+$.
Also, assume
\begin{equation}
\label{eq.obs_prelim_ass} \tau_+ - \tau_- > R_+ + R_- \text{,} \qquad R_\pm := \sup_{ ( \tau_\pm, y ) \in \partial \mc{U} } | y - x_0 | \text{,}
\end{equation}
and fix $t_0 \in ( \tau_-, \tau_+ )$ such that
\begin{equation}
\label{eq.obs_prelim_t0} \tau_+ - t_0 > R_+ \text{,} \qquad t_0 - \tau_- > R_- \text{.}
\end{equation}
Then, for any smooth solution $\phi$ of
\begin{equation}
\label{eq.obs_prelim_wave} \Box \phi |_{ \mc{U} } = 0 \text{,} \qquad \phi |_{ \partial \mc{U} } = 0 \text{,}
\end{equation}
we have the observability estimate
\begin{equation}
\label{eq.obs_prelim} \int_{ \mc{U}_{ \tau_\pm } } ( | \nabla_{t, x} \phi |^2 + \phi^2 ) \lesssim_{ \mc{U}, ( t_0, x_0 ), \tau_+ - \tau_- } \int_{ \Gamma_\ast } | \mc{N} \phi |^2 \text{,}
\end{equation}
where $\mc{N}$ denotes the \emph{Minkowski} outer-pointing unit normal of $\mc{U}$, and where
\begin{equation}
\label{eq.obs_prelim_region} \Gamma_\ast := \{ ( \tau, y ) \in \partial \mc{U}_{ \tau_-, \tau_+ } \mid \mc{N} f_P > 0 \} \text{,} \qquad P := ( t_0, x_0 ) \text{.}
\end{equation}
\end{theorem}

\begin{proof}
By the translation symmetry of Minkowski spacetime, we can assume without loss of generality that $P = ( t_0, x_0 ) := 0$ (and hence $f_P = f$).
For convenience, we also define
\begin{equation}
\label{eql.obs_prelim_abbrev} S := \nabla^\sharp f \text{,} \qquad S_\ast := S + \frac{ n - 1 }{4} \text{.}
\end{equation}
Note that direct computations yield
\begin{equation}
\label{eql.obs_prelim_0} \nabla^2 f = \frac{1}{2} g \text{,} \qquad \Box f = \frac{ n + 1 }{ 2 } \text{.}
\end{equation}

The proof of \eqref{eq.obs_prelim} proceeds in a manner analogous to the multiplier proof in \cite{lionj:ctrlstab_hum}, except that here we use $S_\ast \phi$ as our multiplier.
First, recalling \eqref{eq.obs_prelim_wave} and \eqref{eql.obs_prelim_abbrev}, we see that
\begin{equation}
\label{eql.obs_prelim_1} 0 = \int_{ \mc{U}_{ \tau_-, \tau_+ } } \Box \phi S_\ast \phi = \int_{ \mc{U}_{ \tau_-, \tau_+ } } \Box \phi S \phi + \frac{n - 1}{4} \int_{ \mc{U}_{ \tau_-, \tau_+ } } \Box \phi \cdot \phi := I_1 + I_2 \text{.}
\end{equation}
For $I_2$, we integrate by parts (using the Lorentzian divergence theorem \cite{wald:gr}) to obtain
\begin{equation}
\label{eql.obs_prelim_2} I_2 = - \frac{n - 1}{4} \left. \int_{ \mc{U}_\tau } \partial_t \phi \cdot \phi \right|^{ \tau = \tau_+ }_{ \tau = \tau_- } - \frac{n - 1}{4} \int_{ \mc{U}_{ \tau_-, \tau_+ } } \nabla^\alpha \phi \nabla_\alpha \phi \text{.}
\end{equation}

Similarly, for $I_1$, we integrate by parts twice and recall \eqref{eq.f}, \eqref{eq.obs_prelim_wave}, \eqref{eql.obs_prelim_abbrev}, and \eqref{eql.obs_prelim_0}:
\begin{align}
\label{eql.obs_prelim_3} I_1 &= - \left. \int_{ \mc{U}_\tau } \partial_t \phi S \phi \right|^{ \tau = \tau_+ }_{ \tau = \tau_- } + \int_{ \partial \mc{U}_{ \tau_-, \tau_+ } } \mc{N} \phi S \phi - \int_{ \mc{U}_{ \tau_-, \tau_+ } } \nabla_\alpha \phi \nabla^\alpha ( \nabla^\beta f \nabla_\beta \phi ) \\
\notag &= - \left. \int_{ \mc{U}_\tau } \partial_t \phi S \phi \right|^{ \tau = \tau_+ }_{ \tau = \tau_- } + \int_{ \partial \mc{U}_{ \tau_-, \tau_+ } } \mc{N} f ( \mc{N} \phi )^2 - \int_{ \mc{U}_{ \tau_-, \tau_+ } } \nabla_\alpha \phi \nabla^{ \alpha \beta } f \nabla_\beta \phi \\
\notag &\qquad - \frac{1}{2} \int_{ \mc{U}_{ \tau_-, \tau_+ } } \nabla^\beta f \nabla_\beta ( \nabla_\alpha \phi \nabla^\alpha \phi ) \\
\notag &= - \left. \int_{ \mc{U}_\tau } \partial_t \phi S \phi \right|^{ \tau = \tau_+ }_{ \tau = \tau_- } + \frac{1}{2} \left. \int_{ \mc{U}_\tau } \partial_t f \nabla^\alpha \phi \nabla_\alpha \phi \right|^{ \tau = \tau_+ }_{ \tau = \tau_- } + \frac{1}{2} \int_{ \partial \mc{U}_{ \tau_-, \tau_+ } } \mc{N} f ( \mc{N} \phi )^2 \\
\notag &\qquad - \int_{ \mc{U}_{ \tau_-, \tau_+ } } \nabla^{ \alpha \beta } f \nabla_\alpha \phi \nabla_\beta \phi + \frac{1}{2} \int_{ \mc{U}_{ \tau_-, \tau_+ } } \Box f \nabla_\alpha \phi \nabla^\alpha \phi \\
\notag &= - \left. \int_{ \mc{U}_\tau } \partial_t \phi S \phi \right|^{ \tau = \tau_+ }_{ \tau = \tau_- } + \frac{1}{2} \left. \int_{ \mc{U}_\tau } \partial_t f \nabla^\alpha \phi \nabla_\alpha \phi \right|^{ \tau = \tau_+ }_{ \tau = \tau_- } + \frac{1}{2} \int_{ \partial \mc{U}_{ \tau_-, \tau_+ } } \mc{N} f ( \mc{N} \phi )^2 \\
\notag &\qquad + \frac{n - 1}{4} \int_{ \mc{U}_{ \tau_-, \tau_+ } } \nabla_\alpha \phi \nabla^\alpha \phi \text{.}
\end{align}
In particular, for the boundary terms along $\partial \mc{U}_{ \tau_-, \tau_+ }$ in \eqref{eql.obs_prelim_3}, we observed that the only nonzero components of $\nabla \phi$ lie in the $\mc{N}$-direction.
Combining \eqref{eql.obs_prelim_1}--\eqref{eql.obs_prelim_3}, we then obtain
\begin{equation}
\label{eql.obs_prelim_4} \left. \int_{ \mc{U}_\tau } \left( \partial_t \phi S_\ast \phi - \frac{1}{2} \partial_t f \nabla^\alpha \phi \nabla_\alpha \phi \right) \right|^{ \tau = \tau_+ }_{ \tau = \tau_- } = \frac{1}{2} \int_{ \partial \mc{U}_{ \tau_-, \tau_+ } } \mc{N} f ( \mc{N} \phi )^2 \text{.}
\end{equation}

Using \eqref{eq.f} and \eqref{eql.obs_prelim_abbrev}, we expand
\begin{align}
\label{eql.obs_prelim_5} \partial_t \phi S_\ast \phi - \frac{1}{2} \partial_t f \nabla^\alpha \phi \nabla_\alpha \phi &= \frac{1}{2} t ( \partial_t \phi )^2 + \frac{1}{4} \partial_t \phi \nabla^\alpha ( | x |^2 ) \nabla_\alpha \phi - \frac{ n - 1 }{4} \phi \partial_t \phi + \frac{1}{4} t \nabla^\alpha \phi \nabla_\alpha \phi \\
\notag &= \frac{1}{4} t | \nabla_{ t, x } \phi |^2 + \frac{1}{2} \partial_t \phi \Sigma_\ast \phi \text{,}
\end{align}
where
\begin{equation}
\label{eql.obs_prelim_sigma} \Sigma_\ast \phi := \frac{1}{2} [ \nabla^\alpha ( r^2 ) \nabla_\alpha \phi - ( n - 1 ) \phi ] \text{.}
\end{equation}
Combinining \eqref{eql.obs_prelim_4} and \eqref{eql.obs_prelim_5} then yields
\begin{equation}
\label{eql.obs_prelim_6} \frac{ \tau_+ }{2} \int_{ \mc{U}_{ \tau_+ } } | \nabla_{ t, x } \phi |^2 - \frac{ \tau_- }{2} \int_{ \mc{U}_{ \tau_- } } | \nabla_{ t, x } \phi |^2 = - \left. \int_{ \mc{U}_\tau } \partial_t \phi \Sigma_\ast \phi \right|^{ \tau = \tau_+ }_{ \tau = \tau_- } + \int_{ \partial \mc{U}_{ \tau_-, \tau_+ } } \mc{N} f ( \mc{N} \phi )^2 \text{.}
\end{equation}

Now, at $t = \tau_\pm$, we recall \eqref{eql.obs_prelim_sigma} and bound
\begin{align}
\label{eql.obs_prelim_7} \left| \int_{ \mc{U}_{ \tau_\pm } } \partial_t \phi \Sigma_\ast \phi \right| &\leq \frac{ R_\pm }{2} \int_{ \mc{U}_{ \tau_\pm } } ( \partial_t \phi )^2 + \frac{1}{ 2 R_\pm } \int_{ \mc{U}_{ \tau_\pm } } ( \Sigma_\ast \phi )^2 \\
\notag &= \frac{ R_\pm  }{2} \int_{ \mc{U}_{ \tau_\pm } } ( \partial_t \phi )^2 + \frac{1}{ 2 R_{ \tau_\pm } } \int_{ \mc{U}_{ \tau_\pm } } \frac{1}{4} [ \nabla^\alpha ( r^2 ) \nabla_\alpha \phi ]^2 \\
\notag &\qquad + \frac{1}{ 2 R_\pm } \int_{ \mc{U}_{ \tau_\pm } } \left[ \left( \frac{ n - 1 }{2} \right)^2 \phi^2 + \frac{ n - 1 }{4} \nabla^\alpha ( r^2 ) \nabla_\alpha ( \phi^2 ) \right] \text{.}
\end{align}
By \eqref{eq.obs_prelim_ass}, the first two integrands in the right-hand side of \eqref{eql.obs_prelim_7} can be bounded by
\begin{equation}
\label{eql.obs_prelim_80} \left. \left\{ \frac{ R_\pm }{2} ( \partial_t \phi )^2 + \frac{1}{ 2 R_\pm } \cdot \frac{1}{4} [ \nabla^\alpha ( r^2 ) \nabla_\alpha \phi ]^2 \right\} \right|_{ \mc{U}_\pm } \leq \left. \frac{ R_\pm }{2} | \nabla_{ t, x } \phi |^2 \right|_{ \mc{U}_\pm } \text{.}
\end{equation}
Furthermore, the last term in \eqref{eql.obs_prelim_7} can be simplified via integrating by parts (using the classical divergence theorem on a domain in $\R^n$) and recalling from \eqref{eq.obs_prelim_wave} that $\phi$ vanishes on $\partial \mc{U}$:
\begin{equation}
\label{eql.obs_prelim_81} \int_{ \mc{U}_\pm } \nabla^\alpha ( r^2 ) \nabla_\alpha ( \phi^2 ) = - 2 n \int_{ \mc{U}_\pm } \phi^2 \text{.}
\end{equation}
From \eqref{eql.obs_prelim_7}--\eqref{eql.obs_prelim_81}, we then estimate
\begin{align}
\label{eql.obs_prelim_8} \left| \int_{ \mc{U}_{ \tau_\pm } } \partial_t \phi \Sigma_\ast \phi \right| &\leq \frac{ R_\pm }{2} \int_{ \mc{U}_{ \tau_\pm } } | \nabla_{ t, x } \phi |^2 + \frac{1}{ 2 R_\pm } \int_{ \mc{U}_{ \tau_\pm } } \left[ \frac{ ( n - 1 )^2 }{4} - \frac{ n ( n - 1 ) }{2} \right] \phi^2 \\
\notag &= \frac{ R_\pm }{2} \int_{ \mc{U}_{ \tau_\pm } } | \nabla_{ t, x } \phi |^2 - \frac{ n^2 - 1 }{ 8 R_\pm } \int_{ \mc{U}_{ \tau_\pm } } \phi^2 \text{.}
\end{align}

Finally, combining \eqref{eql.obs_prelim_6} and \eqref{eql.obs_prelim_8}, we conclude that
\begin{align}
\label{eql.obs_prelim_9} &\frac{ \tau_+ }{2} \int_{ \mc{U}_{ \tau_+ } } | \nabla_{ t, x } \phi |^2 - \frac{ \tau_- }{2} \int_{ \mc{U}_{ \tau_- } } | \nabla_{ t, x } \phi |^2 + \frac{ n^2 - 1 }{ 8 } \left( \frac{1}{ R_- } \int_{ \mc{U}_{ \tau_- } } \phi^2 + \frac{ 1 }{ R_+ } \int_{ \mc{U}_{ \tau_+ } } \phi^2 \right) \\
\notag &\quad \leq \frac{ R_+ }{2} \int_{ \mc{U}_{ \tau_+ } } | \nabla_{ t, x } \phi |^2 + \frac{ R_- }{2} \int_{ \mc{U}_{ \tau_- } } | \nabla_{ t, x } \phi |^2 + \int_{ \partial \mc{U}_{ \tau_-, \tau_+ } } \mc{N} f ( \mc{N} \phi )^2 \\
\notag &\quad \leq \frac{ R_+ }{2} \int_{ \mc{U}_{ \tau_+ } } | \nabla_{ t, x } \phi |^2 + \frac{ R_- }{2} \int_{ \mc{U}_{ \tau_- } } | \nabla_{ t, x } \phi |^2 + \int_{ \Gamma_\ast } \mc{N} f ( \mc{N} \phi )^2 \text{,}
\end{align}
In the last step of \eqref{eql.obs_prelim_9}, we observed that the sign of $\mc{N} f ( \mc{N} \phi )^2$ in the boundary integral is simply the sign of $\mc{N} f$, hence the part of $\partial \mc{U}_{ \tau_-, \tau_+ }$ where $\mc{N} f < 0$ can be discarded.
The desired estimate \eqref{eq.obs_prelim} now follows by applying the assumption \eqref{eq.obs_prelim_t0} (with $t_0 = 0$) to \eqref{eql.obs_prelim_9}.\footnote{Although \eqref{eql.obs_prelim_9} does not control the $L^2$-norm of $\phi$ when $n = 1$, this can be recovered using the Poincar\'e inequality.}
\end{proof}

\section{Carleman Estimates} \label{sec.carleman}

The objective of this section is to establish the main Carleman estimate of this paper:

\begin{theorem} \label{thm.carleman_est}
Let $\mc{U}$ be a GTC, and fix $P \in \R^{1+n}$.
Moreover, assume there exists $R > 0$ with
\begin{equation}
\label{eq.carleman_domain} \bar{\mc{U}} \cap \mc{D}_P \subseteq \{ r_P < R \} \text{.}
\end{equation}
Fix also $\varepsilon, a, b > 0$, and suppose the following conditions hold:
\begin{equation}
\label{eq.carleman_choices} a \geq n^2 \text{,} \qquad \varepsilon \ll_n b \ll R^{-1} \text{.}
\end{equation}
Then, there exist constants $C, C' > 0$---which are independent of $\varepsilon, a, b, R$---such that given any
\begin{equation}
\label{eq.carleman_dirichlet} \phi \in C^2 ( \mc{U} ) \cap C^1 ( \bar{\mc{U}} ) \text{,} \qquad \phi |_{ \partial \mc{U} \cap \mc{D}_P } = 0 \text{,}
\end{equation}
we have the Carleman estimate
\begin{align}
\label{eq.carleman_est} &\frac{1}{a} \int_{ \mc{U} \cap \mc{D}_P } \zeta^P_{ a, b; \varepsilon } f_P | \Box \phi |^2 + C' \int_{ \partial \mc{U} \cap \mc{D}_P } \zeta^P_{ a, b; \varepsilon } [ ( 1 - \varepsilon r_P ) \mc{N} f_P + \varepsilon f_P \mc{N} r_P ] \cdot | \mc{N} \phi |^2 \\
\notag &\quad \geq C \varepsilon \int_{ \mc{U} \cap \mc{D}_P } \zeta^P_{ a, b; \varepsilon } r_P^{-1} ( | u_P \cdot \partial_{ u_P } \phi |^2 + | v_P \cdot \partial_{ v_P } \phi |^2 + f_P \cdot g^{ a b } \nasla{}^P_a \phi \nasla{}^P_b \phi ) \\
\notag &\quad\qquad + C b a^2 \int_{ \mc{U} \cap \mc{D}_P } \zeta^P_{ a, b; \varepsilon } f_P^{- \frac{1}{2} } \cdot \phi^2 \text{,}
\end{align}
where:
\begin{itemize}
\item $\zeta^P_{ a, b; \varepsilon }$ is the associated Carleman weight
\begin{equation}
\label{eq.carleman_weight} \zeta^P_{ a, b; \varepsilon } := \left\{ \frac{ f_P }{ ( 1 + \varepsilon u_P ) ( 1 - \varepsilon v_P ) } \cdot \exp \left[ \frac{ 2 b f_P^\frac{1}{2} }{ ( 1 - \varepsilon u_P )^\frac{1}{2} ( 1 + \varepsilon v_P )^\frac{1}{2} } \right] \right\}^{2a} \text{.}
\end{equation}

\item $\mc{N}$ is the outer-pointing unit normal of $\mc{U}$ (with respect to $g$).
\end{itemize}
\end{theorem}

\begin{remark}
We note that by \eqref{eq.carleman_choices}, neither $1 + \varepsilon u_P$ or $1 - \varepsilon v_P$ can vanish on $\bar{\mc{U}} \cap \mc{D}_P$; see Proposition \ref{thm.uvf_bound}.
As a result, the Carleman weight $\zeta^P_{ a, b; \varepsilon }$ in \eqref{eq.carleman_weight} is always well-defined.
\end{remark}

The proof of Theorem \ref{thm.carleman_est} consists of two main steps:
\begin{itemize}
\item First, by a conformal transformation, we reduce the desired estimate \eqref{eq.carleman_est} to a corresponding Carleman estimate in the exterior $\mc{D}$ of a null cone, but with respect to a ``warped" Minkowski metric.
This step is carried out in Section \ref{sec.carleman_warped}, where we discuss this warped metric in detail and then perform this conformal reduction.

\item In Section \ref{sec.carleman_proof}, we prove the Carleman estimate in the warped Minkowski spacetime.
\end{itemize}

\subsection{The Warped Geometry} \label{sec.carleman_warped}

This section is devoted to discussing the warped geometries that form the backbone of the proof of Theorem \ref{thm.carleman_est}.
We first establish some basic properties of these geometries, and we then demonstrate how they are conformally related to the Minkowski metric.
Finally, at the end of this section, we reduce the proof of Theorem \ref{thm.carleman_est} to that of establishing a corresponding Carleman estimate on the warped spacetime; see Theorem \ref{thm.carleman_est_wp}.

\subsubsection{The Warped Metric}

The warped metric can be formally defined as follows:

\begin{definition} \label{def.wp_fct}
Fix a constant $\varepsilon \in \R$, which we refer to as the \emph{warping factor}.
\end{definition}

\begin{definition} \label{def.wp_met}
We define the $\varepsilon$\emph{-warped Minkowski metric} on $\R^{1+n} \setminus \{ r = 0 \}$ by
\begin{equation}
\label{eq.g_wp} \bar{g} := - 4 d u d v + \bar{\rho}^2 \mathring{\gamma} \text{.}
\end{equation}
where $\mathring{\gamma}$ is the unit round metric on the level spheres of $( u, v )$, and where $\rho$ is the \emph{warped radius}:
\begin{equation}
\label{eq.rho_wp} \bar{\rho} := r + 2 \varepsilon f \text{,}
\end{equation}
\end{definition}

\begin{remark}
Note that $\bar{g}$ and its associated objects depend on the warping factor $\varepsilon$.
\end{remark}

We will adopt the following notational conventions regarding objects associated with $\bar{g}$.

\begin{definition} \label{def.wp_notation}
In general, objects defined with respect to the warped metric $\bar{g}$ are denoted with the same symbols as the corresponding objects defined with respect to the Minkowski metric $g$, except that these symbols will have a ``bar" over them.
For example:
\begin{itemize}
\item Let $\bar{\nabla}$ denote the Levi-Civita connection with respect to $\bar{g}$.

\item Let $\bar{\Box} := \bar{g}^{\alpha \beta} \bar{\nabla}_{\alpha \beta}$ denote the wave operator with respect to $\bar{g}$.

\item Let $\bar{\nasla}$ denote the connections induced by $\bar{g}$ on the level spheres of $(u, v)$.
\end{itemize}
\end{definition}

\begin{remark}
Note that $\bar{g}$ is simply the Minkowsi metric $g$ when $\varepsilon = 0$.
In particular, all of the general properties that we establish for $\bar{g}$ will also apply to $g$.
\end{remark}

We now list the results of some basic computations involving $\bar{g}$:

\begin{proposition} \label{thm.f_deriv_wp}
The following properties hold:
\begin{itemize}
\item The $\bar{g}$-gradient $\bar{\nabla}^\sharp f$ of $f$ satisfies
\begin{equation}
\label{eq.f_grad_wp} \bar{\nabla}^\sharp f = \frac{1}{2} ( u \partial_u + v \partial_v ) \text{,} \qquad \bar{\nabla}^\alpha f \bar{\nabla}_\alpha f = f \text{.}
\end{equation}

\item The nonzero components of $\bar{\nabla}^2 f$, in null coordinates, are given by
\begin{equation}
\label{eq.f_hess_wp} \bar{\nabla}_{u v} f \equiv -1 \text{,} \qquad \bar{\nabla}_{a b} f = \left( \frac{1}{2} + \frac{ \varepsilon f }{ \bar{\rho} } \right) \bar{g}_{a b} \text{,}
\end{equation}

\item $f$ also satisfies the following identities:
\begin{equation}
\label{eq.f_box_wp} \bar{\Box} f = \frac{n + 1}{2} + \frac{ (n - 1) \varepsilon f }{ \bar{\rho} } \text{,} \qquad \bar{\nabla}^\alpha f \bar{\nabla}^\beta f \bar{\nabla}_{\alpha \beta} f = \frac{1}{2} f \text{.}
\end{equation}
\end{itemize}
In the above, all indices are raised and lowered using $\bar{g}$.
\end{proposition}

\begin{proposition} \label{thm.f_rho_wp}
The quantity $\bar{\rho}^{-1} f$ satisfies
\begin{equation}
\label{eq.f_rho_deriv_wp} \partial_u \left( \frac{f}{ \bar{\rho} } \right) = - \frac{ v^2 }{ \bar{\rho}^2 } \text{,} \qquad \partial_v \left( \frac{f}{ \bar{\rho} } \right) = \frac{ u^2 }{ \bar{\rho}^2 } \text{,}
\end{equation}
as well as the following wave equation:
\begin{equation}
\label{eq.f_rho_box_wp} \bar{\Box} \left( \frac{f}{ \bar{\rho} } \right) = \frac{ n - 1 }{ 2 \bar{\rho} } \left( 1 - \frac{ 2 \varepsilon f }{ \bar{\rho} } \right) - \frac{ (n - 3) f }{ \bar{\rho}^3 } \text{.}
\end{equation}
\end{proposition}

Next, we recall the region $\mc{D}$ from Definition \ref{def.D}, representing the exterior of the (Minkowski) null cone about the origin.
We can then establish the following basic inequalities on $\mc{D}$:

\begin{proposition} \label{thm.uvf_bound}
The following inequalities hold on $\mc{D}$:
\begin{equation}
\label{eq.uvf_bound} 0 < - u < r \text{,} \qquad 0 < v < r \text{,} \qquad 0 < f < r^2 \text{.}
\end{equation}
Furthermore, if $\varepsilon \geq 0$, then the following inequality holds on $\mc{D}$:
\begin{equation}
\label{eq.rho_est_wp} f^\frac{1}{2} < \bar{\rho} \text{.}
\end{equation}
\end{proposition}

\begin{proof}
First, \eqref{eq.uvf_bound} follows immediately from \eqref{eq.uv}, \eqref{eq.f}, and \eqref{eq.D}.
Then, by \eqref{eq.rho_wp}, we have
\[
f < r^2 = ( \bar{\rho} - 2 \varepsilon f )^2 \text{.}
\]
Since $f > 0$ on $\mc{D}$ and $\varepsilon \geq 0$, taking square roots of the above yields \eqref{eq.rho_est_wp}.
\end{proof}

Next, we derive the null convexity properties of level sets of $f$ (with respect to $\bar{g}$) whenever $\varepsilon$ is positive.\footnote{Null convexity (or pseudoconvexity, in more general contexts) of these level sets can be characterized as follows: $\bar{\nabla}_{ X X } f > 0$ for all null vector fields $X$ that are tangent to the level sets of $f$.
In particular, observe that the preceding property is a consequence of the results for $\bar{\pi}$ in Proposition \ref{thm.TN_wp}.}
This is most straightforwardly shown using a frame that is adapted to $f$.

\begin{definition} \label{def.TN}
We define the following vector fields on $\mc{D}$,
\begin{equation}
\label{eq.TN} T := \frac{1}{2} f^{-\frac{1}{2}} ( - u \partial_u + v \partial_v ) \text{,} \qquad N := \frac{1}{2} f^{-\frac{1}{2}} ( u \partial_u + v \partial_v ) \text{,}
\end{equation}
both of which are normal to the level spheres of $( u, v )$.
\end{definition}

\begin{proposition} \label{thm.TN_wp}
$T$ is everywhere tangent to the level sets of $f$, while $N$ is everywhere normal to the level sets of $f$.
Moreover, $\bar{\nabla}^2 f$ satisfies the following:
\begin{equation}
\label{eq.f_TN_wp} \bar{\nabla}_{T a} f = \bar{\nabla}_{N a} f = \bar{\nabla}_{T N} f \equiv 0 \text{,} \qquad \bar{\nabla}_{T T} f = - \frac{1}{2} \text{,} \qquad \bar{\nabla}_{N N} f = \frac{1}{2} \text{.}
\end{equation}
\end{proposition}

For our purposes, the null convexity is most conveniently captured in the quantity $\pi$ defined below.
This will be exploited in the proof of the ``warped" Carleman estimate, Theorem \ref{thm.carleman_est_wp}.

\begin{definition} \label{def.pseudoconvex_wp}
We define the following modified deformation tensor:
\begin{equation}
\label{eq.pi_wp} \bar{\pi} := \bar{\nabla}^2 f - \bar{h} \cdot \bar{g} \text{,} \qquad \bar{h} := \frac{1}{2} + \frac{ \varepsilon f }{ 2 \bar{\rho} } \text{.}
\end{equation}
In future computations, we will also make use of the following quantity:
\begin{equation}
\label{eq.w_wp} \bar{w} := \frac{1}{2} \bar{\Box} f - \bar{h} = \frac{n - 1}{4} + \frac{ (n - 2) \varepsilon f }{ 2 \bar{\rho} } \text{.}
\end{equation}
\end{definition}

\begin{proposition} \label{thm.pseudoconvex_wp}
The nonvanishing components of $\bar{\pi}$ are given by
\begin{equation}
\label{eq.pseudoconvex_wp} \bar{\pi}_{T T} = \frac{ \varepsilon f }{ 2 \bar{\rho} } \text{,} \qquad \bar{\pi}_{a b} = \frac{ \varepsilon f }{ 2 \bar{\rho} } \cdot \bar{g}_{a b} \text{,} \qquad \bar{\pi}_{N N} = - \frac{ \varepsilon f }{ 2 \bar{\rho} } \text{.}
\end{equation}
Furthermore, $\bar{w}$ satisfies the following wave equation:
\begin{equation}
\label{eq.w_box_wp} \bar{\Box} \bar{w} = - \frac{ (n - 2) \varepsilon }{ 2 \bar{\rho} } \left[ \frac{ (n - 3) f }{ \bar{\rho}^2 } - \frac{ n - 1 }{ 2 } \left( 1 - \frac{ 2 \varepsilon f }{ \bar{\rho} } \right) \right] \text{.}
\end{equation}
\end{proposition}

\begin{remark} \label{rmk.pseudoconvex}
In particular, one can show from the first two identities in \eqref{eq.pseudoconvex_wp} that
\[
\bar{\nabla}_{ X X } f = \bar{\pi}_{ X X } > 0 \text{,}
\]
for any null vector field $X$ tangent to the level sets of $f$.
In other words, these level sets are ($\bar{g}$)\emph{-null convex} in the direction of increasing $f$.
This null convexity is also connected to the more general notion of \emph{pseudoconvexity} (with respect to $\bar{\Box}$) in unique continuation theory; see \cite{hor:lpdo4, ler_robb:unique}.
\end{remark}

\subsubsection{The Conformal Isometry}

The next step is to describe how this warped geometry, as described in Definitions \ref{def.wp_fct} and \ref{def.wp_met}, is conformally related to Minkowski geometry.

\begin{definition} \label{def.conformal}
Let $R > 0$ and $\varepsilon \in \R$, and assume $| \varepsilon | \ll_n R^{-1}$.
We then define the map
\[
\bar{\Phi}: \mc{D} \cap \{ r < R \} \rightarrow \mc{D} \text{,}
\]
in terms of null coordinates (about the origin) by
\begin{align}
\label{eq.conformal} \bar{\Phi} ( u, v, \omega ) &:= ( \bar{u} ( u, v, \omega ), \bar{v} ( u, v, \omega ), \bar{\omega} ( u, v, \omega ) ) \\
\notag &:= ( u ( 1 + \varepsilon u )^{-1}, v ( 1 - \varepsilon v )^{-1}, \omega ) \text{.}
\end{align}
In addition, for future convenience, we define the function
\begin{equation}
\label{eq.conformal_factor} \xi := ( 1 + \varepsilon u ) ( 1 - \varepsilon v ) \text{.}
\end{equation}
\end{definition}

\begin{remark}
Note that by \eqref{eq.uvf_bound}, the right-hand side of \eqref{eq.conformal} is well-defined.
\end{remark}

\begin{proposition} \label{thm.met_conf}
Assume the setting of Definition \ref{def.conformal}.
Then:
\begin{itemize}
\item The following identities hold:
\begin{equation}
\label{eq.rf_conf} f \circ \bar{\Phi} = \xi^{-1} f \text{,} \qquad \bar{\rho} \circ \bar{\Phi} = \xi^{-1} r \text{.}
\end{equation}

\item $\bar{\Phi}$ defines a conformal isometry between $\mc{D} \cap \{ r < R \}$ and an open subset of $\mc{D}$.
In particular,
\begin{equation}
\label{eq.met_conf} \bar{\Phi}_\ast \bar{g} = \xi^{-2} g |_{ \mc{D} \cap \{ r < R \} } \text{,}
\end{equation}
where $\bar{\Phi}_\ast$ denotes the pullback of $\bar{g}$ through $\bar{\Phi}$.
\end{itemize}
\end{proposition}

\begin{proposition} \label{thm.conf_wave}
Assume the setting of Definition \ref{def.conformal}.
Also, let $\Omega$ be an open subset of $\mc{D}$, let $\bar{\phi} \in C^2 ( \Omega )$, and let $\phi = \bar{\phi} \circ \bar{\Phi}$.
Then, the following identity holds:
\begin{equation}
\label{eq.conf_wave} \left[ \bar{\Box} + \frac{ (n - 1)^2 \varepsilon }{ 2 ( \bar{\rho} \circ \bar{\Phi} ) } \right] ( \xi^\frac{n - 1}{2} \bar{\phi} ) = \xi^\frac{n + 3}{2} \Box \phi \text{.}
\end{equation}
\end{proposition}

We now establish a number of elementary comparisons showing that objects with respect to $\bar{g}$ differ little from corresponding objects with respect to $g$:

\begin{proposition} \label{thm.uvf_comp}
Assume the setting of Definition \ref{def.conformal}.
Then:
\begin{itemize}
\item The following comparisons hold on $\mc{D} \cap \{ r < R \}$:
\begin{equation}
\label{eq.conf_comp} ( 1 + \varepsilon u )^n \simeq 1 \text{,} \qquad ( 1 - \varepsilon v )^n \simeq 1 \text{,} \qquad \xi^n \simeq 1 \text{.}
\end{equation}

\item The following comparisons hold on $\mc{D} \cap \{ r < R \}$:
\begin{equation}
\label{eq.uvf_comp} - ( u \circ \bar{\Phi} ) \simeq - u \text{,} \qquad v \circ \bar{\Phi} \simeq v \text{,} \qquad f \circ \bar{\Phi} \simeq f \text{.}
\end{equation}

\item The following comparisons hold on $\mc{D} \cap \{ r < R \}$:
\begin{equation}
\label{eq.conf_d_comp} | \partial_u \xi | \simeq \varepsilon \text{,} \qquad | \partial_v \xi | \simeq \varepsilon \text{.}
\end{equation}

\item For any open subset $\Omega \subseteq \mc{D}$ and $\bar{\phi} \in C^1 ( \Omega )$, we have
\begin{equation}
\label{eq.deriv_comp} | \partial_u \bar{\phi} | \simeq | \partial_u \phi | \text{,} \qquad | \partial_v \bar{\phi} | \simeq | \partial_v \phi | \text{,} \qquad \bar{g}^{ab} \bar{\nasla}_a \bar{\phi} \bar{\nasla}_b \bar{\phi} \simeq g^{ab} \nasla_a \phi \nasla_b \phi \text{.}
\end{equation}
\end{itemize}
\end{proposition}

\begin{proof}
Since $0 < -u, v < R$ by \eqref{eq.uvf_bound}, the smallness assumption on $\varepsilon$ implies the first two comparisons in \eqref{eq.conf_comp}; the remaining part of \eqref{eq.conf_comp} now follows from \eqref{eq.conformal_factor}.
Using \eqref{eq.conf_comp}, then:
\begin{itemize}
\item \eqref{eq.uvf_comp} follows from \eqref{eq.conformal} and \eqref{eq.rf_conf}.

\item \eqref{eq.conf_d_comp} follows from direct computations along with \eqref{eq.conf_comp}.

\item \eqref{eq.deriv_comp} follows from \eqref{eq.conformal} and \eqref{eq.conf_comp}.\qedhere
\end{itemize}
\end{proof}

\begin{proposition} \label{thm.gtc_conf}
Assume the setting of Definition \ref{def.conformal}, and suppose $\mc{U}$ is a GTC satisfying
\[
\bar{\mc{U}} \cap \mc{D} \subseteq \{ r < R \} \text{.}
\]
Then, $\bar{\Phi} ( \mc{U} \cap \mc{D} )$ is an open subset of $\mc{D}$, and its boundary in $\mc{D}$ is the hypersurface $\bar{\Phi} ( \partial \mc{U} \cap \mc{D} )$.
\end{proposition}

\begin{proof}
This is an immediate consequence of $\bar{\Phi}$ being a conformal isometry (see Proposition \ref{thm.met_conf}) and of the fact that conformal isometries preserve the causal geometry.
\end{proof}

\subsubsection{The Warped Carleman Estimate}

We now reduce the proof of Theorem \ref{thm.carleman_est} to establishing an intermediate Carleman estimate on the warped exterior $( \mc{D}, \bar{g} )$.
The precise statement of this intermediate estimate is given in the subsequent theorem:

\begin{theorem} \label{thm.carleman_est_wp}
Fix $R > 0$, and let $\mc{U}$ be a GTC satisfying
\begin{equation}
\label{eq.carleman_domain_wp} \bar{\mc{U}} \cap \mc{D} \subseteq \{ r < R \} \text{.}
\end{equation}
Fix also $\varepsilon, a, b > 0$ such that the following conditions hold:
\begin{equation}
\label{eq.carleman_choices_wp} a \geq n^2 \text{,} \qquad \varepsilon \ll_n b \ll R^{-1} \text{.}
\end{equation}
Then, for any $\phi \in C^2 ( \bar{\Phi} ( \mc{U} \cap \mc{D} ) ) \cap C^1 ( \bar{\Phi} ( \bar{\mc{U}} \cap \mc{D} ) )$ that is uniformly $C^1$-bounded and satisfies
\begin{equation}
\label{eq.carleman_dirichlet_wp} \phi |_{ \bar{\Phi} ( \partial \mc{U} \cap \mc{D} ) } = 0 \text{,}
\end{equation}
we have the Carleman estimate
\begin{align}
\label{eq.carleman_est_wp} &\frac{1}{2 a} \int_{ \bar{\Phi} ( \mc{U} \cap \mc{D} ) } \zeta_{a, b} f | \bar{\Box} \phi |^2 + \int_{ \bar{\Phi} ( \partial \mc{U} \cap \mc{D} ) } \zeta_{a, b} \bar{\mc{N}} f | \bar{\mc{N}} \phi |^2 \\
\notag &\quad \geq \frac{ \varepsilon }{ 8 } \int_{ \bar{\Phi} ( \mc{U} \cap \mc{D} ) } \zeta_{a, b} \bar{\rho}^{-1} ( | u \cdot \partial_u \phi |^2 + | v \cdot \partial_v \phi |^2 + f \bar{g}^{ab} \bar{\nasla}_a \phi \bar{\nasla}_b \phi ) + \frac{ b a^2 }{4} \int_{ \bar{\Phi} ( \mc{U} \cap \mc{D} ) } \zeta_{a, b} f^{- \frac{1}{2} } \phi^2 \text{,}
\end{align}
where all integrals are with respect to volume forms induced by $\bar{g}$, and where:
\begin{itemize}
\item $\bar{\Phi}$ is as given in Definition \ref{def.conformal}.

\item $\zeta_{a, b}$ is the (warped) Carleman weight
\begin{equation}
\label{eq.carleman_weight_wp} \zeta_{a, b} := f^{2 a} e^{ 4 a b f^\frac{1}{2} } \text{.}
\end{equation}

\item $\bar{\mc{N}}$ is the outer-pointing unit normal of $\bar{\Phi} ( \mc{U} \cap \mc{D} )$ (with respect to $\bar{g}$).
\end{itemize}
\end{theorem}

\begin{remark}
Note that by the conclusions of Proposition \ref{thm.gtc_conf}, both the integral along $\bar{\Phi} ( \partial \mc{U} \cap \mc{D} )$ and the outer unit normal $\bar{\mc{N}}$ to $\bar{\Phi} ( \mc{U} \cap \mc{D} )$ are well-defined objects.
\end{remark}

\begin{remark}
The factors $f^{1/2}$ present in the last term of \eqref{eq.carleman_est_wp} and in \eqref{eq.carleman_weight_wp} could be replaced by $f^p$, for some $p > 0$, provided the constants of the inequality \eqref{eq.carleman_est_wp} are also adjusted; see, for instance, a similar estimate in \cite{alex_shao:uc_global}.
However, we will not need this flexibility here.
\end{remark}

Below, we show that by assuming Theorem \ref{thm.carleman_est_wp}, we can recover our main Carleman estimate, Theorem \ref{thm.carleman_est}.
This effectively reduces the proof of Theorem \ref{thm.carleman_est} to that of Theorem \ref{thm.carleman_est_wp}.

\begin{proof}[Proof (Theorem \ref{thm.carleman_est_wp} $\Rightarrow$ Theorem \ref{thm.carleman_est})]
Assume the hypotheses of Theorem \ref{thm.carleman_est}; in particular, we let $R$, $\mc{U}$, $a$, $b$, $\varepsilon$, $\phi$ be as in its statement.
Throughout the proof, we will use ``$d \bar{g}$" and ``$d g$" to denote integrals with respect to volume forms induced by $\bar{g}$ and $g$, respectively.
Also, by the translation symmetry of Minkowski spacetime, we can assume without loss of generality that $P = 0$.

Consider the (uniformly $C^1$-bounded) function
\begin{equation}
\label{eql.carleman_est_0} \bar{\phi} := \phi \circ \bar{\Phi}^{-1} \in C^2 ( \bar{\Phi} ( \mc{U} \cap \mc{D} ) ) \cap C^1 ( \bar{\Phi} ( \bar{\mc{U}} \cap \mc{D} ) ) \text{,}
\end{equation}
which we note satisfies the Dirichlet boundary condition \eqref{eq.carleman_dirichlet_wp}.
Applying Theorem \ref{thm.carleman_est_wp} to the above $R$, $\mc{U}$, $a$, $b$, $\varepsilon$, and with $\xi^\frac{n - 1}{2} \bar{\phi}$ in the place of $\phi$, we obtain
\begin{align}
\label{eql.carleman_est_00} &\frac{1}{2 a} \int_{ \bar{\Phi} ( \mc{U} \cap \mc{D} ) } \zeta_{a, b} f | \bar{\Box} ( \xi^\frac{n-1}{2} \bar{\phi} ) |^2 \cdot d \bar{g} + \int_{ \bar{\Phi} ( \partial \mc{U} \cap \mc{D} ) } \zeta_{a, b} \bar{\mc{N}} f | \bar{\mc{N}} ( \xi^\frac{n - 1}{2} \bar{\phi} ) |^2 \cdot d \bar{g} \\
\notag &\quad \geq \frac{ \varepsilon }{ 8 } \int_{ \bar{\Phi} ( \mc{U} \cap \mc{D} ) } \zeta_{a, b} \bar{\rho}^{-1} [ | u \cdot \partial_u ( \xi^\frac{n-1}{2} \bar{\phi} ) |^2 + | v \cdot \partial_v ( \xi^\frac{n-1}{2} \bar{\phi} ) |^2 + f \bar{g}^{ab} \xi^{n - 1} \bar{\nasla}_a \bar{\phi} \bar{\nasla}_b \bar{\phi} ] \cdot d \bar{g} \\
\notag &\quad\qquad + \frac{ b a^2 }{4} \int_{ \bar{\Phi} ( \mc{U} \cap \mc{D} ) } \zeta_{a, b} f^{- \frac{1}{2} } \xi^{n-1} \bar{\phi}^2 \cdot d \bar{g} \text{.}
\end{align}

We now pull each integral in \eqref{eql.carleman_est_00} back through the diffeomorphism $\bar{\Phi}$ and bound some extraneous factors.
First, for the last term on the right-hand side of \eqref{eql.carleman_est_00}, we have that
\begin{align}
\label{eql.carleman_est_10} \int_{ \bar{\Phi} ( \mc{U} \cap \mc{D} ) } \zeta_{a, b} f^{- \frac{1}{2} } \xi^{n-1} \bar{\phi}^2 \cdot d \bar{g} &= \int_{ \mc{U} \cap \mc{D} } \zeta^0_{a, b; \varepsilon} ( \xi^{-1} f )^{ - \frac{1}{2} } \xi^{n-1} \phi^2 \cdot \xi^{-n-1} d g \\
\notag &\gtrsim \int_{ \mc{U} \cap \mc{D} } \zeta^0_{a, b; \varepsilon} f^{ - \frac{1}{2} } \phi^2 \cdot d g \text{.}
\end{align}
where we in particular note the following:
\begin{itemize}
\item The extra factor $\xi^{-n-1}$ arises from the change of volume forms from $\bar{g}$ to $g$.

\item By \eqref{eq.rf_conf}, all instances of $f$ in the left-hand side of \eqref{eql.carleman_est_10} are replaced by $\xi^{-1} f$.
Note that in particular, this results in the weight $\zeta_{a, b}$ being replaced by $\zeta^0_{a, b; \varepsilon}$; see \eqref{eq.carleman_weight}.

\item The last inequality in \eqref{eql.carleman_est_10} is a consequence of \eqref{eq.conf_comp}.
\end{itemize}

By similar reasoning, we also obtain
\begin{equation}
\label{eql.carleman_est_11} \int_{ \bar{\Phi} ( \mc{U} \cap \mc{D} ) } \zeta_{a, b} \bar{\rho}^{-1} f \bar{g}^{ab} \xi^{n - 1} \bar{\nasla}_a \bar{\phi} \bar{\nasla}_b \bar{\phi} \cdot d \bar{g} \gtrsim \int_{ \mc{U} \cap \mc{D} } \zeta^0_{a, b; \varepsilon} r^{-1} f g^{ab} \nasla_a \phi \nasla_b \phi \cdot d g \text{,}
\end{equation}
where the factor $\bar{\rho}^{-1}$ is also handled using \eqref{eq.rf_conf}.
Next, writing
\[
\bar{\Box} ( \xi^\frac{n-1}{2} \bar{\phi} ) = \left[ \bar{\Box} + \frac{ (n - 1)^2 \varepsilon }{ 2 \bar{\rho} } \right] ( \xi^\frac{n - 1}{2} \bar{\phi} ) - \frac{ (n - 1)^2 \varepsilon }{ 2 \bar{\rho} } ( \xi^\frac{n - 1}{2} \bar{\phi} ) \text{,}
\]
and recalling \eqref{eq.conf_wave}, we see (using the same reasoning as before) that
\begin{align}
\label{eql.carleman_est_12} \int_{ \bar{\Phi} ( \mc{U} \cap \mc{D} ) } \zeta_{a, b} f | \bar{\Box} ( \xi^\frac{n-1}{2} \bar{\phi} ) |^2 \cdot d \bar{g} &\lesssim \int_{ \mc{U} \cap \mc{D} } \zeta^0_{a, b; \varepsilon} f | \Box \phi |^2 \cdot d g + n^4 \varepsilon^2 \int_{ \mc{U} \cap \mc{D} } \zeta^0_{a, b; \varepsilon} r^{-2} f \phi^2 \cdot d g \text{.}
\end{align}

For the boundary term, we observe from Proposition \ref{thm.gtc_conf} that $\bar{\Phi} ( \partial \mc{U} \cap \mc{D} )$ is mapped via $\bar{\Phi}^{-1}$ to $\partial \mc{U} \cap \mc{D}$.
Moreover, Proposition \ref{thm.gtc_conf} implies that the push-forward of $\bar{\mc{N}}$ through $\bar{\Phi}^{-1}$ is precisely $\xi \mc{N}$, where $\mc{N}$ is as in the statement of Theorem \ref{thm.carleman_est}.
As a result, we have
\begin{align}
\label{eql.carleman_est_13} \int_{ \bar{\Phi} ( \partial \mc{U} \cap \mc{D} ) } \zeta_{a, b} \bar{\mc{N}} f | \bar{\mc{N}} ( \xi^\frac{n - 1}{2} \bar{\phi} ) |^2 \cdot d \bar{g} &= \int_{ \partial \mc{U} \cap \mc{D} } \zeta^0_{a, b; \varepsilon} ( \xi \mc{N} ) ( \xi^{-1} f ) | ( \xi \mc{N} ) ( \xi^\frac{n - 1}{2} \phi ) |^2 \cdot \xi^{-n} d g \\
\notag &= \int_{ \partial \mc{U} \cap \mc{D} } \zeta^0_{a, b; \varepsilon} \xi^2 \mc{N} ( \xi^{-1} f ) | \mc{N} \phi |^2 \cdot d g \text{,}
\end{align}
where in the last step, we also used \eqref{eq.carleman_dirichlet}.

For the remaining first-order bulk terms, we apply \eqref{eq.conformal}, \eqref{eq.conf_comp}, \eqref{eq.conf_d_comp}, and \eqref{eq.deriv_comp} to obtain
\begin{align}
\label{eql.carleman_est_14} \int_{ \bar{\Phi} ( \mc{U} \cap \mc{D} ) } \zeta_{a, b} \bar{\rho}^{-1} | u \cdot \partial_u ( \xi^\frac{n-1}{2} \bar{\phi} ) |^2 \cdot d \bar{g} &\geq C \int_{ \mc{U} \cap \mc{D} } \zeta^0_{a, b; \varepsilon} r^{-1} | u \cdot \partial_u ( \xi^\frac{n-1}{2} \bar{\phi} ) |^2 \cdot d g \\
\notag &\geq C_1 \int_{ \mc{U} \cap \mc{D} } \zeta^0_{a, b; \varepsilon} r^{-1} | u \cdot \partial_u \phi |^2 \cdot d g \\
\notag &\qquad - C_2 \varepsilon^2 n^2 \int_{ \mc{U} \cap \mc{D} } \zeta^0_{a, b; \varepsilon} r^{-1} u^2 \phi^2 \cdot d g \text{,}
\end{align}
for some universal constants $C, C_1, C_2 > 0$.
A similar computation yields
\begin{align}
\label{eql.carleman_est_15} \int_{ \bar{\Phi} ( \mc{U} \cap \mc{D} ) } \zeta_{a, b} \bar{\rho}^{-1} | v \cdot \partial_v ( \xi^\frac{n-1}{2} \bar{\phi} ) |^2 \cdot d \bar{g} &\geq C_1 \int_{ \mc{U} \cap \mc{D} } \zeta^0_{a, b; \varepsilon} r^{-1} | v \cdot \partial_v \phi |^2 \cdot d g \\
\notag &\qquad - C_2 \varepsilon^2 n^2 \int_{ \mc{U} \cap \mc{D} } \zeta^0_{a, b; \varepsilon} r^{-1} v^2 \phi^2 \cdot d g \text{.}
\end{align}

Finally, combining \eqref{eql.carleman_est_00} with \eqref{eql.carleman_est_10}--\eqref{eql.carleman_est_15} yields
\begin{align}
\label{eql.carleman_est_20} &\frac{1}{a} \int_{ \mc{U} \cap \mc{D} } \zeta^0_{a, b; \varepsilon} f | \Box \phi |^2 \cdot d g + C' \int_{ \partial \mc{U} \cap \mc{D} } \zeta^0_{a, b; \varepsilon} \xi^2 \mc{N} ( \xi^{-1} f ) | \mc{N} \phi |^2 \cdot d g \\
\notag &\quad \geq C \varepsilon \int_{ \mc{U} \cap \mc{D} } \zeta^0_{a, b; \varepsilon} r^{-1} [ | u \cdot \partial_u \phi |^2 + | v \cdot \partial_v \phi |^2 + f g^{ab} \nasla_a \phi \nasla_b \phi ] \cdot d g \\
\notag &\quad\qquad + C b a^2 \int_{ \mc{U} \cap \mc{D} } \zeta^0_{a, b; \varepsilon} f^{- \frac{1}{2} } \phi^2 \cdot d g - \frac{ C'' n^4 \varepsilon^2 }{a} \int_{ \mc{U} \cap \mc{D} } \zeta^0_{a, b; \varepsilon} r^{-2} f \phi^2 \cdot d g \\
\notag &\quad\qquad - C'' \varepsilon^3 n^2 \int_{ \mc{U} \cap \mc{D} } \zeta^0_{a, b; \varepsilon} r^{-1} ( u^2 + v^2 ) \phi^2 \cdot d \bar{g} \text{,}
\end{align}
for some constants $C, C', C'' > 0$.
By \eqref{eq.carleman_domain}, \eqref{eq.carleman_choices}, and \eqref{eq.uvf_bound}, we have
\begin{align}
\label{eql.carleman_est_21} \frac{ \varepsilon^2 n^4 }{a} \cdot \frac{f}{ r^2 } &\leq \varepsilon^2 a \leq \varepsilon \cdot \varepsilon a^2 \cdot f^\frac{1}{2} f^{ - \frac{1}{2} } \ll b \cdot R^{-1} a^2 \cdot R f^{ - \frac{1}{2} } = b a^2 f^{ - \frac{1}{2} } \text{,} \\
\notag \frac{ \varepsilon^3 n^2 ( u^2 + v^2 ) }{r} &\leq 2 \varepsilon^3 n^2 R^2 \cdot r^{-1} f^\frac{1}{2} \cdot f^{ -\frac{1}{2} } \leq 2 \varepsilon^3 n^2 R^2 \cdot f^{ - \frac{1}{2} } \ll b a^2 f^{ - \frac{1}{2} } \text{.}
\end{align}
In particular, \eqref{eql.carleman_est_21} implies that the last two (negative) terms in the right-hand side of \eqref{eql.carleman_est_20} can be absorbed into the third last (positive) term.
This results in the inequality
\begin{align}
\label{eql.carleman_est_22} &\frac{1}{a} \int_{ \mc{U} \cap \mc{D} } \zeta^0_{a, b; \varepsilon} \cdot f | \Box \phi |^2 \cdot d g + C' \int_{ \partial \mc{U} \cap \mc{D} } \zeta^0_{a, b; \varepsilon} \cdot \xi^2 \mc{N} ( \xi^{-1} f ) | \mc{N} \phi |^2 \cdot d g \\
\notag &\quad \geq C \varepsilon \int_{ \mc{U} \cap \mc{D} } \zeta^0_{a, b; \varepsilon} \cdot r^{-1} [ | u \cdot \partial_u \phi |^2 + | v \cdot \partial_v \phi |^2 + f g^{ab} \nasla_a \phi \nasla_b \phi ] \cdot d g \\
\notag &\quad\qquad + C b a^2 \int_{ \mc{U} \cap \mc{D} } \zeta^0_{a, b; \varepsilon} \cdot f^{- \frac{1}{2} } \phi^2 \cdot d g \text{,}
\end{align}
for some constants $C, C' > 0$.
Theorem \ref{thm.carleman_est} now follows from \eqref{eql.carleman_est_22} and the identity
\[
\xi^2 \mc{N} ( \xi^{-1} f ) = \xi \mc{N} f - f \mc{N} \xi = ( 1 - \varepsilon r ) \mc{N} f + \varepsilon f \mc{N} r \text{.} \qedhere
\]
\end{proof}

\subsection{Proof of Theorem \ref{thm.carleman_est_wp}} \label{sec.carleman_proof}

It remains only to prove the warped Carleman estimate, Theorem \ref{thm.carleman_est_wp}, in order to complete the proof of Theorem \ref{thm.carleman_est}.
This is the topic of the present subsection.

Throughout (only) this subsection, we will assume that all indices are raised and lowered using $\bar{g}$.
We begin by defining some auxiliary quantities to aid with the proof:

\begin{definition} \label{def.F_ab}
Assume the hypotheses of Theorem \ref{thm.carleman_est_wp}, and define
\begin{equation}
\label{eq.F} F := F (f) := - a ( \log f + 2 b f^\frac{1}{2} ) \text{.}
\end{equation}
Furthermore, for brevity, we will let $\prime$ denote differentiation with respect to $f$.
\end{definition}

\begin{definition} \label{def.F_conj}
Assume the hypotheses of Theorem \ref{thm.carleman_est_wp}, and define the operators
\begin{equation}
\label{eq.F_conj} \bar{\mc{L}} := e^{-F} \bar{\Box} e^F \text{,} \qquad \bar{S} := \bar{\nabla}^\sharp f \text{,} \qquad \bar{S}_w := \bar{S} + \bar{w} \text{.}
\end{equation}
\end{definition}

\begin{lemma} \label{thm.F_deriv}
Assuming the hypotheses of Theorem \ref{thm.carleman_est_wp}, we have that
\begin{equation}
\label{eq.F_deriv} e^{-2 F} = \zeta_{a, b} \text{,} \qquad F' = - a ( f^{-1} + b f^{- \frac{1}{2} } ) \text{.}
\end{equation}
\end{lemma}

\begin{proof}
These are direct computations.
\end{proof}

\subsubsection{The Pointwise Identity}

The first step in proving Theorem \ref{thm.carleman_est_wp} is to establish a pointwise identity for the conjugated wave operator $\bar{\mc{L}}$ defined in \eqref{eq.F_conj}:

\begin{lemma} \label{thm.carleman_id_ptwise}
Assume the hypotheses of Theorem \ref{thm.carleman_est_wp}, and suppose also that $\psi \in C^2 ( \bar{\Phi} ( \mc{U} \cap \mc{D} ) )$.
Then, at every point of $\bar{\Phi} ( \mc{U} \cap \mc{D} )$, we have the identity
\begin{align}
\label{eq.carleman_id_ptwise} - \bar{\mc{L}} \psi \bar{S}_w \psi + \bar{\nabla}^\beta \bar{P}_\beta &= - 2 F' \cdot | \bar{S}_w \psi |^2 + \frac{ \varepsilon f }{ 2 \bar{\rho} } ( | T \psi |^2 + | \bar{\nasla} \psi |^2 - | N \psi |^2 ) \\
\notag &\qquad - \frac{ \varepsilon f }{ \bar{\rho} } F' \cdot \psi \bar{S}_w \psi + \frac{1}{2} \left[ ( f \mc{A} )' + \frac{ \varepsilon f }{ \bar{\rho} } \mc{A} - \bar{\Box} \bar{w} \right] \cdot \psi^2 \text{,}
\end{align}
where the reparametrization $\mc{A} := \mc{A} (f)$ of $f$ is given by
\begin{equation}
\label{eq.carleman_id_A} \mc{A} := f ( F' )^2 + ( f F' )' = a^2 f^{-1} + b a \left( 2 a - \frac{1}{2} \right) f^{- \frac{1}{2} } + b^2 a^2 \text{.}
\end{equation}
and where the current $\bar{P} := \bar{P} [ \psi ]$ is given by
\begin{equation}
\label{eq.carleman_id_current} \bar{P}_\beta := \bar{S} \psi \bar{\nabla}_\beta \psi - \frac{1}{2} \bar{\nabla}_\beta f \cdot \bar{\nabla}^\mu \psi \bar{\nabla}_\mu \psi + \bar{w} \cdot \psi \bar{\nabla}_\beta \psi + \frac{1}{2} ( \mc{A} \bar{\nabla}_\beta f - \bar{\nabla}_\beta \bar{w} ) \cdot \psi^2 \text{.}
\end{equation}
\end{lemma}

\begin{proof}
First, we expand $\bar{\mc{L}} \psi$ by applying Proposition \ref{thm.f_deriv_wp}, Definition \ref{def.pseudoconvex_wp}, and Definition \ref{def.F_ab}:
\begin{align*}
\bar{\mc{L}} \psi &= \bar{\Box} \psi + 2 F' \cdot \bar{S} \psi + ( F' )^2 \bar{\nabla}^\alpha f \bar{\nabla}_\alpha f \cdot \psi + F'' \bar{\nabla}^\alpha f \bar{\nabla}_\alpha f \cdot \psi + F' \bar{\Box} f \cdot \psi \\
&= \bar{\Box} \psi + 2 F' \cdot \bar{S}_w \psi + f ( F' )^2 \cdot \psi + ( f F'' + 2 \bar{h} F' ) \cdot \psi \\
&= \bar{\Box} \psi + 2 F' \cdot \bar{S}_w \psi + ( \mc{A} + \varepsilon f \bar{\rho}^{-1} F' ) \cdot \psi \text{,}
\end{align*}
where $\mc{A}$ is as in \eqref{eq.carleman_id_A}.
Multiplying the above by $\bar{S}_w \psi$ yields
\begin{equation}
\label{eql.carleman_id_ptwise_1} \bar{\mc{L}} \psi \bar{S}_w \psi = \bar{\Box} \psi \bar{S}_w \psi + 2 F' \cdot | \bar{S}_w \psi |^2 + ( \mc{A} + \varepsilon f \bar{\rho}^{-1} F' ) \cdot \psi \bar{S}_w \psi \text{.}
\end{equation}

We can now use Proposition \ref{thm.f_deriv_wp} and Definition \ref{def.pseudoconvex_wp} to obtain
\begin{align*}
\mc{A} \cdot \psi \bar{S}_w \psi &= \frac{1}{2} \mc{A} \cdot \bar{\nabla}^\alpha f \bar{\nabla}_\alpha ( \psi^2 ) + \mc{A} \bar{w} \cdot \psi^2 \\
&= \frac{1}{2} \bar{\nabla}^\alpha ( \mc{A} \bar{\nabla}_\alpha f \cdot \psi^2 ) - \left( \frac{1}{2} f \mc{A}' + \frac{1}{2} \mc{A} \bar{\Box} f - \mc{A} \bar{w} \right) \cdot \psi^2 \\
&= \frac{1}{2} \bar{\nabla}^\alpha ( \mc{A} \bar{\nabla}_\alpha f \cdot \psi^2 ) - \left( \frac{1}{2} f \mc{A}' + \bar{h} \mc{A} \right) \cdot \psi^2 \\
&= \frac{1}{2} \bar{\nabla}^\alpha ( \mc{A} \bar{\nabla}_\alpha f \cdot \psi^2 ) - \frac{1}{2} ( f \mc{A} )' \cdot \psi^2 - \frac{ \varepsilon }{ 2 } \frac{ f }{ \bar{\rho} } \mc{A} \cdot \psi^2 \text{.}
\end{align*}
Thus, letting
\begin{equation}
\label{eql.carleman_est_wp_PA} \bar{P}^A_\beta := \frac{1}{2} \mc{A} \bar{\nabla}_\beta f \cdot \psi^2 \text{,}
\end{equation}
we see from \eqref{eql.carleman_id_ptwise_1} and the above that
\begin{equation}
\label{eql.carleman_id_ptwise_2} \bar{\mc{L}} \psi \bar{S}_w \psi = \bar{\nabla}^\alpha \bar{P}^A_\alpha + \bar{\Box} \psi \bar{S}_w \psi + 2 F' \cdot | \bar{S}_w \psi |^2 - \frac{ \varepsilon f }{ \bar{\rho} } F' \cdot \psi \bar{S}_w \psi - \frac{1}{2} \left[ ( f \mc{A} )' + \frac{ \varepsilon f }{ \bar{\rho} } \mc{A} \right] \cdot \psi^2 \text{.}
\end{equation}

Consider next the stress-energy tensor for $\bar{\Box}$ applied to $\psi$,
\begin{equation}
\label{eql.carleman_est_wp_Q} \bar{Q}_{\alpha \beta} := \bar{\nabla}_\alpha \psi \bar{\nabla}_\beta \psi - \frac{1}{2} \bar{g}_{\alpha\beta} \bar{\nabla}^\mu \psi \bar{\nabla}_\mu \psi \text{,}
\end{equation}
and recall that
\[
\bar{\nabla}^\beta ( \bar{Q}_{\alpha\beta} \bar{S}^\alpha ) = \bar{\Box} \psi \bar{S} \psi + \bar{\nabla}_{\alpha \beta} f \cdot \bar{\nabla}^\alpha \psi \bar{\nabla}^\beta \psi - \frac{1}{2} \bar{\Box} f \cdot \bar{\nabla}^\mu \psi \bar{\nabla}_\mu \psi \text{.}
\]
In addition, we note that
\[
\bar{\nabla}^\beta \left( \bar{w} \cdot \psi \bar{\nabla}_\beta \psi - \frac{1}{2} \bar{\nabla}_\beta \bar{w} \cdot \psi^2 \right) = \bar{w} \cdot \psi \bar{\Box} \psi + \bar{w} \cdot \bar{\nabla}^\beta \psi \bar{\nabla}_\beta \psi - \frac{1}{2} \bar{\Box} \bar{w} \cdot \psi^2 \text{.}
\]
Consequently, defining the modified current
\begin{equation}
\label{eql.carleman_est_wp_PQ} \bar{P}^Q_\beta := \bar{Q}_{\alpha\beta} \bar{S}^\alpha + \bar{w} \cdot \psi \bar{\nabla}_\beta \psi - \frac{1}{2} \bar{\nabla}_\beta \bar{w} \cdot \psi^2 \text{,}
\end{equation}
and recalling \eqref{eq.pi_wp} and \eqref{eq.pseudoconvex_wp}, we see that
\begin{align}
\label{eql.carleman_id_ptwise_3} \bar{\nabla}^\beta \bar{P}^Q_\beta &= \bar{\Box} \psi \bar{S}_w \psi + \bar{\pi}_{\alpha\beta} \bar{\nabla}^\alpha \psi \bar{\nabla}^\beta \psi - \frac{1}{2} \bar{\Box} \bar{w} \cdot \psi^2 \\
\notag &= \bar{\Box} \psi \bar{S}_w \psi + \frac{ \varepsilon f }{ 2 \bar{\rho} } ( | T \psi |^2 + \bar{g}^{ab} \bar{\nasla}_a \psi \bar{\nasla}_b \psi - | N \psi |^2 ) - \frac{1}{2} \bar{\Box} \bar{w} \cdot \psi^2 \text{.}
\end{align}

The desired identities \eqref{eq.carleman_id_ptwise}--\eqref{eq.carleman_id_current} now follow from \eqref{eql.carleman_est_wp_PA}, \eqref{eql.carleman_id_ptwise_2}, \eqref{eql.carleman_est_wp_PQ}, and \eqref{eql.carleman_id_ptwise_3}.
\end{proof}

\subsubsection{A Pointwise Inequality}

The next step is to derive from Lemma \ref{thm.carleman_id_ptwise} a pointwise inequality, in which all the non-divergence terms have a definite sign.

\begin{lemma} \label{thm.carleman_est_ptwise}
Assume the hypotheses of Theorem \ref{thm.carleman_est_wp}, and suppose also that $\psi \in C^2 ( \bar{\Phi} ( \mc{U} \cap \mc{D} ) )$.
Then, at every point of $\bar{\Phi} ( \mc{U} \cap \mc{D} )$, we have the inequality
\begin{align}
\label{eq.carleman_est_ptwise} \frac{1}{ 4 a } f | \bar{\mc{L}} \psi |^2 + \bar{\nabla}^\beta \bar{P}_\beta &\geq \frac{ \varepsilon f }{ 2 \bar{\rho} } ( | T \psi |^2 + \bar{g}^{ab} \bar{\nasla}_a \psi \bar{\nasla}_b \psi ) + \frac{1}{4} a | \tilde{N} \psi |^2 + \frac{1}{4} b a^2 f^{- \frac{1}{2} } \cdot \psi^2 \text{,}
\end{align}
where $\tilde{N}$ denotes the operator
\begin{equation}
\label{eq.carleman_est_tilde} \tilde{N} := f^{ \frac{n - 1}{4} } N f^{ - \frac{n - 1}{4} } \text{.}
\end{equation}
\end{lemma}

\begin{proof}
Throughout the proof, we let $C$ denote positive universal constants that may change between lines.
Note that from \eqref{eq.uvf_bound}, \eqref{eq.carleman_domain_wp}, \eqref{eq.carleman_choices_wp}, and \eqref{eq.F_deriv}, we have
\begin{equation}
\label{eql.F_ineq} - F' > 0 \text{,} \qquad 1 + b f^\frac{1}{2} \simeq 1 \text{.}
\end{equation}
We begin by applying the inequality
\[
| \bar{\mc{L}} \psi | | \bar{S}_w \psi | \leq - \frac{1}{ 4 F' } | \bar{\mc{L}} \psi |^2 - F' | \bar{S}_w \psi |^2
\]
to $\psi$ and the identity \eqref{eq.carleman_id_ptwise} to obtain
\begin{align}
\label{eql.carleman_est_ptwise_1} - \frac{1}{ 4 F' } | \bar{\mc{L}} \psi |^2 + \bar{\nabla}^\beta \bar{P}_\beta &\geq - F' \cdot | \bar{S}_w \psi |^2 + \frac{ \varepsilon f }{ 2 \bar{\rho} } ( | T \psi |^2 + \bar{g}^{ab} \bar{\nasla}_a \psi \bar{\nasla}_b \psi - | N \psi |^2 ) \\
\notag &\qquad - \frac{ \varepsilon f }{ \bar{\rho} } F' \cdot \psi \bar{S}_w \psi + \frac{1}{2} \left[ ( f \mc{A} )' + \frac{ \varepsilon f }{ \bar{\rho} } \mc{A} - \bar{\Box} \bar{w} \right] \cdot \psi^2 \text{.}
\end{align}

Moreover, by \eqref{eq.f_grad_wp}, \eqref{eq.rho_est_wp}, \eqref{eq.TN}, and \eqref{eq.carleman_est_tilde}, we have that
\[
\bar{S}_w \psi = f^\frac{1}{2} \tilde{N} \psi + \frac{ (n - 2) \varepsilon f }{ 2 \bar{\rho} } \cdot \psi \geq f^\frac{1}{2} \tilde{N} \psi - \varepsilon n C f^\frac{1}{2} \cdot | \psi | \text{.}
\]
Combining the above with \eqref{eq.uvf_bound}, \eqref{eq.rho_est_wp}, \eqref{eq.w_box_wp}, \eqref{eq.carleman_choices_wp}, and \eqref{eq.F_deriv}, we see that
\begin{align*}
- \frac{ \varepsilon f }{ \bar{\rho} } F' \cdot \psi \bar{S}_w \psi &\geq - \varepsilon a ( 1 + b f^\frac{1}{2} ) ( | \psi | | \tilde{N} \psi | + \varepsilon n C \cdot \psi^2 ) \\
&\geq - a ( 1 + b f^\frac{1}{2} ) \left( \frac{1}{4} | \tilde{N} \psi |^2 + \varepsilon^2 n C \cdot \psi^2 \right) \text{,} \\
- F' \cdot | \bar{S}_w \psi |^2 &\geq a ( 1 + b f^\frac{1}{2} ) ( | \tilde{N} \psi |^2 - \varepsilon n C \cdot | \psi | | \tilde{N} \psi | - \varepsilon^2 n^2 C \cdot \psi^2 ) \text{,} \\
&\geq a ( 1 + b f^\frac{1}{2} ) \left( \frac{3}{4} | \tilde{N} \psi |^2 - \varepsilon^2 n^2 C \cdot \psi^2 \right) \text{,} \\
- \frac{ \varepsilon f }{ 2 \bar{\rho} } | N \psi |^2 &\geq - \varepsilon C f^\frac{1}{2} ( | \tilde{N} \psi |^2 + n f^{ - \frac{1}{2} } \cdot | \psi | | \tilde{N} \psi | + n^2 f^{-1} \cdot \psi^2 ) \\
&\geq - \varepsilon C ( f^\frac{1}{2} \cdot | \tilde{N} \psi |^2 + n^2 f^{- \frac{1}{2} } \cdot \psi^2 ) \text{,} \\
- \frac{1}{2} \bar{\Box} \bar{w} \cdot \psi^2 &\geq - \varepsilon n^2 C f^{ - \frac{1}{2} } \cdot \psi^2 \text{.}
\end{align*}
Note that the combined right-hand sides consist of one positive term and several error terms.

Now, abbreviating
\begin{equation}
\label{eql.carleman_est_ptwise_B} \mc{B} := - F' \cdot | \bar{S}_w \psi |^2 - \frac{ \varepsilon f }{ 2 \bar{\rho} } | N \psi |^2 - \frac{ \varepsilon f }{ \bar{\rho} } F' \cdot \psi \bar{S}_w \psi - \frac{1}{2} \bar{\Box} \bar{w} \cdot \psi^2 \text{,}
\end{equation}
we see from summing the above that
\begin{align*}
\mc{B} \geq a ( 1 + b f^\frac{1}{2} ) \left( \frac{1}{2} | \tilde{N} \psi |^2 - \varepsilon^2 n^2 C \cdot \psi^2 \right) - \varepsilon C f^\frac{1}{2} \cdot | \tilde{N} \psi |^2 - \varepsilon C n^2 f^{ - \frac{1}{2} } \cdot \psi^2 \text{.}
\end{align*}
Noting from \eqref{eq.uvf_bound} and \eqref{eq.carleman_choices_wp} that
\begin{align*}
\frac{1}{2} a ( 1 + b f^\frac{1}{2} ) - \varepsilon C f^\frac{1}{2} &\geq \frac{1}{4} a \text{,} \\
a \varepsilon^2 n^2 C \cdot \psi^2 = \varepsilon f^\frac{1}{2} \cdot a \varepsilon n^2 C \cdot f^{- \frac{1}{2} } \psi^2 &\ll a^2 \varepsilon C \cdot f^{ - \frac{1}{2} } \psi^2 \text{,}
\end{align*}
we conclude that
\begin{equation}
\label{eql.carleman_est_ptwise_2} \mc{B} \geq \frac{1}{4} a \cdot | \tilde{N} \psi |^2 - a^2 \varepsilon C \cdot f^{ - \frac{1}{2} } \psi^2 \text{.}
\end{equation}

In addition, recalling \eqref{eq.carleman_id_A}, we see that
\[
( f \mc{A} )' = \frac{1}{2} b a \left( 2 a - \frac{1}{2} \right) f^{- \frac{1}{2} } + b^2 a^2 \text{.}
\]
Using \eqref{eq.carleman_choices_wp} and \eqref{eq.carleman_id_A}, we then obtain
\begin{align}
\label{eql.carleman_est_ptwise_3} ( f \mc{A} )' + \frac{ \varepsilon f }{ \rho } \mc{A} &\geq \frac{1}{2} b a \left( 2 a - \frac{1}{2} \right) f^{- \frac{1}{2}} + b^2 a^2 - \varepsilon f^\frac{1}{2} \left[ a^2 f^{-1} + b a \left( 2 a - \frac{1}{2} \right) f^{- \frac{1}{2} } + b^2 a^2 \right] \\
\notag &\geq \frac{1}{2} b a^2 f^{- \frac{1}{2} } - a^2 \varepsilon C \cdot f^{ - \frac{1}{2} } \text{.}
\end{align}

Finally, combining \eqref{eql.carleman_est_ptwise_1}, \eqref{eql.carleman_est_ptwise_2}, and \eqref{eql.carleman_est_ptwise_3} yields
\begin{align}
\label{eql.carleman_est_ptwise_4} - \frac{1}{ 4 F' } | \bar{\mc{L}} \psi |^2 + \bar{\nabla}^\beta \bar{P}_\beta &\geq \frac{1}{4} a \cdot | \tilde{N} \psi |^2 + \frac{ \varepsilon f }{ 2 \bar{\rho} } ( | T \psi |^2 + \bar{g}^{ab} \bar{\nasla}_a \psi \bar{\nasla}_b \psi - | N \psi |^2 ) \\
\notag &\qquad + \left( \frac{1}{2} b - \varepsilon C \right) a^2 f^{ - \frac{1}{2} } \cdot \psi^2 \text{.}
\end{align}
The desired estimate \eqref{eq.carleman_est_ptwise} now follows from \eqref{eql.carleman_est_ptwise_4}, the observation
\[
- \frac{1}{ 4 F' } \leq \frac{ f }{ 4 a ( 1 + b f^\frac{1}{2} ) } \leq \frac{ f }{ 4 a } \text{,}
\] 
and the assumption $\varepsilon \ll b$ from \eqref{eq.carleman_choices_wp}.
\end{proof}

\subsubsection{Reversal of Conjugation}

We now establish a pointwise inequality for the original wave operator by undoing the conjugation process from Definition \ref{def.F_conj}.

\begin{lemma} \label{thm.carleman_est_rev}
Assume the hypotheses of Theorem \ref{thm.carleman_est_wp}, and let $\phi \in C^2 ( \bar{\Phi} ( \mc{U} \cap \mc{D} ) ) \cap C^1 ( \bar{\Phi} ( \bar{\mc{U}} \cap \mc{D} ) )$.
Then, there exists $C > 0$ such that on any point of $\bar{\Phi} ( \mc{U} \cap \mc{D} )$, we have the pointwise inequality
\begin{align}
\label{eq.carleman_est_rev} \frac{1}{ 4 a } f e^{ -2 F } | \bar{\Box} \phi |^2 + \bar{\nabla}^\beta \bar{P}^\star_\beta &\geq \frac{ \varepsilon }{ 16 \bar{\rho} } e^{ -2 F } ( | u \cdot \partial_u \phi |^2 + | v \cdot \partial_v \phi |^2 + f \cdot \bar{g}^{ab} \bar{\nasla}_a \phi \bar{\nasla}_b \phi ) \\
\notag &\qquad + \frac{1}{8} b a^2 f^{- \frac{1}{2} } e^{ - 2 F } \cdot \phi^2 \text{,}
\end{align}
where $\bar{P}^\star := \bar{P}^\star [ \phi ]$ is the $1$-form on $\bar{\Phi} ( \bar{\mc{U}} \cap \mc{D} )$ given by
\begin{align}
\label{eq.carleman_rev_current} \bar{P}^\ast_\beta &:= \bar{S} ( e^{ - F } \phi ) \bar{\nabla}_\beta ( e^{ - F } \phi ) - \frac{1}{2} \bar{\nabla}_\beta f \cdot \bar{\nabla}^\mu ( e^{ - F } \phi ) \bar{\nabla}_\mu ( e^{ - F } \phi ) + \bar{w} \cdot e^{ - F } \phi \bar{\nabla}_\beta ( e^{ - F } \phi ) \\
\notag &\qquad + \frac{1}{2} ( \mc{A} \bar{\nabla}_\beta f - \bar{\nabla}_\beta \bar{w} ) \cdot e^{ - 2 F } \phi^2 \text{.}
\end{align}
Furthermore, $\bar{P}^\star$ satisfies the following:
\begin{itemize}
\item If \eqref{eq.carleman_dirichlet_wp} also holds for $\phi$, then we have on $\bar{\Phi} ( \partial \mc{U} \cap \mc{D} )$ that
\begin{equation}
\label{eq.carleman_rev_dirichlet} \bar{P}^\ast ( \bar{\mc{N}} ) |_{ \bar{\Phi} ( \partial \mc{U} \cap \mc{D} ) } = \frac{1}{2} \zeta_{a, b} \cdot \bar{\mc{N}} f | \bar{\mc{N}} \phi |^2 |_{ \bar{\Phi} ( \partial \mc{U} \cap \mc{D} ) } \text{.}
\end{equation}

\item $\bar{P}^\star$ satisfies the following estimate on $\bar{\Phi} ( \bar{\mc{U}} \cap \mc{D} )$:
\begin{equation}
\label{eq.carleman_rev_currest} | \bar{P}^\star ( \bar{S} ) | \lesssim R^2 e^a f^{2 a} ( | \partial_u \phi |^2 + | \partial_v \phi |^2 + \bar{g}^{ab} \bar{\nasla}_a \phi \bar{\nasla}_b \phi + a^2 f^{-1} \phi^2 ) \text{.}
\end{equation}
\end{itemize}
\end{lemma}

\begin{proof}
We begin with the properties for $\bar{P}^\star$.
If \eqref{eq.carleman_dirichlet_wp} holds, then \eqref{eq.carleman_rev_current} simplifies to
\[
\bar{P}^\star ( \bar{\mc{N}} ) |_{ \bar{\Phi} ( \partial \mc{U} \cap \mc{D} ) } = e^{ - 2 F } \left. \left( \bar{S} \phi \bar{\mc{N}} \phi - \frac{1}{2} \bar{\mc{N}} f \cdot | \bar{\mc{N}} \phi |^2 \right) \right|_{ \bar{\Phi} ( \partial \mc{U} \cap \mc{D} ) } \text{.}
\]
Since $\bar{S} - \bar{\mc{N}} f \cdot \bar{\mc{N}}$ is precisely the $\bar{g}$-orthogonal projection of $\bar{S}$ onto $\bar{\Phi} ( \partial \mc{U} \cap \mc{D} )$, we have
\[
\bar{S} \phi |_{ \bar{\Phi} ( \partial \mc{U} \cap \mc{D} ) } = \bar{\mc{N}} f \cdot \bar{\mc{N}} \phi |_{ \bar{\Phi} ( \partial \mc{U} \cap \mc{D} ) } \text{.}
\]
Combining the above and \eqref{eq.F_deriv} results in \eqref{eq.carleman_rev_dirichlet}.

Next, we observe from \eqref{eq.f}, \eqref{eq.carleman_choices_wp}, and \eqref{eq.F_deriv} that
\[
| \partial_u ( e^{ - F } \phi ) | \lesssim e^{ - F } ( | \partial_u \phi | + a f^{-1} | v | | \phi | ) \text{,} \qquad | \partial_v ( e^{ - F } \phi ) | \lesssim e^{ - F } ( | \partial_u \phi | + a f^{-1} | u | | \phi | ) \text{.}
\]
Using this, along with \eqref{eq.f_grad_wp}, \eqref{eq.uvf_bound}, \eqref{eq.carleman_domain_wp}, and \eqref{eq.F_conj}, we estimate
\begin{align*}
| \bar{S} ( e^{ - F } \phi ) \bar{S} ( e^{ - F } \phi ) | &\lesssim e^{ - 2 F } ( | u | | \partial_u \phi | + | v | | \partial_v \phi | + a f^{-1} | u v | | \phi | )^2 \\
&\lesssim R^2 e^{ - 2 F } ( | \partial_u \phi |^2 + | \partial_v \phi |^2 + a^2 f^{-1} \phi^2 ) \text{,} \\
| \bar{S} f \cdot \bar{\nabla}^\mu ( e^{ - F } \phi ) \bar{\nabla}_\mu ( e^{ - F } \phi ) | &\lesssim e^{ - 2 F } f [ | \partial_u \phi |^2 + | \partial_v \phi |^2 + \bar{g}^{ab} \bar{\nasla}_a \phi \bar{\nasla}_b \phi + a^2 f^{-2} ( | u |^2 + | v |^2 ) \phi^2 ] \\
&\lesssim R^2 e^{ - 2 F } ( | \partial_u \phi |^2 + | \partial_v \phi |^2 + \bar{g}^{ab} \bar{\nasla}_a \phi \bar{\nasla}_b \phi + a^2 f^{-1} \phi^2 ) \text{,} \\
| \bar{w} \cdot e^{ - F } \phi \cdot \bar{S} ( e^{ - F } \phi ) | &\lesssim e^{ - F } | \phi | ( | u | | \partial_u \phi | + | v | | \partial_v \phi | + a f^{-1} | u v | | \phi | ) \\
&\lesssim R^2 e^{ - 2 F } ( | \partial_u \phi |^2 + | \partial_v \phi |^2 + a^2 f^{-1} \phi^2 ) \text{.}
\end{align*}
Recalling in addition \eqref{eq.f_rho_deriv_wp}, \eqref{eq.w_wp}, \eqref{eq.carleman_choices_wp}, and \eqref{eq.carleman_id_A} yields
\begin{align*}
| \mc{A} \cdot \bar{S} f | &\lesssim a^2 + b a^2 f^\frac{1}{2} + b^2 a^2 f \lesssim a^2 \text{,} \\
| \bar{S} \bar{w} | &\lesssim \frac{ \varepsilon n f ( | u | + | v | ) }{ \bar{\rho}^2 } \lesssim \varepsilon n R \lesssim a^2 \text{,}
\end{align*}
from which we then conclude
\[
| ( \mc{A} \bar{\nabla}_\beta f + \bar{\nabla}_\beta \bar{w} ) \cdot e^{ - 2 F } \phi^2 | \lesssim a^2 e^{ - 2 F } \phi^2 \lesssim R^2 e^{ -2 F } \cdot a^2 f^{-1} \phi^2 \text{.}
\]
Combining all the above and recalling \eqref{eq.carleman_rev_current} yields
\[
| \bar{P}^\star ( \bar{S} ) | \lesssim R^2 e^{ - 2 F } ( | \partial_u \phi |^2 + | \partial_v \phi |^2 + \bar{g}^{ab} \bar{\nasla}_a \phi \bar{\nasla}_b \phi + a^2 f^{-1} \phi^2 ) \text{.}
\]
Estimating $e^{ - 2 F }$ using \eqref{eq.carleman_choices_wp} and \eqref{eq.F} now results in \eqref{eq.carleman_rev_currest}.

Finally, for \eqref{eq.carleman_est_rev}, we set $\psi := e^{-F} \phi$ and apply \eqref{eq.F_conj} and \eqref{eq.carleman_est_ptwise} to this $\psi$, which yields
\begin{align}
\label{eql.carleman_est_rev_1} \frac{1}{ 4 a } f e^{ -2 F } | \bar{\Box} \phi |^2 + \bar{\nabla}^\beta \bar{P}^\star_\beta &\geq \frac{ \varepsilon f }{ 2 \bar{\rho} } e^{ -2 F } ( | T \phi |^2 + \bar{g}^{ab} \bar{\nasla}_a \phi \bar{\nasla}_b \phi ) + \frac{1}{4} a | \tilde{N} ( e^{ - F } \phi ) |^2 \\
\notag &\qquad + \frac{1}{4} b a^2 f^{- \frac{1}{2} } e^{ - 2 F } \cdot \phi^2 \text{,}
\end{align}
Applying \eqref{eq.f_grad_wp}, \eqref{eq.TN}, \eqref{eq.carleman_choices_wp}, \eqref{eq.F_deriv}, and \eqref{eq.carleman_est_tilde}, we see that
\begin{align*}
e^{-2 F} | N \phi |^2 &\lesssim | \tilde{N} ( e^{-F} \phi ) |^2 + a^2 e^{-2 F} f^{-1} \cdot \phi^2 \text{,} \\
\varepsilon e^{-2 F} f^\frac{1}{2} | N \phi |^2 &\ll a | \tilde{N} ( e^{-F} \phi ) |^2 + b a^2 e^{-2 F} f^{ - \frac{1}{2} } \cdot \phi^2 \text{,}
\end{align*}
From the above, we conclude
\begin{align}
\label{eql.carleman_est_rev_2} \frac{1}{ 4 a } f e^{ -2 F } | \bar{\Box} \phi |^2 + \bar{\nabla}^\beta \bar{P}^\star_\beta &\geq \frac{ \varepsilon f }{ 2 \bar{\rho} } e^{ -2 F } ( | T \phi |^2 + \bar{g}^{ab} \bar{\nasla}_a \phi \bar{\nasla}_b \phi ) + \varepsilon f^\frac{1}{2} e^{ - 2 F } | N \phi |^2 \\
\notag &\qquad + \frac{1}{8} b a^2 f^{- \frac{1}{2} } e^{ - 2 F } \cdot \phi^2 \text{.}
\end{align}
We also note from \eqref{eq.f} and \eqref{eq.TN} that
\begin{align*}
| u \partial_u \phi |^2 + | v \partial_v \phi |^2 &\leq 4 f ( | T \phi | + | N \phi | )^2 + 4 f ( | T \phi | + | N \phi | )^2 \\
&\leq 8 f ( | T \phi |^2 + | N \phi |^2 ) \text{,}
\end{align*}
hence \eqref{eq.rho_est_wp} implies
\begin{align*}
\frac{ \varepsilon f }{ 2 \bar{\rho} } ( | T \phi |^2 + \bar{g}^{ab} \bar{\nasla}_a \phi \bar{\nasla}_b \phi ) + \varepsilon f^\frac{1}{2} | N \phi |^2 &\geq \frac{ \varepsilon f }{ 2 \bar{\rho} } ( | T \phi |^2 + \bar{g}^{ab} \bar{\nasla}_a \phi \bar{\nasla}_b \phi + | N \phi |^2 ) \\
&\geq \frac{ \varepsilon }{ 16 \bar{\rho} } ( | u \cdot \partial_u \phi |^2 + | v \cdot \partial_v \phi |^2 + f \bar{g}^{ab} \bar{\nasla}_a \phi \bar{\nasla}_b \phi ) \text{.}
\end{align*}
Combining the above with \eqref{eql.carleman_est_rev_2} now results in \eqref{eq.carleman_est_rev}.
\end{proof}

\subsubsection{The Integral Estimate}

In order to complete the proof of Theorem \ref{thm.carleman_est_wp}, we integrate \eqref{eq.carleman_est_rev} over $\bar{\Phi} ( \mc{U} \cap \mc{D} )$.
However, here we must be especially careful with boundary terms.

Let $\phi$ be as in the statement of Theorem \ref{thm.carleman_est_wp}.
For sufficiently small $\delta > 0$, we define the sets
\[
\mc{G}_\delta := \bar{\Phi} ( \mc{U} \cap \mc{D} ) \cap \{ f > \delta \} \text{,} \qquad \mc{H}_\delta := \bar{\Phi} ( \mc{U} \cap \mc{D} ) \cap \{ f = \delta \} \text{,}
\]
on which $r, \bar{\rho}, f$ are all bounded from above and below by positive constants; see \eqref{eq.uvf_bound}, \eqref{eq.rho_est_wp}, and \eqref{eq.carleman_domain_wp}.
Integrating \eqref{eq.carleman_est_rev}, with the above $\phi$, over $\mc{G}_\delta$ and recalling \eqref{eq.F_deriv}, we obtain 
\begin{align*}
&\frac{1}{ 4 a } \int_{ \mc{G}_\delta } \zeta_{a, b} f | \bar{\Box} \phi |^2 + \int_{ \mc{G}_\delta } \bar{\nabla}^\alpha \bar{P}^\star_\alpha \\
&\quad \geq \frac{ \varepsilon }{ 16 } \int_{ \mc{G}_\delta } \zeta_{a, b} \bar{\rho}^{-1} ( | u \cdot \partial_u \phi |^2 + | v \cdot \partial_v \phi |^2 + f \bar{g}^{ab} \bar{\nasla}_a \phi \bar{\nasla}_b \phi ) + \frac{ b a^2 }{ 8 } \int_{ \mc{G}_\delta } \zeta_{a, b} f^{- \frac{1}{2} } \cdot \phi^2 \text{.}
\end{align*}
Taking a limit as $\delta \searrow 0$ and applying the monotone covergence theorem yields
\begin{align}
\label{eql.carleman_est_wp_1} &\frac{1}{ 4 a } \int_{ \bar{\Phi} ( \mc{U} \cap \mc{D} ) } \zeta_{a, b} f | \bar{\Box} \phi |^2 + \lim_{ \delta \searrow 0 } \int_{ \mc{G}_\delta } \bar{\nabla}^\alpha \bar{P}^\star_\alpha \\
\notag &\quad \geq \frac{ \varepsilon }{ 16 } \int_{ \bar{\Phi} ( \mc{U} \cap \mc{D} ) } \zeta_{a, b} \bar{\rho}^{-1} ( | u \partial_u \phi |^2 + | v \partial_v \phi |^2 + f \bar{g}^{ab} \bar{\nasla}_a \phi \bar{\nasla}_b \phi ) + \frac{ b a^2 }{ 8 } \int_{ \bar{\Phi} ( \mc{U} \cap \mc{D} ) } \zeta_{a, b} f^{- \frac{1}{2} } \phi^2 \text{.}
\end{align}
provided the limit on the left-hand side exists.

Next, we apply the divergence theorem to the term with the limit.
Noting that $\{ f = \delta \}$ is a timelike hypersurface of $( \mc{D}, \bar{g} )$, with outward (from $\{ f > \delta \}$) unit normal $- f^{ - \frac{1}{2} } \bar{S}$, we obtain
\begin{align}
\label{eql.carleman_est_wp_2} \lim_{ \delta \searrow 0 } \int_{ \mc{G}_\delta } \bar{\nabla}^\alpha \bar{P}^\star_\alpha &= \lim_{ \delta \searrow 0 } \int_{ \bar{\Phi} ( \partial \mc{U} \cap \mc{D} ) \cap \{ f > \delta \} } \bar{P}^\star ( \bar{\mc{N}} ) - \lim_{ \delta \searrow 0 } \int_{ \mc{H}_\delta } f^{ - \frac{1}{2} } \bar{P}^\star ( \bar{S} ) \\
\notag &= \frac{1}{2} \int_{ \bar{\Phi} ( \partial \mc{U} \cap \mc{D} ) } \zeta_{a, b} \bar{\mc{N}} f \cdot | \bar{\mc{N}} \phi |^2 - \lim_{ \delta \searrow 0 } \delta^{ - \frac{1}{2} } \int_{ \mc{H}_\delta } \bar{P}^\star ( \bar{S} ) \text{,}
\end{align}
where in the last step, we applied \eqref{eq.carleman_rev_dirichlet} and the fact that $\phi$ is uniformly $C^1$-bounded.
Moreover, applying the estimate \eqref{eq.carleman_rev_currest}, the remaining limit can be bounded as
\begin{align}
\label{eql.carleman_est_wp_3} \left| \lim_{ \delta \searrow 0 } \delta^{ - \frac{1}{2} } \int_{ \mc{H}_\delta } \bar{P}^\star ( \bar{S} ) \right| &\lesssim R^2 e^a \lim_{ \delta \searrow 0 } \delta^{ 2 a - \frac{1}{2} } \int_{ \mc{H}_\delta } ( | \partial_u \phi |^2 + | \partial_v \phi |^2 + \bar{g}^{ab} \bar{\nasla}_a \phi \bar{\nasla}_b \phi ) \\
\notag &\qquad + R^2 e^a \lim_{ \delta \searrow 0 } \delta^{ 2 a - \frac{1}{2} } \int_{ \mc{H}_\delta } a^2 f^{-1} \phi^2 \\
\notag &\lesssim C \lim_{ \delta \searrow 0 } \delta^{ 2 a - \frac{3}{2} } \int_{ \mc{H}_\delta } 1 \text{,}
\end{align}
where $C > 0$ is some constant depending on both $\phi$ and $R$.

To control this limit, we foliate $\mc{H}_\delta$ using level sets of $t$ and apply the coarea formula.
Let $\bar{D} t$ denote the $\bar{g}$-gradient of $t$ on $\mc{H}_\delta$, with respect to the induced metric.
Note that $\bar{D} t$ and $T$ must point in the same direction, since both vector fields are tangent to $\mc{H}_\delta$ and are normal to the level sets of $( u, v )$.
As a result, by \eqref{eq.uv} and \eqref{eq.TN}, we see that
\[
| \bar{g} ( \bar{D} t, \bar{D} t ) |^\frac{1}{2} = | g ( \bar{D} t, T ) | = \frac{1}{2} f^{ - \frac{1}{2} } ( - u \partial_u t + v \partial_v t ) = \frac{1}{2} f^{ - \frac{1}{2} } r \text{.}
\]
Moreover, noting from \eqref{eq.f} that
\[
r^{-1} |_{ \mc{H}_\delta } = ( t^2 + 4 \delta )^{ - \frac{1}{2} } |_{ \mc{H}_\delta } \lesssim \delta^{ - \frac{1}{2} } \text{,}
\]
then the coarea formula, along with \eqref{eq.uvf_bound} and \eqref{eq.carleman_domain_wp}, yields
\begin{align}
\label{eql.carleman_est_wp_4} \lim_{ \delta \searrow 0 } \delta^{ 2 a - \frac{3}{2} } \int_{ \mc{H}_\delta } 1 &\lesssim \lim_{ \delta \searrow 0 } \delta^{ 2 a - \frac{3}{2} } \int_{ - R }^R \int_{ \Sph^{n - 1} } | \bar{g} ( \bar{D} t, \bar{D} t ) |^{ -\frac{1}{2} } \bar{\rho}^{n - 1} |_{ ( t, r = \sqrt{ t^2 + 4 \delta }, \omega ) } d \mathring{\gamma}_\omega d t \\
\notag &\lesssim R^{n - 1} \lim_{ \delta \searrow 0 } \delta^{ 2 a - 1 } \int_{ - R }^R r^{-1} |_{ ( t, r = \sqrt{ t^2 + 4 \delta }, \omega ) } dt \\
\notag &\lesssim R^n \lim_{ \delta \searrow 0 } \delta^{ 2 a - \frac{3}{2} } \\
\notag &= 0 \text{,}
\end{align}
where we recalled in the last step that $2 a > \frac{3}{2}$ by \eqref{eq.carleman_choices_wp}.

Finally, combining \eqref{eql.carleman_est_wp_1}--\eqref{eql.carleman_est_wp_4} yields \eqref{eq.carleman_est_wp} and completes the proof of Theorem \ref{thm.carleman_est_wp}.

\section{Observability Estimates} \label{sec.obs}

In this section, we apply the Carleman estimates of Theorem \ref{thm.carleman_est} to prove our two most general observability estimates for solutions of linear wave equations on GTCs:
\begin{enumerate}
\item The \emph{exterior estimate} (stated in Theorem \ref{thm.obs_ext}), where the reference point of the corresponding Carleman estimate (represented as $P$ in Theorem \ref{thm.carleman_est}) lies outside the GTC.

\item The \emph{interior estimate} (stated in Theorem \ref{thm.obs_int}), where the reference point of the corresponding Carleman estimate lies within the GTC.\footnote{Interior observability bounds have previously been obtained from multiplier, but not Carleman, methods.}
\end{enumerate}
The exterior and interior estimates are discussed in Sections \ref{sec.obs_ext} and \ref{sec.obs_int}, respectively.
Some corollaries of these estimates, such as Theorem \ref{thm.intro_obs_main}, are discussed in Section \ref{sec.obs_cor}.

\subsection{Exterior Observability} \label{sec.obs_ext}

We now discuss the first of our general observability inequalities, based on Carleman estimates centered outside the given GTC.

\begin{theorem} \label{thm.obs_ext}
Let $\mc{U}$ be a GTC in $\R^{1+n}$, and fix $P \in \R^{1+n} \setminus \bar{\mc{U}}$ and $0 < \delta \ll 1$.
In addition, assume $\mc{U} \cap \mc{D}_P$ is bounded, and consider the setting of Problem \ref{prb.linear_wave}.
Moreover:
\begin{itemize}
\item Let $\mc{X}, V$ be as in \eqref{eq.XV}, and define the constants
\begin{align}
\label{eq.obs_ext_MR} M_0 := \sup_{ \mc{U} \cap \mc{D}_P } | V | \text{,} &\qquad M_1 := \sup_{ \mc{U} \cap \mc{D}_P } | \mc{X}^{ t, x } | \text{,} \\
\notag R_+ := \sup_{ \mc{U} \cap \mc{D}_P } r_P \text{,} &\qquad R_- := \inf_{ \mc{U}_{ t (P) } } r_P \text{.}
\end{align}

\item Let $\mc{N}$ be the outward-pointing ($g$-)unit normal to $\mc{U}$, and let
\begin{equation}
\label{eq.obs_ext_delta} \mc{S} := \left( 1 - \frac{ \delta^2 r_P }{ R_+ } \right) \mc{N} f_P + \frac{ \delta^2 f_P }{ R_+ } \mc{N} r_P \text{} \qquad \Gamma_+ := \partial \mc{U} \cap \mc{D}_P \cap \{ \mc{S} > 0 \} \text{.}
\end{equation}
\end{itemize}
Then, there exist $C, N > 0$, depending on $\mc{U}$, such that the observability estimate,
\begin{equation}
\label{eq.obs_ext_est} \int_{ \mc{U}_{ t (P) } } ( | \nabla_{ t, x } \phi |^2 + \phi^2 ) \leq \frac{ C }{ \delta^2 R_- } \left( \frac{ 16 R_+ }{ R_- } \right)^{ N \left( n + R_+ + \frac{ R_+^\frac{4}{3} M_0^\frac{2}{3} }{ \delta^\frac{1}{3} } + \frac{ R_+^4 M_1^2 }{ \delta^2 R_-^2 } \right) } \int_{ \Gamma_+ } | \mc{S} | | \mc{N} \phi |^2 \text{,}
\end{equation}
holds for any solution $\phi \in C^2 ( \mc{U} ) \cap C^1 ( \bar{\mc{U}} )$ of \eqref{eq.linear_wave} that also satisfies $\phi |_{ \partial \mc{U} \cap \mc{D}_P } = 0$.
\end{theorem}

\begin{remark}
Note that the assumption that $\mc{U} \cap \mc{D}_P$ is bounded in the statement of Theorem \ref{thm.obs_ext} implies that all of the constants $M_0, M_1, R_\pm$ in \eqref{eq.obs_ext_MR} are finite.
\end{remark}

\begin{remark}
Note that the dependence of \eqref{eq.obs_ext_est} on $M_0$ and $M_1$ matches those found in \cite{duy_zhang_zua:obs_opt}.
In particular, the dependence on $M_0$ is known to be sharp; see the discussions in \cite{duy_zhang_zua:obs_opt}.
\end{remark}

\begin{remark}
Since $r_P \neq 0$ on $\mc{D}_P$, the quantity $\mc{S}$ in \eqref{eq.obs_ext_delta} is well-defined and smooth.
\end{remark}

In principle, the proof of Theorem \ref{thm.obs_ext} is similar to the standard process of obtaining observability from Carleman estimates.
The primary new technical difficulty is that the pseudoconvexity and the Carleman weight in \eqref{eq.carleman_est} degenerate toward the null cone centered about the point $P$.
As a result, one must be considerably more careful when absorbing terms (mainly, those involving $\nabla \phi$) in the Carleman estimate.
In particular, various constants must be tracked more carefully.

The remainder of this section is dedicated to the proof of Theorem \ref{thm.obs_ext}.

\subsubsection{Application of the Carleman Estimate}

Since $\mc{U} \cap \mc{D}_P$ is bounded, it follows from \eqref{eq.obs_ext_MR} that \eqref{eq.carleman_domain} holds, with $R := R_+$.
In addition, choose $a \geq n$ large enough so it also satisfies
\begin{equation}
\label{eql.obs_ext_a} a \gg_{ \mc{U} } R_+ \text{,} \qquad a \gg_{ \mc{U} } \delta^{ - \frac{1}{3} } R_+^\frac{4}{3} M_0^\frac{2}{3} \text{,} \qquad a \gg_{ \mc{U} } \delta^{-2} R_-^{-2} R_+^4 M_1^2 \text{,}
\end{equation}
and choose $\varepsilon$ and $b$ as follows (note in particular that \eqref{eq.carleman_choices} holds, with $R_+$ in place of $R$):
\begin{equation}
\label{eql.obs_ext_b} \varepsilon := \delta^2 R_+^{-1} \text{,} \qquad b := \delta R_+^{-1} \text{.}
\end{equation}
The key step is to apply the Carleman estimate \eqref{eq.carleman_est}, with the above values of $a, b, \varepsilon$, to our given $\mc{U}$ and $P$.
From \eqref{eq.carleman_est}, as well as \eqref{eq.linear_wave}, \eqref{eq.uvf_bound}, \eqref{eq.obs_ext_delta}, and \eqref{eql.obs_ext_b}, we see that
\begin{align}
\label{eql.obs_ext_1} &\frac{2}{a} \int_{ \mc{U} \cap \mc{D}_P } \zeta^P_{ a, b; \varepsilon } f_P | \nabla_{ \mc{X} } \phi |^2 + \frac{2}{a} \int_{ \mc{U} \cap \mc{D}_P } \zeta^P_{ a, b; \varepsilon } f_P V^2 \cdot \phi^2 + \int_{ \partial \mc{U} \cap \mc{D}_P } \zeta^P_{ a, b; \varepsilon } \mc{S} \cdot | \mc{N} \phi |^2 \\
\notag &\quad \geq \frac{ C \delta^2 }{ R_+^2 } \int_{ \mc{U} \cap \mc{D}_P } \zeta^P_{ a, b; \varepsilon } ( | u_P \partial_{ u_P } \phi |^2 + | v_P \partial_{ v_P } \phi |^2 + f_P g^{ a b } \nasla^P_a \phi \nasla^P_b ) + \frac{ C \delta a^2 }{ R_+^2 } \int_{ \mc{U} \cap \mc{D}_P } \zeta^P_{ a, b; \varepsilon } \phi^2 \text{.}
\end{align}
Here, all notations are as in the statement of Theorem \ref{thm.carleman_est}.

Let $I_1$, $I_0$, $I_\Gamma$ denote the first, second, and third terms in the left-hand side of \eqref{eql.obs_ext_1}, respectively.
Recalling \eqref{eq.uvf_bound} and our conditions \eqref{eql.obs_ext_a} for $a$, we see that
\[
\frac{1}{a} \int_{ \mc{U} \cap \mc{D}_P } \zeta^P_{ a, b; \varepsilon } f_P V^2 \cdot \phi^2 \leq \frac{ R_+^2 M_0^2 }{ a } \int_{ \mc{U} \cap \mc{D}_P } \zeta^P_{ a, b; \varepsilon } \cdot \phi^2 \ll_{ \mc{U} } \frac{ \delta a^2 }{ R_+^2 } \int_{ \mc{U} \cap \mc{D}_P } \zeta^P_{ a, b; \varepsilon } \cdot \phi^2 \text{.}
\]
As a result, $I_0$ can be absorbed into the right-hand side of \eqref{eql.obs_ext_1}:
\begin{equation}
\label{eql.obs_ext_2} I_1 + I_\Gamma \geq C \int_{ \mc{U} \cap \mc{D}_P } \zeta^P_{ a, b; \varepsilon } \left[ \frac{ \delta^2 }{ R_+^2 } ( | u_P \partial_{ u_P } \phi |^2 + | v_P \partial_{ v_P } \phi |^2 + f_P g^{ab} \nasla^P_a \phi \nasla^P_b \phi ) + \frac{ \delta a^2 }{ R_+^2 } \phi^2 \right] \text{.}
\end{equation}

We now partition $\mc{U} \cap \mc{D}_P$ into
\begin{align}
\label{eql.obs_ext_Udec} \mc{U}_\leq &:= \mc{U} \cap \mc{D}_P \cap \left\{ \frac{ f_P }{ ( 1 + \varepsilon u_P ) ( 1 - \varepsilon v_P ) } \leq \frac{ R_-^2 }{ 64 } \right\} \text{,} \\
\notag \mc{U}_> &:= \mc{U} \cap \mc{D}_P \cap \left\{ \frac{ f_P }{ ( 1 + \varepsilon u_P ) ( 1 - \varepsilon v_P ) } > \frac{ R_-^2 }{ 64 } \right\} \text{,}
\end{align}
Note that on $\mc{U}_>$, we have from \eqref{eq.uvf_bound}, \eqref{eq.conf_comp}, and \eqref{eql.obs_ext_Udec} that
\[
v_P = \frac{ f_P }{ - u_P } \gtrsim \frac{ R_-^2 }{ R_+ } \text{,} \qquad - u_P = \frac{ f_P }{ v_P } \gtrsim \frac{ R_-^2 }{ R_+ } \text{.}
\]
Thus, shrinking domains from $\mc{U} \cap \mc{D}_P$ to $\mc{U}_>$, we see from \eqref{eql.obs_ext_2} that
\begin{equation}
\label{eql.obs_ext_3} I_1 + I_\Gamma \geq C \int_{ \mc{U}_> } \zeta^P_{ a, b; \varepsilon } \left[ \frac{ \delta^2 R_-^2 }{ R_+^3 } ( - u_P | \partial_{ u_P } \phi |^2 + v_P | \partial_{ v_P } \phi |^2 + v_P g^{ab} \nasla^P_a \phi \nasla^P_b \phi ) + \frac{ \delta a^2 }{ R_+^2 } \phi^2 \right] \text{.}
\end{equation}

We now decompose $I_1$ as
\begin{equation}
\label{eql.obs_ext_I1dec} I_1 = \frac{2}{a} \int_{ \mc{U}_\leq } \zeta^P_{ a, b; \varepsilon } f_P | \nabla_{ \mc{X} } \phi |^2 + \frac{2}{a} \int_{ \mc{U}_> } \zeta^P_{ a, b; \varepsilon } f_P | \nabla_{ \mc{X} } \phi |^2 := I_{1, \leq} + I_{1, >} \text{.}
\end{equation}
From \eqref{eq.uvf_bound}, \eqref{eq.obs_ext_MR}, and \eqref{eql.obs_ext_a}, we obtain
\begin{align*}
I_{1, >} &\leq \frac{ R_+ M_1^2 }{ a } \int_{ \mc{U}_> } \zeta^P_{ a, b; \varepsilon } ( - u_P | \partial_{ u_P } \phi |^2 + v_P | \partial_{ v_P } \phi |^2 + v_P g^{ab} \nasla^P_a \phi \nasla^P_b \phi ) \\
&\ll_{ \mc{U} } \frac{ \delta^2 R_-^2 }{ R_+^3 } \int_{ \mc{U}_> } \zeta^P_{ a, b; \varepsilon } ( - u_P | \partial_{ u_P } \phi |^2 + v_P | \partial_{ v_P } \phi |^2 + v_P g^{ab} \nasla^P_a \phi \nasla^P_b \phi ) \text{.}
\end{align*}
Thus, $I_{1, >}$ can be absorbed into the right-hand side of \eqref{eql.obs_ext_3}, and we have
\begin{equation}
\label{eql.obs_ext_4} I_{1, \leq} + I_\Gamma \geq C \int_{ \mc{U}_> } \zeta^P_{ a, b; \varepsilon } \left[ \frac{ \delta^2 R_-^2 }{ R_+^3 } ( - u_P | \partial_{ u_P } \phi |^2 + v_P | \partial_{ v_P } \phi |^2 + v_P g^{ab} \nasla^P_a \phi \nasla^P_b \phi ) + \frac{ \delta a^2 }{ R_+^2 } \phi^2 \right] \text{.}
\end{equation}

\subsubsection{Applications of Energy Estimates}

Since $\partial \mc{U}$ is timelike, then by \eqref{eq.uv} and \eqref{eq.obs_ext_MR}, we have
\begin{equation}
\label{eql.obs_ext_50} r_P |_{ \mc{U}_\tau } \geq \frac{ 3 R_- }{4} \text{,} \qquad -u_P |_{ \mc{U}_\tau } \geq \frac{ R_- }{4} \text{,} \qquad v_P |_{ \mc{U}_\tau } \geq \frac{ R_- }{4} \text{,}
\end{equation}
for any $\tau \in \R$ with $| \tau - t ( P ) | \leq \frac{1}{4} R_-$.
As a result, by \eqref{eq.f}, \eqref{eq.carleman_weight}, and \eqref{eql.obs_ext_50}, we have that
\[
\left. \frac{ f_P }{ ( 1 + \varepsilon u_P ) ( 1 - \varepsilon v_P ) } \right|_{ \mc{U}_\tau } \geq \frac{ R_-^2 }{ 16 } \text{,} \qquad \zeta_{ a, b; \varepsilon } |_{ \mc{U}_\tau } \geq \left( \frac{ R_- }{4} e^\frac{ b R_- }{4} \right)^{ 4 a } \text{,}
\]
for the same $\tau$ as above.
In particular, we see that
\[
\mc{U} \cap \left\{  - \frac{ R_- }{4} < t_P < \frac{ R_- }{4} \right\} \subseteq \mc{U}_> \text{,}
\]
so that by shrinking domains, applying Fubini's theorem, and recalling \eqref{eql.obs_ext_50} again, we obtain
\begin{equation}
\label{eql.obs_ext_5} I_{1, \leq} + I_\Gamma \geq C \left( \frac{ R_- }{ 4 } e^\frac{ b R_- }{4} \right)^{4 a} \int_{ t ( P ) - \frac{ R_- }{4} }^{ t ( P ) + \frac{ R_- }{4} } \int_{ \mc{U}_\tau } \left[ \frac{ \delta^2 R_-^3 }{ R_+^3 } \cdot | \nabla_{ t, x } \phi |^2 + \frac{ \delta a^2 }{ R_+^2 } \cdot \phi^2 \right] d \tau \text{.}
\end{equation}

Since \eqref{eq.obs_ext_MR} and \eqref{eql.obs_ext_a} imply that $a \gg R_+ \geq R_-$, we see that
\[
\frac{ \delta a^2 }{ R_+^2 } \geq \frac{ \delta^2 R_-^2 }{ R_+^2 } \geq \frac{ \delta^2 R_-^3 }{ R_+^3 } \text{,} \qquad \frac{ \delta a^2 }{ R_+^2 } \geq \frac{ \delta a^\frac{1}{2} ( \delta^{ - \frac{1}{3} } R_+^\frac{4}{3} M_0^\frac{2}{3} )^\frac{3}{2} }{ R_+^2 } = \delta^\frac{1}{2} a^\frac{1}{2} M_0 \geq \frac{ \delta^2 R_-^3 M_0 }{ R_+^3 } \text{,}
\]
and hence we conclude
\begin{equation}
\label{eql.obs_ext_6} I_{1, \leq} + I_\Gamma \geq \frac{ C \delta^2 R_-^3 }{ R_+^3 } \left( \frac{ R_- }{ 4 } e^\frac{ b R_- }{4} \right)^{4 a} \int_{ t (P) - \frac{ R_- }{4} }^{ t (P) + \frac{ R_- }{4} } \int_{ \mc{U}_\tau } [ | \nabla_{ t, x } \phi |^2 + ( 1 + M_0 ) \phi^2 ] d \tau \text{.}
\end{equation}
Applying the energy estimate \eqref{eq.energy_est} to each $\mc{U}_\tau$ in \eqref{eql.obs_ext_6} yields
\[
\int_{ \mc{U}_{ t (P) } } [ | \nabla_{ t, x } \phi |^2 + ( 1 + M_0 ) \phi^2 ] \leq C e^{ K R_- ( 1 + M_0^\frac{1}{2} + M_1 ) } \int_{ \mc{U}_\tau } [ | \nabla_{ t, x } \phi |^2 + ( 1 + M_0 ) \phi^2 ] \text{,}
\]
where the constants $C, K > 0$ also depend on $\mc{U}$.
Hence, it follows that
\begin{equation}
\label{eql.obs_ext_7} I_{1, \leq} + I_\Gamma \geq \frac{ C \delta^2 R_-^4 }{ R_+^3 } \left( \frac{ R_- }{ 4 } e^\frac{ b R_- }{4} \right)^{4 a} e^{ - K R_- ( 1 + M_0^\frac{1}{2} + M_1 ) } \int_{ \mc{U}_{ t (P) } } [ | \nabla_{ t, x } \phi |^2 + ( 1 + M_0 ) \phi^2 ] \text{.}
\end{equation}

For $I_{1, \leq}$, we use \eqref{eq.carleman_weight}, \eqref{eq.uvf_bound}, \eqref{eq.conf_comp}, \eqref{eql.obs_ext_b}, and \eqref{eql.obs_ext_Udec} to bound
\[
\zeta^P_{ a, b; \varepsilon } |_{ \mc{U}_\leq } \leq \left( \frac{ R_- }{ 8 } e^\frac{ b R_- }{ 8 } \right)^{ 4 a } \text{,} \qquad f_P |_{ \mc{U}_\leq } \leq R_-^2 \text{.}
\]
We then recall \eqref{eq.obs_ext_MR}, \eqref{eql.obs_ext_a}, and \eqref{eql.obs_ext_I1dec} to obtain
\[
I_{1, \leq} \leq \frac{ C' R_-^2 M_1^2 }{ a } \left( \frac{ R_- }{ 8 } e^\frac{ b R_- }{8} \right)^{4 a} \int_{ \mc{U} \cap \mc{D}_P } | \nabla_{ t, x } \phi |^2 \leq \frac{ C' \delta^2 R_-^4 }{ R_+^4 } \left( \frac{ R_- }{ 8 } e^\frac{ b R_- }{8} \right)^{4 a} \int_{ \mc{U} \cap \mc{D}_P } | \nabla_{ t, x } \phi |^2 \text{.}
\]
Note that since $\partial \mc{U}$ is timelike, then \eqref{eq.obs_ext_MR} implies that $t_P |_{ \mc{U} \cap \mc{D}_P }$ must be bounded by $R_+$.
Consequently, applying \eqref{eq.energy_est_loc} and Fubini's theorem to the above yields
\begin{equation}
\label{eql.obs_ext_80} I_{1, \leq} \leq \frac{ C' \delta^2 R_-^4 }{ R_+^3 } \left( \frac{ R_- }{ 8 } e^\frac{ b R_- }{8} \right)^{4 a} e^{ K' R_+ ( 1 + M_0^\frac{1}{2} + M_1 ) } \int_{ \mc{U}_{ t (P) } } [ | \nabla_{ t, x } \phi |^2 + ( 1 + M_0 ) \phi^2 ] \text{,}
\end{equation}
where $C', K' > 0$ also depend on $\mc{U}$.
Combining \eqref{eql.obs_ext_7} and \eqref{eql.obs_ext_80} then yields
\begin{align}
\label{eql.obs_ext_8} I_\Gamma &\geq \frac{ \delta^2 R_-^4 }{ R_+^3 } \left[ C \left( \frac{ R_- }{ 4 } e^\frac{ b R_- }{4} \right)^{4 a} e^{ - K R_- ( 1 + M_0^\frac{1}{2} + M_1 ) } - C' \left( \frac{ R_- }{ 8 } e^\frac{ b R_- }{8} \right)^{4 a} e^{ K' R_+ ( 1 + M_0^\frac{1}{2} + M_1 ) } \right] \\
\notag &\qquad \cdot \int_{ \mc{U}_{ t (P) } } [ | \nabla_{ t, x } \phi |^2 + ( 1 + M_0 ) \phi^2 ] \text{.}
\end{align}

We now claim that
\begin{equation}
\label{eql.obs_ext_absorb} C' \left( \frac{ R_- }{ 8 } e^\frac{ b R_- }{8} \right)^{4 a} e^{ K' R_+ ( 1 + M_0^\frac{1}{2} + M_1 ) } \ll C \left( \frac{ R_- }{ 4 } e^\frac{ b R_- }{4} \right)^{4 a} e^{ - K R_- ( 1 + M_0^\frac{1}{2} + M_1 ) } \text{.}
\end{equation}
To prove \eqref{eql.obs_ext_absorb}, it suffices to show, for known $c, k > 0$ depending on $\mc{U}$, that
\[
c \left( \frac{1}{2} \right)^{4a} e^{ k R_+ ( 1 + M_0^\frac{1}{2} + M_1 ) } \ll 1 \text{.}
\]
For this, we need only show that
\begin{equation}
\label{eql.obs_ext_absorb2} a \gg_{ \mc{U} } R_+ ( 1 + M_0^\frac{1}{2} + M_1 ) \text{,}
\end{equation}
which is a consequence of \eqref{eql.obs_ext_a}, since
\[
R_+ M_0^\frac{1}{2} \lesssim \delta + \delta^{ - \frac{1}{3} } R_+^\frac{4}{3} M_0^\frac{2}{3} \ll_{ \mc{U} } a \text{,} \qquad R_+ M_1 \lesssim \delta^2 + \delta^{-2} R_-^{-2} R_+^4 M_1^2 \ll_{ \mc{U} } a \text{.}
\]

Having proved \eqref{eql.obs_ext_absorb}, we combine this with \eqref{eql.obs_ext_8} and obtain
\begin{equation}
\label{eql.obs_ext_9} I_\Gamma \geq \frac{ C \delta^2 R_-^4 }{ R_+^3 } \left( \frac{ R_- }{ 4 } e^\frac{ b R_- }{4} \right)^{4 a} e^{ - K R_- ( 1 + M_0^\frac{1}{2} + M_1 ) } \int_{ \mc{U}_{ t (P) } } ( | \nabla_{ t, x } \phi |^2 + \phi^2 ) \text{.}
\end{equation}

\subsubsection{The Boundary Term}

Since the integrand in $I_\Gamma$ is negative when $\mc{S}$ is, then \eqref{eql.obs_ext_9} becomes
\begin{equation}
\label{eql.obs_ext_A} \int_{ \Gamma_+ } \zeta^P_{ a, b; \varepsilon } | \mc{S} | \cdot | \mc{N} \phi |^2 \geq \frac{ C \delta^2 R_-^4 }{ R_+^3 } \left( \frac{ R_- }{4} \right)^{4a} e^{ - K R_- ( 1 + M_0^\frac{1}{2} + M_1 ) } \int_{ \mc{U}_{ t (P) } } ( | \nabla_{ t, x } \phi |^2 + \phi^2 ) \text{.}
\end{equation}
Moreover, \eqref{eq.carleman_weight}, \eqref{eq.uvf_bound}, \eqref{eq.conf_comp}, \eqref{eq.obs_ext_MR}, and \eqref{eql.obs_ext_b} imply
\[
\zeta^P_{ a, b; \varepsilon } \leq ( 4 f_P )^{2a} \leq ( 2 R_+ )^{4a} \text{,}
\]
hence \eqref{eql.obs_ext_A} becomes
\[
\int_{ \mc{U}_{ t (P) } } ( | \nabla_{ t, x } \phi |^2 + \phi^2 ) \leq \frac{ C }{ \delta^2 R_- } \left( \frac{ 8 R_+ }{ R_- } \right)^{4a + 3} e^{ K R_- ( 1 + M_0^\frac{1}{2} + M_1 ) } \int_{ \Gamma_+ } | \mc{S} | \cdot | \mc{N} \phi |^2 \text{.}
\]
Combining the above with \eqref{eql.obs_ext_absorb2} then yields
\begin{equation}
\label{eql.obs_ext_B} \int_{ \mc{U}_{ t (P) } } ( | \nabla_{ t, x } \phi |^2 + \phi^2 ) \leq \frac{ C }{ \delta^2 R_- } \left( \frac{ 16 R_+ }{ R_- } \right)^{4a + 3} \int_{ \Gamma_+ } | \mc{S} | \cdot | \mc{N} \phi |^2 \text{.}
\end{equation}
Choosing $a \geq n$ in \eqref{eql.obs_ext_B} so that it also satisfies the assumptions \eqref{eql.obs_ext_a} results in \eqref{eq.obs_ext_est}.

\subsection{Interior Observability} \label{sec.obs_int}

In this subsection, we prove the second of our general observability inequalities, based on applying Carleman estimates that are centered within $\mc{U}$.

\begin{theorem} \label{thm.obs_int}
Let $\mc{U}$ be a GTC in $\R^{1+n}$, let $0 < \delta \ll 1$, and fix $P_1, P_2 \in \mc{U}$ with
\begin{equation}
\label{eq.obs_int_Pi} P_1 \neq P_2 \text{,} \qquad t ( P_1 ) = t ( P_2 ) := t_0 \text{.}
\end{equation}
Also, assume $\mc{U} \cap ( \mc{D}_{ P_1 } \cup \mc{D}_{ P_2 } )$ is bounded, and consider the setting of Problem \ref{prb.linear_wave}.
Moreover:
\begin{itemize}
\item Let $\mc{X}, V$ be as in \eqref{eq.XV}, and define the constants
\begin{align}
\label{eq.obs_int_MR} M_0 := \max_{ i = 1, 2 } \sup_{ \mc{U} \cap \mc{D}_{ P_i } } | V | \text{,} &\qquad M_1 := \max_{ i = 1, 2 } \sup_{ \mc{U} \cap \mc{D}_{ P_i } } | \mc{X}^{ t, x } | \text{,} \\
\notag R_+ := \max_{ i = 1, 2 } \sup_{ \mc{U} \cap \mc{D}_{ P_i } } r_{ P_i } \text{,} &\qquad R_- := \frac{1}{2} | x ( P_2 ) - x ( P_1 ) | \text{.}
\end{align}

\item Let $\mc{N}$ be the outward-pointing ($g$-)unit normal to $\mc{U}$, and define, for $i \in \{ 1, 2 \}$,
\begin{equation}
\label{eq.obs_int_delta} \mc{S}_i := \left( 1 - \frac{ \delta^2 r_{ P_i } }{ R_+ } \right) \mc{N} f_{ P_i } + \frac{ \delta^2 f_{ P_i } }{ R_+ } \mc{N} r_{ P_i } \text{,} \qquad \Gamma^i_+ := \partial \mc{U} \cap \mc{D}_{ P_i } \cap \{ \mc{S}_i > 0 \} \text{.}
\end{equation}
\end{itemize}
Then, there exist $C, N > 0$, depending on $\mc{U}$, such that the observability estimate,
\begin{equation}
\label{eq.obs_int_est} \int_{ \mc{U}_{ t_0 } } ( | \nabla_{ t, x } \phi |^2 + \phi^2 ) \leq \frac{ C }{ \delta^2 R_- } \left( \frac{ 16 R_+ }{ R_- } \right)^{ N \left( n + R_+ + \frac{ R_+^\frac{4}{3} M_0^\frac{2}{3} }{ \delta^\frac{1}{3} } + \frac{ R_+^4 M_1^2 }{ \delta^2 R_-^2 } \right) } \sum_{ i = 1 }^2 \int_{ \Gamma^i_+ } | \mc{S}_i | | \mc{N} \phi |^2 \text{.}
\end{equation}
holds for any solution $\phi \in C^2 ( \mc{U} ) \cap C^1 ( \bar{\mc{U}} )$ of \eqref{eq.linear_wave} that also satisfies $\phi |_{ \partial \mc{U} \cap ( \mc{D}_{ P_1 } \cup \mc{D}_{ P_2 } ) } = 0$.
\end{theorem}

\begin{remark}
Again, the assumption that $\mc{U} \cap ( \mc{D}_{ P_1 } \cup \mc{D}_{ P_2 } )$ is bounded implies that the constants $M_0, M_1, R_\pm$ in \eqref{eq.obs_int_MR} are finite.
Note also the $\mc{S}_i$'s in \eqref{eq.obs_int_delta} are again well-defined and smooth.
\end{remark}

The main difference between the exterior and interior estimates (Theorems \ref{thm.obs_ext} and \ref{thm.obs_int}, respectively) is that the latter requires the application of \emph{two Carleman estimates, about two distinct reference points} ($P_1$, $P_2$ in Theorem \ref{thm.obs_int}).
The technical reason is that the weights in the right-hand side of the Carleman estimate \eqref{eq.carleman_est} vanish at the center point $P$, which now lies in the domain.

\begin{remark}
Moreover, for any $P \in \mc{U}$, the boundary region $\partial \mc{U} \cap \mc{D}_P$ fails the geometric control condition.
Indeed, null geodesics traveling through $P$ (i.e., along $\partial \mc{D}_P$) fail to touch $\partial \mc{U} \cap \mc{D}_P$.
\end{remark}

We prove Theorem \ref{thm.obs_int} in the remainder of this subsection.
Much of this proof is analogous to that of Theorem \ref{thm.obs_ext}; for those parts, we omit some details and refer the reader to Section \ref{sec.obs_ext}.

\subsubsection{Applications of the Carleman Estimate}

Fix $i \in \{ 1, 2 \}$.
Note from \eqref{eq.obs_int_MR} that
\begin{equation}
\label{eql.obs_int_Ui} \mc{U}^i := \mc{U} \cap \mc{D}_{ P_i } \subseteq \mc{D}_{ P_i } \cap \{ r_{ P_i } < R_+ \} \text{.}
\end{equation}
Let $a \geq n$ be large enough to also satisfy
\begin{equation}
\label{eql.obs_int_a} a \gg_{ \mc{U} } R_+ \text{,} \qquad a \gg_{ \mc{U} } \delta^{ - \frac{1}{3} } R_+^\frac{4}{3} M_0^\frac{2}{3} \text{,} \qquad a \gg_{ \mc{U} } \delta^{-2} R_-^{-2} R_+^4 M_1^2 \text{,}
\end{equation}
and choose $\varepsilon$ and $b$ as follows:
\begin{equation}
\label{eql.obs_int_b} \varepsilon := \delta^2 R_+^{-1} \text{,} \qquad b := \delta R_+^{-1} \text{.}
\end{equation}
Applying \eqref{eq.carleman_est}, with $a, b, \varepsilon$ as above, to $\mc{U}$ and $P_i$ yields
\begin{align}
\label{eql.obs_int_1} &\frac{2}{a} \int_{ \mc{U}^i } \zeta_{ a, b; \varepsilon }^{ P_i } f_{ P_i } | \nabla_{ \mc{X} } \phi |^2 + \frac{2}{a} \int_{ \mc{U}^i } \zeta_{ a, b; \varepsilon }^{ P_i } f_{ P_i } V^2 \cdot \phi^2 + \int_{ \partial \mc{U} \cap \mc{D}_{ P_i } } \zeta_{ a, b; \varepsilon }^{ P_i } \mc{S}_i \cdot | \mc{N} \phi |^2 \\
\notag &\quad \geq \frac{ C \delta^2 }{ R_+^2 } \int_{ \mc{U}^i } \zeta_{ a, b; \varepsilon }^{ P_i } ( | u_{ P_i } \partial_{ u_{ P_i } } \phi |^2 + | v_{ P_i } \partial_{ v_{ P_i } } \phi |^2 + f_{ P_i } g^{ab} \nasla^{ P_i }_a \phi \nasla^{ P_i }_b \phi ) + \frac{ C \delta a^2 }{ R_+^2 } \int_{ \mc{U}^i } \zeta_{ a, b; \varepsilon }^{ P_i } \phi^2 \text{,}
\end{align}
where we also recalled \eqref{eq.linear_wave}, \eqref{eq.uvf_bound}, \eqref{eq.obs_int_delta}, and \eqref{eql.obs_int_b}.

Let $I^i_1$, $I^i_0$, $I^i_\Gamma$ be the terms in the left-hand side of \eqref{eql.obs_int_1}.
Like in the proof of Theorem \ref{thm.obs_ext}, we apply \eqref{eq.uvf_bound} and \eqref{eql.obs_int_a} to show that $I^i_0$ can be absorbed into the right-hand side of \eqref{eql.obs_int_1}:
\begin{equation}
\label{eql.obs_int_2} I^i_1 + I^i_\Gamma \geq C \int_{ \mc{U}^i } \zeta_{ a, b; \varepsilon }^{ P_i } \frac{ \delta^2 }{ R_+^2 } \left[ ( | u_{ P_i } \partial_{ u_{ P_i } } \phi |^2 + | v_{ P_i } \partial_{ v_{ P_i } } \phi |^2 + f_{ P_i } g^{ab} \nasla^{ P_i }_a \phi \nasla^{ P_i }_b \phi ) + \frac{ \delta a^2 }{ R_+^2 } \phi^2 \right] \text{.}
\end{equation}
We now partition $\mc{U}_i$ into
\begin{align}
\label{eql.obs_int_Udec} \mc{U}^i_\leq &:= \mc{U}^i \cap \left\{ \frac{ f_{ P_i } }{ ( 1 + \varepsilon u_{ P_i } ) ( 1 - \varepsilon v_{ P_i } ) } \leq \frac{ R_-^2 }{ 64 } \right\} \text{,} \\
\notag \mc{U}^i_> &:= \mc{U}^i \cap \left\{ \frac{ f_{ P_i } }{ ( 1 + \varepsilon u_{ P_i } ) ( 1 - \varepsilon v_{ P_i } ) } > \frac{ R_-^2 }{ 64 } \right\} \text{,}
\end{align}
and we decompose $I^i_1$ as
\[
I^i_1 = \frac{2}{a} \int_{ \mc{U}^i_\leq } \zeta_{ a, b; \varepsilon }^{ P_i } f_{ P_i } | \nabla_{ \mc{X} } \phi |^2 + \frac{2}{a} \int_{ \mc{U}^i_> } \zeta_{ a, b; \varepsilon }^{ P_i } f_{ P_i } | \nabla_{ \mc{X} } \phi |^2 := I^i_{1, \leq} + I^i_{1, >} \text{.}
\]

Since \eqref{eq.uvf_bound} and \eqref{eq.conf_comp} imply that on $\mc{U}^i_>$,
\[
v_{ P_i } \gtrsim \frac{ R_-^2 }{ R_+ } \text{,} \qquad - u_{ P_i } \gtrsim \frac{ R_-^2 }{ R_+ } \text{,}
\]
a similar argument as in the proof of Theorem \ref{thm.obs_ext}, using \eqref{eq.uvf_bound}, \eqref{eq.obs_int_MR}, and \eqref{eql.obs_int_a}, imply that $I^i_{1, >}$ can be absorbed into the right-hand side of \eqref{eql.obs_int_2}.
As a result, \eqref{eql.obs_int_2} now becomes
\begin{equation}
\label{eql.obs_int_4} I^i_{1, \leq} + I^i_\Gamma \geq C \int_{ \mc{U}^i_> } \zeta_{ a, b; \varepsilon }^{ P_i } \left[ \frac{ \delta^2 R_-^2 }{ R_+^3 } ( - u_{ P_i } | \partial_{ u_{ P_i } } \phi |^2 + v_{ P_i } | \partial_{ v_{ P_i } } \phi |^2 + v_{ P_i } g^{ab} \nasla^{ P_i }_a \phi \nasla^{ P_i }_b \phi ) + \frac{ \delta a^2 }{ R_+^2 } \phi^2 \right] \text{.}
\end{equation}

\subsubsection{Applications of the Energy Estimates}

Since $\partial \mc{U}$ is timelike, then on any
\[
\mc{V}^i_\tau := \mc{U}^i_\tau \cap \left\{ r_{ P_i } > \frac{ 3 R_- }{4} \right\} \text{,} \qquad | \tau - t_0 | \leq \frac{ R_- }{4} \text{,}
\]
we have from \eqref{eq.uv} and \eqref{eq.carleman_weight} that
\begin{align}
\label{eql.obs_int_50} - u_{ P_i } |_{ \mc{V}^i_\tau } \geq \frac{ R_- }{4} \text{,} &\qquad v_{ P_i } |_{ \mc{V}^i_\tau } \geq \frac{ R_- }{4} \text{,} \\
\notag \frac{ f_{ P_i } }{ ( 1 + \varepsilon u_{ P_i } ) ( 1 - \varepsilon v_{ P_i } ) } |_{ \mc{V}^i_\tau } \geq \frac{ R_-^2 }{ 16 } \text{,} &\qquad \zeta_{ a, b; \varepsilon }^{ P_i } |_{ \mc{V}^i_\tau } \geq \left( \frac{ R_- }{4} e^\frac{ b R_- }{4} \right)^{ 4a } \text{.}
\end{align}
In particular, we see that
\[
\mc{V}^i_\tau \subseteq \mc{U}^i_> \text{,} \qquad | \tau - t_0 | \leq \frac{ R_- }{4} \text{,}
\]
and, similar to the proof of Theorem \ref{thm.obs_ext}, it follows from \eqref{eql.obs_int_a} that
\begin{align}
\label{eql.obs_int_6} I^i_{1, \leq} + I^i_\Gamma &\geq C \left( \frac{ R_- }{ 4 } e^\frac{ b R_- }{4} \right)^{4 a} \int_{ t_0 - \frac{ R_- }{4} }^{ t_0 + \frac{ R_- }{4} } \int_{ \mc{V}^i_\tau } \left[ \frac{ \delta^2 R_-^3 }{ R_+^3 } \cdot | \nabla_{ t, x } \phi |^2 + \frac{ \delta a^2 }{ R_+^2 } \cdot \phi^2 \right] d \tau \\
\notag &\geq \frac{ C \delta^2 R_-^3 }{ R_+^3 } \left( \frac{ R_- }{ 4 } e^\frac{ b R_- }{4} \right)^{4 a} \int_{ t_0 - \frac{ R_- }{4} }^{ t_0 + \frac{ R_- }{4} } \int_{ \mc{V}^i_\tau } [ | \nabla_{ t, x } \phi |^2 + ( 1 + M_0 ) \phi^2 ] d \tau \text{.}
\end{align}

Next, we observe that by \eqref{eq.obs_int_MR} and the triangle inequality,
\[
\mc{V}^1_\tau \cup \mc{V}^2_\tau = \mc{U}_\tau \text{,} \qquad | \tau - t_0 | \leq \frac{ R_- }{4} \text{,}
\]
Thus, summing \eqref{eql.obs_int_6} over $i \in \{ 1, 2 \}$ and applying \eqref{eq.energy_est} yields
\begin{align}
\label{eql.obs_int_7} \sum_{ i = 1}^2 ( I^i_{1, \leq} + I^i_\Gamma ) &\geq \frac{ C \delta^2 R_-^3 }{ R_+^3 } \left( \frac{ R_- }{ 4 } e^\frac{ b R_- }{4} \right)^{4 a} \int_{ t_0 - \frac{ R_- }{4} }^{ t_0 + \frac{ R_- }{4} } \int_{ \mc{U}_\tau } [ | \nabla_{ t, x } \phi |^2 + ( 1 + M_0 ) \phi^2 ] d \tau \\
\notag &\geq \frac{ C \delta^2 R_-^4 }{ R_+^3 } \left( \frac{ R_- }{ 4 } e^\frac{ b R_- }{4} \right)^{4 a} e^{ - K R_- ( 1 + M_0^\frac{1}{2} + M_1 ) } \int_{ \mc{U}_{ t_0 } } [ | \nabla_{ t, x } \phi |^2 + ( 1 + M_0 ) \phi^2 ] \text{,}
\end{align}
where $C, K > 0$ depend on $\mc{U}$.
Moreover, analogous to the proof of Theorem \ref{thm.obs_ext}, we apply \eqref{eq.energy_est_loc}, \eqref{eq.carleman_weight}, \eqref{eq.uvf_bound}, \eqref{eq.obs_int_MR}, \eqref{eql.obs_int_a}, and \eqref{eql.obs_int_Udec} to estimate, for some $C', K'$ depending on $\mc{U}$,
\begin{align}
\label{eql.obs_int_80} I^i_{1, \leq} &\leq \frac{ C' \delta^2 R_-^4 }{ R_+^4 } \left( \frac{ R_- }{ 8 } e^\frac{ b R_- }{8} \right)^{4 a} \int_{ \mc{U}^i } | \nabla_{ t, x } \phi |^2 \\
\notag &\leq \frac{ C' \delta^2 R_-^4 }{ R_+^3 } \left( \frac{ R_- }{ 8 } e^\frac{ b R_- }{8} \right)^{4 a} e^{ K' R_+ ( 1 + M_0^\frac{1}{2} + M_1 ) } \int_{ \mc{U}_{ t_0 } } [ | \nabla_{ t, x } \phi |^2 + ( 1 + M_0 ) \phi^2 ] \text{.}
\end{align}
By the same reasoning as in the proof of Theorem \ref{thm.obs_ext} (see \eqref{eql.obs_ext_absorb}), we have that
\[
C' \left( \frac{ R_- }{ 8 } e^\frac{ b R_- }{8} \right)^{4 a} e^{ K' R_+ ( 1 + M_0^\frac{1}{2} + M_1 ) } \ll C \left( \frac{ R_- }{ 4 } e^\frac{ b R_- }{4} \right)^{4 a} e^{ - K R_- ( 1 + M_0^\frac{1}{2} + M_1 ) } \text{,}
\]
so that combining all the above with \eqref{eql.obs_int_7} and \eqref{eql.obs_int_80} yields
\begin{equation}
\label{eql.obs_int_9} \sum_{i = 1}^2 I^i_\Gamma \geq \frac{ C \delta^2 R_-^4 }{ R_+^3 } \left( \frac{ R_- }{ 4 } e^\frac{ b R_- }{4} \right)^{4 a} e^{ - K R_- ( 1 + M_0^\frac{1}{2} + M_1 ) } \int_{ \mc{U}_{ t_0 } } ( | \nabla_{ t, x } \phi |^2 + \phi^2 ) \text{.}
\end{equation}

\subsubsection{The Boundary Terms}

Noting the signs of $\mc{S}_i$ in $I^i_\Gamma$, we see from \eqref{eq.obs_int_delta} and \eqref{eql.obs_int_9} that
\begin{equation}
\label{eql.obs_int_A} \sum_{ i = 1 }^2 \int_{ \Gamma_+^i } \zeta_{ a, b; \varepsilon }^{ P_i } | \mc{S}_i | \cdot | \mc{N} \phi |^2 \geq \frac{ C \delta^2 R_-^4 }{ R_+^3 } \left( \frac{ R_- }{4} \right)^{4a} e^{ - K R_- ( 1 + M_0^\frac{1}{2} + M_1 ) } \int_{ \mc{U}_{ t_0 } } ( | \nabla_{ t, x } \phi |^2 + \phi^2 ) \text{.}
\end{equation}
The remainder of the proof again proceeds similarly to the corresponding argument in Theorem \ref{thm.obs_ext}.
First, we can use \eqref{eq.carleman_weight}, \eqref{eq.uvf_bound}, \eqref{eq.conf_comp}, \eqref{eq.obs_int_MR}, and \eqref{eql.obs_int_b} to bound
\[
\zeta^{ P_i }_{ a, b; \varepsilon } \leq ( 2 R_+ )^{4a} \text{,}
\]
so that \eqref{eql.obs_int_A} becomes
\begin{align}
\label{eql.obs_int_12} \int_{ \mc{U}_{ t_0 } } ( | \nabla_{ t, x } \phi |^2 + \phi^2 ) &\leq \frac{ C }{ \delta^2 R_- } \left( \frac{ 8 R_+ }{ R_- } \right)^{4a + 3} e^{ K R_- ( 1 + M_0^\frac{1}{2} + M_1 ) } \sum_{ i = 1 }^2 \int_{ \Gamma^i_+ } | \mc{S}_i | \cdot | \mc{N} \phi |^2 \\
\notag &\leq \frac{ C }{ \delta^2 R_- } \left( \frac{ 16 R_+ }{ R_- } \right)^{4a + 3} \sum_{ i = 1 }^2 \int_{ \Gamma^i_+ } | \mc{S}_i | \cdot | \mc{N} \phi |^2 \text{.}
\end{align}
(In the last step, we used the direct analogue of \eqref{eql.obs_ext_absorb2} in the current setting.)
Making a suitable choice for $a \geq n$ in accordance with \eqref{eql.obs_int_a} results in \eqref{eq.obs_int_est}.

\subsection{Some Corollaries} \label{sec.obs_cor}

Finally, we discuss various consequences of Theorems \ref{thm.obs_ext} and \ref{thm.obs_int}.

\subsubsection{The Observation Region}

Our first corollary compares the observation regions from Theorems \ref{thm.obs_ext} and \ref{thm.obs_int} with corresponding regions in existing Carleman method results.
For this, we examine the quantities $\mc{S}$ and $\mc{S}_i$ in \eqref{eq.obs_ext_delta} and \eqref{eq.obs_int_delta}, whose signs determine the observation region.

\begin{proposition} \label{thm.obs_calc}
Let $\mc{U}$ be a GTC, and let $P \in \R^{1+n} \setminus \partial \mc{U}$.
Moreover:
\begin{itemize}
\item Let $\mc{N}$ denote the outer-pointing (Minkowski) unit normal of $\mc{U}$, and let $\nu^t: \partial \mc{U} \rightarrow \R$ and $\nu: \partial \mc{U} \rightarrow \R^n$ denote the Cartesian components of $\mc{N}$:
\begin{equation}
\label{eq.obs_calc_N} \mc{N} := \nu^t \partial_t + \sum_{ i = 1 }^n \nu^i \partial_{ x^i } \text{,} \qquad \nu := ( \nu^1, \dots, \nu^n ) \text{.}
\end{equation}

\item Let $\theta_{ \nu, P }: \partial \mc{U} \rightarrow ( - \pi, \pi ]$ represent the angle between the vectors $\nu$ and $x_P |_{ \partial \mc{U} }$.
\end{itemize}
Then:
\begin{itemize}
\item The following identity holds on $\partial \mc{U} \cap \mc{D}_P$:
\begin{equation}
\label{eq.obs_calc_pre} \mc{N} r_P = \nu \cdot \frac{ x_P }{ r_P } = \sqrt{ 1 + ( \nu^t )^2 } \cdot \cos \theta_{ \nu, P } \text{,}
\end{equation}

\item Furthermore, for any $\varepsilon > 0$, we have on $\partial \mc{U} \cap \mc{D}_P$ that
\begin{align}
\label{eq.obs_calc} ( 1 - \varepsilon r_P ) \mc{N} f_P + \varepsilon f_P \mc{N} r_P &= \frac{1}{2} \left[ 1 - \frac{ \varepsilon ( r_P^2 + t_P^2 ) }{ 2 r_P } \right] ( \nu \cdot x_P ) - \frac{1}{2} ( 1 - \varepsilon r_P ) t_P \nu^t \text{.}
\end{align}
\end{itemize}
\end{proposition}

\begin{proof}
First, by Definition \ref{def.mink_shift} and \eqref{eq.obs_calc_N}, we have that
\[
\mc{N} ( r_P^2 ) = 2 \sum_{ i = 1 }^n \nu^i \partial_{ x_i } | x_P |^2 = 2 \nu \cdot x_P \text{,} \qquad \mc{N} r_P = \frac{1}{ 2 r_P } \mc{N} ( r_P^2 ) = \nu \cdot \frac{ x_P }{ r_P } = | \nu | \cdot \cos \theta_{ \nu, x_P } \text{.}
\]
Since $g ( \mc{N}, \mc{N} ) = 1$, it follows that
\[
| \nu | = \sqrt{ 1 + ( \nu^t )^2 } \text{,}
\]
which completes the proof of \eqref{eq.obs_calc_pre}.

Next, using \eqref{eq.ruvf_shift}, we expand
\begin{align*}
( 1 - \varepsilon r_P ) \mc{N} f_P + \varepsilon f_P \mc{N} r_P &= \frac{1}{2} ( 1 - \varepsilon r_P ) ( r_P \mc{N} r_P - t_P \mc{N} t_P ) + \frac{1}{4} \varepsilon ( r_P^2 - t_P^2 ) \mc{N} r_P \\
&= \frac{1}{2} \left[ r_P - \frac{1}{2} \varepsilon ( r_P^2 + t_P^2 ) \right] \mc{N} r_P - \frac{1}{2} ( 1 - \varepsilon r_P ) t_P \mc{N} t_P \text{.}
\end{align*}
Since $\mc{N} t_P = \nu^t$ by \eqref{eq.obs_calc_N}, then \eqref{eq.obs_calc_pre} and the above identity imply \eqref{eq.obs_calc}.
\end{proof}

\begin{remark} \label{rmk.obs_standard}
Note that by \eqref{eq.obs_calc_pre}, the condition $\mc{N} r_P > 0$ is equivalent to both
\[
x_P \cdot \nu = [ x - x (P) ] \cdot \nu > 0 \text{,} \qquad \cos \theta_{ \nu, P } > 0 \text{,}
\]
that is, the criterion for the observation region obtained from standard Carleman estimate methods for linear wave equations on a time-independent domain; see, for instance, \eqref{eq.intro_obs_region} and \cite{lasie_trigg_zhang:wave_global_uc, lionj:ctrlstab_hum, zhang:obs_wave_pot}.
\end{remark}

Using Proposition \ref{thm.obs_calc}, we can now reformulate Theorems \ref{thm.obs_ext} and \ref{thm.obs_int} so that the regions of observation are characterized in terms of the aforementioned angle $\theta_{ \nu, P }$:

\begin{theorem} \label{thm.obs_angle}
Let $\mc{U}$ be a GTC, and fix $P \in \R^{1+n} \setminus \partial \mc{U}$ and $0 < \delta \ll 1$.
In addition:
\begin{itemize}
\item Let $\mc{N}$ be the outward-pointing ($g$-)unit normal to $\mc{U}$, let $\nu^t$ and $\nu$ be defined as in \eqref{eq.obs_calc_N}, and let $\theta_{ \nu, P }: \partial \mc{U} \rightarrow ( - \pi, \pi ]$ denote the angle between $\nu$ and $x_P$.

\item Consider the setting of Problem \ref{prb.linear_wave}, and let $\mc{X}$, $V$ be as in \eqref{eq.XV}.

\item Assume $\mc{U} \cap \mc{D}_P$ is a bounded subset of $\R^{1+n}$, and let $M_0$, $M_1$ be as in \eqref{eq.obs_ext_MR}.

\item Let $\Gamma_{ P, \delta }$ denote the following boundary region:
\begin{equation}
\label{eq.obs_angle_Gamma} \Gamma_{ P, \delta } := \partial \mc{U} \cap \mc{D}_P \cap \left\{ \cos \theta_{ \nu, P } > ( 1 - \delta^2 )^{ \operatorname{sgn} ( t_P \nu^t ) } \cdot \frac{ t_P \nu^t }{ r_P \sqrt{ 1 + ( \nu^t )^2 } } \right\} \text{.}
\end{equation}
\end{itemize}
Then, for any solution $\phi \in C^2 ( \mc{U} ) \cap C^1 ( \bar{\mc{U}} )$ of \eqref{eq.linear_wave} that also satisfies $\phi |_{ \partial \mc{U} \cap \mc{D}_P } = 0$:
\begin{itemize}
\item If $P \not\in \bar{\mc{U}}$, then we have the observability inequality
\begin{equation}
\label{eq.obs_angle_ext} \int_{ \mc{U}_{ t (P) } } ( | \nabla_{ t, x } \phi |^2 + \phi^2 ) \lesssim_{ \mc{U}, P, \delta, M_0, M_1 } \int_{ \Gamma_{ P, \delta } } | \mc{N} \phi |^2 \text{.}
\end{equation}

\item If $P \in \mc{U}$, then for any open subset $\mc{Y}_{ P, \delta }$ of $\partial \mc{U}$ containing $\bar{\Gamma}_{ P, \delta }$, we have that
\begin{equation}
\label{eq.obs_angle_int} \int_{ \mc{U}_{ t (P) } } ( | \nabla_{ t, x } \phi |^2 + \phi^2 ) \lesssim_{ \mc{U}, P, \delta, M_0, M_1, \mc{Y}_{ P, \delta } } \int_{ \mc{Y}_{ P, \delta } } | \mc{N} \phi |^2 \text{.}
\end{equation}
\end{itemize}
\end{theorem}

\begin{remark}
In particular, recalling Remark \ref{rmk.obs_standard}, we can make, modulo restrictions to $\mc{D}_P$, the following comparison between the region $\Gamma_{ P, \delta }$ in Theorem \ref{thm.obs_angle} and the classical region $\theta_{ P, \delta } > 0$:
\begin{itemize}
\item At points where $t_P \nu^t > 0$ (that is, where $\mc{U}$ expands away from $P$), the observation region $\Gamma_{ P, \delta }$ is smaller than the corresponding region $\cos \theta_{ \nu, P } > 0$ in classical results.

\item At points where $t_P \nu^t < 0$ (that is, where $\mc{U}$ shrinks away from $P$), the observation region $\Gamma_{ P, \delta }$ is larger than the corresponding region $\cos \theta_{ \nu, P } > 0$ in classical results.

\item At points where $t_P \nu^t = 0$, the observation region $\Gamma_{ P, \delta }$ matches the classical results.
\end{itemize}
\end{remark}

\begin{proof}
First, we consider the case $P \not\in \bar{\mc{U}}$.
Applying Theorem \ref{thm.obs_ext} to $\mc{U}$, $P$, and $\delta$ yields
\[
\int_{ \mc{U}_{ t (P) } } ( | \nabla_{ t, x } \phi |^2 + \phi^2 ) \lesssim_{ \mc{U}, P, \delta, M_0, M_1 } \int_{ \Gamma_+ } | \mc{N} \phi |^2
\]
for $\phi$ as in the assumptions of the theorem, and with $\Gamma_+$ as defined in \eqref{eq.obs_ext_delta}.
Thus, to complete the proof in the case, it suffices to show that this $\Gamma_+$ is contained in $\Gamma_{ P, \delta }$.

Let $R_+$ be as in \eqref{eq.obs_ext_MR}.
Applying \eqref{eq.obs_calc_pre} and \eqref{eq.obs_calc}, with $\varepsilon := \delta^2 R_+^{-1}$, we see that on $\Gamma_+$,
\begin{align*}
0 &< ( 1 - \varepsilon r_P ) \mc{N} f_P + \varepsilon f_P \mc{N} r_P \\
&= \frac{1}{2} \left[ 1 - \frac{ \delta^2 ( r_P^2 + t_P^2 ) }{ 2 R_+ r_P } \right] r_P \sqrt{ 1 + ( \nu^t )^2 } \cdot \cos \theta_{ \nu, P } - \frac{1}{2} \left( 1 - \frac{ \delta^2 r_P }{ R_+ } \right) \cdot t_P \nu^t \text{,}
\end{align*}
or equivalently,
\begin{equation}
\label{eql.obs_angle_1} \cos \theta_{ \nu, x_P } > \left[ \frac{ 1 - \frac{ \delta^2 ( r_P^2 + t_P^2 ) }{ 2 r_P R_+ } }{ 1 - \frac{ \delta^2 r_P }{ R_+ } } \right] \cdot \frac{ t_P \nu_t }{ r_P \sqrt{ 1 + ( \nu^t )^2 } } \text{.}
\end{equation}
Using that $| t_P | \leq r_P \leq R_+$ on $\partial \mc{U} \cap \mc{D}_P$, we observe that:
\begin{itemize}
\item If $t_P \nu^t = 0$, then \eqref{eql.obs_angle_1} implies $\cos \theta_{ \nu, x_P } > 0$.

\item If $t_P \nu^t > 0$, then \eqref{eql.obs_angle_1} implies
\[
\cos \theta_{ \nu, x_P } > \left( 1 - \frac{ 2 \delta^2 r_P^2 }{ 2 r_P R_+ } \right) \cdot \frac{ t_P \nu_t }{ r_P \sqrt{ 1 + ( \nu^t )^2 } } \geq ( 1 - \delta^2 ) \cdot \frac{ t_P \nu_t }{ r_P \sqrt{ 1 + ( \nu^t )^2 } } \text{.}
\]

\item If $t_P \nu^t < 0$, then \eqref{eql.obs_angle_1} implies
\[
\cos \theta_{ \nu, x_P } > \frac{1}{ 1 - \frac{ \delta^2 r_P }{ R_+ } } \cdot \frac{ t_P \nu_t }{ r_P \sqrt{ 1 + ( \nu^t )^2 } } \geq \frac{1}{ 1 - \delta^2 } \cdot \frac{ t_P \nu_t }{ r_P \sqrt{ 1 + ( \nu^t )^2 } } \text{.}
\]
\end{itemize}
Thus, comparing the above with \eqref{eq.obs_angle_Gamma}, we conclude that $\Gamma_+ \subseteq \Gamma_{ P, \delta }$, as desired.

Next, consider the case $P \in \mc{U}$, and let $\mc{Y}_{ P, \delta }$ be as assumed.
Since $\mc{U} \cap \mc{D}_P$ is bounded, $\bar{\Gamma}_{ P, \delta }$ is compact.
Thus, by continuity, we can find nearby $P_1, P_2 \in \mc{U}$, with $t ( P_1 ) = t ( P_2 ) = t ( P )$, with
\begin{equation}
\label{eql.obs_angle_2} \Gamma_{ P_1, \delta } \cup \Gamma_{ P_2, \delta } \subseteq \mc{Y}_{ P, \delta } \text{,}
\end{equation}
where $\Gamma_{ P_i, \delta }$ is defined as in \eqref{eq.obs_angle_Gamma}, but with $P$ replaced by $P_i$.

Applying Theorem \ref{thm.obs_int} to $\mc{U}$, $( P_1, P_2 )$, and $\delta$ yields
\[
\int_{ \mc{U}_{ t (P) } } ( | \nabla_{ t, x } \phi |^2 + \phi^2 ) \lesssim_{ \mc{U}, P, \delta, M_0, M_1, \mc{Y}_{ P, \delta } } \int_{ \Gamma_+^1 \cup \Gamma_+^2 } | \mc{N} \phi |^2 \text{,}
\]
for $\phi$ as in the assumptions of the theorem, and with $\Gamma_+^1$, $\Gamma_+^2$ as defined in \eqref{eq.obs_int_delta}.\footnote{Note that the constant $R_-$ in \eqref{eq.obs_int_est} depends on $P$ and our choice of $\mc{Y}_{ P, \delta }$.}
Consequently, to complete the proof, it suffices to show that $\Gamma_+^1 \cup \Gamma_+^2$ is contained in $\mc{Y}_{ P, \delta }$.

For this, we apply \eqref{eq.obs_calc_pre} and \eqref{eq.obs_calc} as before in order to obtain
\[
\Gamma_+^1 \subseteq \Gamma_{ P_1, \delta } \text{,} \qquad \Gamma_+^2 \subseteq \Gamma_{ P_2, \delta } \text{.}
\]
Combining this with \eqref{eql.obs_angle_2} yields, as desired,
\[
\Gamma_+^1 \cup \Gamma_+^2 \subseteq \Gamma_{ P_1, \delta } \cup \Gamma_{ P_2, \delta } \subseteq \mc{Y}_{ P, \delta } \text{.} \qedhere
\]
\end{proof}

\subsubsection{Static Domains}

Next, we consider the special case of time-independent domains mentioned in Example \ref{ex.static_cyl}.
Applying Theorem \ref{thm.obs_angle} to this setting results in the following:

\begin{corollary} \label{thm.obs_static}
Let $\Omega \subseteq \R^n$ be open, and consider the time-independent GTC
\begin{equation}
\label{eq.obs_static_U} \mc{U} := \R \times \Omega \text{.}
\end{equation}
In addition:
\begin{itemize}
\item Fix $x_0 \in \R^n$, as well as $\tau_\pm \in \R$ satisfying
\begin{equation}
\label{eq.obs_static_R} \tau_+ - \tau_- > 2 R \text{,} \qquad R := \sup_{ y \in \partial \Omega } | y - x_0 | \text{.}
\end{equation}

\item Consider the setting of Problem \ref{prb.linear_wave}, and let $\mc{X}$, $V$ be as in \eqref{eq.XV}.

\item Let $\nu: \partial \Omega \rightarrow \R^n$ represent the outer unit normal to $\Omega$ in $\R^n$, and define
\begin{equation}
\label{eq.obs_static_Gamma} \Gamma_{ x_0; \tau_\pm } := \mc{D}_P \cap \{ ( \tau, y ) \in \partial \mc{U}_{ \tau_-, \tau_+ } \mid ( y - x_0 ) \cdot \nu (y) > 0 \} \text{,} \qquad P := \left( \frac{ \tau_+ + \tau_- }{2}, x_0 \right) \text{.}
\end{equation}
\end{itemize}
Then, for any solution $\phi \in C^2 ( \mc{U} ) \cap C^1 ( \bar{\mc{U}} )$ of \eqref{eq.linear_wave} that also satisfies $\phi |_{ \partial \mc{U}_{ \tau_-, \tau_+ } } = 0$:
\begin{itemize}
\item If $x_0 \not\in \bar{\Omega}$, then we have the observability inequality
\begin{equation}
\label{eq.obs_static_ext} \int_{ \mc{U}_{ \tau_\pm } } ( | \nabla_{ t, x } \phi |^2 + \phi^2 ) \lesssim_{ \Omega, x_0, \tau_\pm, V, \mc{X} } \int_{ \Gamma_{ x_0; \tau_\pm } } \left| \sum_{ i = 1 }^n \nu^i \partial_{ x_i } \phi \right|^2 \text{.}
\end{equation}

\item If $x_0 \in \bar{\Omega}$, then for any open subset $\mc{Y}_{ x_0; \tau_\pm }$ of $\partial \mc{U}$ containing $\bar{\Gamma}_{ x_0; \tau_\pm }$,
\begin{equation}
\label{eq.obs_static_int} \int_{ \mc{U}_{ \tau_\pm } } ( | \nabla_{ t, x } \phi |^2 + \phi^2 ) \lesssim_{ \Omega, x_0, \tau_\pm, V, \mc{X}, \mc{Y}_{ x_0; \tau_\pm } } \int_{ \mc{Y}_{ x_0; \tau_\pm } } \left| \sum_{ i = 1 }^n \nu^i \partial_{ x_i } \phi \right|^2 \text{.}
\end{equation}
\end{itemize}
\end{corollary}

\begin{remark}
In particular, we can directly compare the results in Corollary \ref{thm.obs_static} with those from classical Carleman-based methods.
Recalling Remark \ref{rmk.obs_standard}, we see that:
\begin{itemize}
\item When $x_0 \not\in \bar{\Omega}$, the observation region $\Gamma_{ x_0, \tau_\pm }$ is simply the observation region from classical results, but further restricted to the null cone exterior $\mc{D}_P$.

\item When $x_0 \in \bar{\Omega}$, the observation region $\mc{Y}_{ x_0, \tau_\pm }$ is strictly larger than the standard observation region restricted to $\mc{D}_P$, though by an arbitrarily small amount.
This arises from the fact that the interior estimate \eqref{eq.obs_int_est} must be applied at two different points.
\end{itemize}
\end{remark}

\begin{proof}
First, we observe that $\nu$ in the theorem statement coincides with the $\nu$ defined in Theorem \ref{thm.obs_angle}.
In addition, here $\nu$ also represents the Minkowski outer normal to $\mc{U}$.

Let us first assume that $x_0 \not\in \partial \Omega$, and hence $P \not\in \partial \mc{U}$.
We now apply Theorem \ref{thm.obs_angle} to the above $\mc{U}$ and $P$, along with any fixed $0 < \delta \ll 1$.
We claim that $\Gamma_{ x_0; \tau_\pm }$ coincides with $\Gamma_{ P, \delta }$, as defined in \eqref{eq.obs_angle_Gamma}; this is a consequence of the following observations:
\begin{itemize}
\item Since $\nu^t \equiv 0$ for our $\mc{U}$, we have from \eqref{eq.obs_calc_pre} and \eqref{eq.obs_angle_Gamma} that
\[
\Gamma_{ P, \delta } = \partial \mc{U} \cap \mc{D}_P \cap \{ \cos \theta_{ \nu, P } > 0 \} = \mc{D}_P \cap \{ ( y, \tau ) \in \partial \mc{U} \mid x_P ( \tau, y ) \cdot \nu (y) > 0 \} \text{.}
\]

\item The assumption $\tau_+ - \tau_- > 2 R$ in \eqref{eq.obs_static_R} implies that $\mc{D}_P \cap \partial \mc{U} \subseteq \partial \mc{U}_{ \tau_-, \tau_+ }$.
\end{itemize}
The application of Theorem \ref{thm.obs_angle} now splits into two cases:
\begin{itemize}
\item If $x_0 \not\in \bar{\Omega}$, then $P \not\in \bar{\mc{U}}$, and \eqref{eq.energy_est} and \eqref{eq.obs_angle_ext} yield
\[
\int_{ \mc{U}_{ \tau_\pm } } ( | \nabla_{ t, x } \phi |^2 + \phi^2 ) \lesssim_{ \Omega, \tau_\pm, V, \mc{X} } \int_{ \mc{U}_{ t ( P ) } } ( | \nabla_{ t, x } \phi |^2 + \phi^2 ) \lesssim_{ \Omega, x_0, \tau_\pm, V, \mc{X} } \int_{ \Gamma_{ x_0; \tau_\pm } } \left| \sum_{ i = 1 }^n \nu^i \partial_{ x_i } \phi \right|^2 \text{.}
\]

\item Similarly, if $x_0 \in \Omega$, then $P \in \mc{U}$, and \eqref{eq.energy_est} and \eqref{eq.obs_angle_int} yield
\[
\int_{ \mc{U}_{ \tau_\pm } } ( | \nabla_{ t, x } \phi |^2 + \phi^2 ) \lesssim_{ \Omega, x_0, \tau_\pm, M_0, M_1, \mc{Y}_{ x_0, \tau_\pm } } \int_{ \mc{Y}_{ x_0; \tau_\pm } } \left| \sum_{ i = 1 }^n \nu^i \partial_{ x_i } \phi \right|^2 \text{.}
\]
\end{itemize}

Finally, the remaining case $x_0 \in \partial \Omega$ can be obtained by applying Theorem \ref{thm.obs_angle} to a nearby point
\[
P' = \left( \frac{ \tau_+ + \tau_- }{2}, x_0' \right) \text{,} \qquad x_0' \not\in \Omega \text{,}
\]
with $| x_0' - x_0 |$ small enough so that $\Gamma_{ P', \delta } \subseteq \mc{Y}_{ x_0; \tau_\pm }$.\footnote{Here, $\Gamma_{ P', \delta }$ is defined as in \eqref{eq.obs_angle_Gamma}, but with $P$ replaced by $P'$.}
\end{proof}

\subsubsection{A Unified Estimate}

We conclude with a precise version of the result roughly stated in Theorem \ref{thm.intro_obs_main}.
While this is slightly weaker (in terms of the observation region) than Theorems \ref{thm.obs_ext} and \ref{thm.obs_int}, it unifies the interior and exterior estimates and provides a cleaner statement.

\begin{theorem} \label{thm.obs_combo}
Let $\mc{U} \subseteq \R^{1+n}$ be a GTC, and fix $x_0 \in \R^n$ and $\tau_\pm \in \R$ such that
\begin{equation}
\label{eq.obs_combo_R} \tau_+ - \tau_- > R_+ + R_- \text{,} \qquad R_\pm := \sup_{ ( \tau_\pm, y ) \in \partial \mc{U} } | y - x_0 | \text{.}
\end{equation}
\begin{itemize}
\item In addition, choose $t_0 \in ( \tau_-, \tau_+ )$ so that
\begin{equation}
\label{eq.obs_combo_t0} t_0 - \tau_- > R_- \text{,} \qquad \tau_+ - t_0 > R_+ \text{.}
\end{equation}

\item Consider the setting of Problem \ref{prb.linear_wave}, and let $\mc{X}$, $V$ be as in \eqref{eq.XV}.

\item Let $\mc{N}$ denote the outer-pointing ($g$-)unit normal of $\mc{U}$, let
\begin{equation}
\label{eq.obs_combo_Gamma} \Gamma_\dagger := \partial \mc{U}_{ \tau_-, \tau_+ } \cap \mc{D}_P \cap \{ \mc{N} f_P > 0 \} \text{,} \qquad P := ( t_0, x_0 ) \text{,}
\end{equation}
and let $\mc{Y}_\dagger$ be a neighborhood of $\bar{\Gamma}_\dagger$ in $\partial \mc{U}$.
\end{itemize}
Then, for any solution $\phi \in C^2 ( \mc{U} ) \cap C^1 ( \bar{\mc{U}} )$ of \eqref{eq.linear_wave} that also satisfies $\phi |_{ \partial \mc{U}_{ \tau_-, \tau_+ } \cap \mc{D}_P } = 0$,
\begin{equation}
\label{eq.obs_combo} \int_{ \mc{U}_{ \tau_\pm } } ( | \nabla_{ t, x } \phi |^2 + \phi^2 ) \lesssim_{ \mc{U}, P, \tau_\pm, V, \mc{X}, \mc{Y}_\dagger } \int_{ \mc{Y}_\dagger } | \mc{N} \phi |^2 \text{.}
\end{equation}
\end{theorem}

\begin{remark}
See Figure \ref{fig.intro_obs_main} for visual depictions of the observation region in Theorem \ref{thm.obs_combo}.
\end{remark}

\begin{proof}
We divide the proof into cases depending on whether $P$, as defined in \eqref{eq.obs_combo_Gamma}, lies in $\bar{\mc{U}}$.
First, if $P \not\in \bar{\mc{U}}$, we apply Theorem \ref{thm.obs_ext} to $\mc{U}$, $P$, and some $0 < \delta \ll 1$ to be fixed later; this yields
\begin{equation}
\label{eql.obs_combo_1} \int_{ \mc{U}_{ t (P) } } ( | \nabla_{ t, x } \phi |^2 + \phi^2 ) \lesssim_{ \mc{U}, P, \tau_\pm, V, \mc{X}, \delta } \int_{ \Gamma_+ } | \mc{N} \phi |^2 \text{,}
\end{equation}
with $\Gamma_+$ as in \eqref{eq.obs_ext_delta}.
By continuity and \eqref{eq.obs_ext_MR}, taking $\delta$ arbitrarily small makes $\mc{S}$ (as defined in \eqref{eq.obs_ext_delta}) arbitrarily close to $\mc{N} f_P$.
Thus, by taking $0 < \delta \ll 1$ sufficiently small (depending on $\mc{Y}_\dagger$), we have that $\Gamma_+ \subseteq \mc{Y}_\dagger$.
Recalling \eqref{eq.energy_est} and \eqref{eql.obs_combo_1} then yields, as desired,
\[
\int_{ \mc{U}_{ \tau_\pm } } ( | \nabla_{ t, x } \phi |^2 + \phi^2 ) \lesssim_{ \mc{U}, \tau_\pm, V, \mc{X} } \int_{ \mc{U}_{ t (P) } } ( | \nabla_{ t, x } \phi |^2 + \phi^2 ) \lesssim_{ \mc{U}, P, \tau_\pm, V, \mc{X}, \mc{Y}_\dagger } \int_{ \mc{Y}_\dagger } | \mc{N} \phi |^2 \text{.}
\]

Next, if $P \in \partial \mc{U}$, then we apply Theorem \ref{thm.obs_ext} to a point $P'$ near $P$ that is not in $\bar{\mc{U}}$.
Again, by continuity, if $P'$ is sufficiently near $P$ and $\delta$ is sufficiently small, then the resulting control region $\Gamma_+$ (now associated with $P'$) lies inside $\mc{Y}_\dagger$, and \eqref{eq.energy_est} yields the desired bound:
\[
\int_{ \mc{U}_{ \tau_\pm } } ( | \nabla_{ t, x } \phi |^2 + \phi^2 ) \lesssim_{ \mc{U}, \tau_\pm, V, \mc{X} } \int_{ \mc{U}_{ t (P') } } ( | \nabla_{ t, x } \phi |^2 + \phi^2 ) \lesssim_{ \mc{U}, P, \tau_\pm, V, \mc{X}, \mc{Y}^\dagger } \int_{ \mc{Y}_\dagger } | \mc{N} \phi |^2 \text{.}
\]

Finally, when $P \in \mc{U}$, we apply Theorem \ref{thm.obs_int}, with respect to some $0 < \delta \ll 1$ and distinct points $P_1, P_2 \in \mc{U}$ with $t ( P ) = t ( P_1 ) = t ( P_2 )$.
By continuity, if $\delta$ is sufficiently small, and if $P_1$, $P_2$ are sufficiently close to $P$, then $\Gamma_+^1 \cup \Gamma_+^2$ (as defined in \eqref{eq.obs_int_delta}) lies in $\mc{Y}_\dagger$, and the result follows from \eqref{eq.energy_est} and \eqref{eq.obs_int_est} in a manner analogous to the preceding two cases.
\end{proof}

\section{Consequences and Applicatons} \label{sec.app}

In this section, we provide further discussions relating to the observability results in Section \ref{sec.obs}.

\subsection{The Case $n = 1$} \label{sec.app_1d}

The first objective is to take a closer look at GTCs in one spatial dimension.
Here, we extend known results in the literature; we show, in full generality, that one can recover observability for linear waves on time-dependent domains up to the optimal required timespan.

More specifically, throughout this subsection, we will study the following setting:

\begin{definition} \label{def.gtc_1d}
Fix two curves $\ell_1$ and $\ell_2$ in $\R^{1+1}$, parametrized as
\begin{equation}
\label{eq.gtc_1d_ell} \ell_1 ( \tau ) := ( \tau, \lambda_1 ( \tau ) ) \text{,} \qquad \ell_2 ( \tau ) := ( \tau, \lambda_2 ( \tau ) ) \text{,}
\end{equation}
where $\lambda_1, \lambda_2: \R \rightarrow \R$ are smooth and satisfy, for all $\tau \in \R$,
\begin{equation}
\label{eq.gtc_1d_lambda} \lambda_1 ( \tau ) < \lambda_2 ( \tau ) \text{,} \qquad | \lambda_1' ( \tau ) | < 1 \text{,} \qquad | \lambda_2' ( \tau ) | < 1 \text{.}
\end{equation}
We now let $\mc{U}^\ell$ be the region bounded by $\ell_1$ and $\ell_2$, that is,
\begin{equation}
\label{eq.gtc_1d_U} \mc{U}^\ell := \{ ( \tau, y ) \in \R^{1+1} \mid \lambda_1 ( \tau ) < y < \lambda_2 ( \tau ) \} \text{,} \qquad \partial \mc{U}^\ell = \ell_1 \cup \ell_2 \text{.}
\end{equation}
\end{definition}

In particular, the second and third conditions in \eqref{eq.gtc_1d_lambda} imply that $\ell_1$ and $\ell_2$ are timelike, and hence $\partial \mc{U}^\ell$ is timelike as well.
We also observe the following basic facts:

\begin{proposition} \label{thm.gtc_1d}
Let the region $\mc{U}^\ell \subseteq \R^{1+1}$ be as in Definition \ref{def.gtc_1d}.
Then, $\mc{U}^\ell$ is a GTC, and the outward-pointing ($g$-)unit normal $\mc{N}$ to $\mc{U}^\ell$ satisfies
\begin{equation}
\label{eq.gtc_1d_normal} \left[ 1 - | \lambda_1' ( \tau ) |^2 \right]^\frac{1}{2} \mc{N} |_{ \ell_1 ( \tau ) } = - ( \lambda_1' ( \tau ), 1 ) \text{,} \qquad \left[ 1 - | \lambda_2' ( \tau ) |^2 \right]^\frac{1}{2} \mc{N} |_{ \ell_2 } ( \tau, \lambda_2 ( \tau ) ) = ( \lambda_2' ( \tau ), 1 ) \text{.}
\end{equation}
\end{proposition}

\subsubsection{One-Sided Observation}

We first consider the case in which we observe only on $\ell_2$.
(In terms of controllability, this corresponds to the problem of imposing a control only on $\ell_2$.)

\begin{theorem} \label{thm.obs_ext_1d}
Let $\mc{U}^\ell$ be as in Definition \ref{def.gtc_1d}, and fix $\tau_\pm \in \R$.
Consider the setting of Problem \ref{prb.linear_wave}, in the case $n = 1$, and let $\mc{X}$, $V$ be as in \eqref{eq.XV}.
In addition, assume:
\begin{itemize}
\item There exists $T_- > 0$ such that the forward, leftward null ray emanating from the point $( \tau_-, \lambda_2 ( \tau_- ) ) \in \ell_2$ intersects $\ell_1$ at time $\tau_- + T_-$, that is,
\begin{equation}
\label{eq.obs_ext_1d_ass1} \lambda_2 ( \tau_- ) - T_- = \lambda_1 ( \tau_- + T_- ) \text{.}
\end{equation}

\item There exists $T_+ > 0$ such that the forward, rightward null ray emanating from the point $( \tau_- + T_-, \lambda_1 ( \tau_- + T_- ) ) \in \ell_1$ hits $\ell_2$ at time $\tau_- + T_- + T_+$, that is,
\begin{equation}
\label{eq.obs_ext_1d_ass2} \lambda_1 ( \tau_- + T_- ) + T_+ = \lambda_2 ( \tau_- + T_- + T_+ ) \text{.}
\end{equation}

\item The following relation holds:
\begin{equation}
\label{eq.obs_ext_1d_ass3} \tau_+ > \tau_- + T_- + T_+ \text{.}
\end{equation}
\end{itemize}
Then, for any solution $\phi \in C^2 ( \mc{U}^\ell ) \cap C^1 ( \bar{\mc{U}}^\ell )$ of \eqref{eq.linear_wave} that also satisfies $\phi |_{ \ell_1 \cup \ell_2 } = 0$, we have
\begin{equation}
\label{eq.obs_ext_1d} \int_{ \mc{U}^\ell_{ \tau_\pm } } ( | \nabla_{ t, x } \phi |^2 + \phi^2 ) \lesssim_{ \ell_1, \ell_2, \tau_\pm, V, \mc{X} } \int_{ \ell_2 \cap \{ \tau_- < t < \tau_+ \} } | \partial_x \phi |^2 \text{.}
\end{equation}
\end{theorem}

\begin{remark}
Note that the only possibility of the assumptions \eqref{eq.obs_ext_1d_ass1}, \eqref{eq.obs_ext_1d_ass2} failing is if either $\ell_1$ or $\ell_2$ very quickly becomes asymptotically null in the future of $\{ t = \tau_- \}$.
\end{remark}

\begin{remark}
Theorem \ref{thm.obs_ext_1d} states that, with initial data at $t = \tau_-$, we have observability from $\ell_2$ if the timespan is strictly greater than $T_+ + T_-$.
Moreover, this timespan $T_+ + T_-$ is precisely what is required by the geometric control condition and hence is optimal.
\end{remark}

\begin{proof}
We begin by fixing $P = ( t_0, x_0 ) \in \R^{1+1}$ satisfying
\begin{equation}
\label{eql.obs_ext_1d_P} 0 < t_0 - ( \tau_- + T_- ) \ll_{ \tau_+ } 1 \text{,} \qquad x_0 := \lambda_2 ( \tau_- ) - ( t_0 - \tau_- ) \text{.}
\end{equation}
Note that $P$ lies on the forward, leftward null line from $( \tau_-, \lambda_2 ( \tau_- ) ) \in \ell_2$ and is slightly to the left of $\ell_1$.
By continuity, along with \eqref{eq.obs_ext_1d_ass1} and \eqref{eq.obs_ext_1d_ass2}, we see that as long as $t_0 - ( \tau_- + T_- )$ is sufficiently small, the forward, rightward null line from $P$ intersects $\ell_2$ before time $t = \tau_+$.

We now apply Theorem \ref{thm.obs_angle}, with $\mc{U} = \mc{U}^\ell$ and $P \in \R^{1+1} \setminus \bar{\mc{U}}^\ell$ as above, and with $0 < \delta \ll 1$ sufficiently small.
In particular, we let $\nu$, $\nu^t$, and $\theta_{ \nu, P }$ be as in the statement of Theorem \ref{thm.obs_angle} (see also Proposition \ref{thm.obs_calc}), again with the above $\mc{U}^\ell$ and $P$.
Observe, from \eqref{eq.gtc_1d_normal}, the following:
\begin{itemize}
\item On each $( \tau, y ) \in \mc{D}_P \cap \ell_1$, we have $\nu ( \tau, y ) < 0$ and $x_P ( \tau, y ) > 0$.

\item On each $( \tau, y ) \in \mc{D}_P \cap \ell_2$, we have $\nu ( \tau, y ) > 0$ and $x_P ( \tau, y ) > 0$.
\end{itemize}
In particular, this implies:
\begin{equation}
\label{eql.obs_ext_1d_0} \cos \theta_{ \nu, P } |_{ \mc{D}_P \cap \ell_1 } \equiv -1 < - \frac{ | t_P \nu^t | }{ r_P \sqrt{ 1 + ( \nu^t )^2 } } \text{,} \qquad \cos \theta_{ \nu, P } |_{ \mc{D}_P \cap \ell_2 } \equiv 1 > \frac{ | t_P \nu^t | }{ r_P \sqrt{ 1 + ( \nu^t )^2 } } \text{.}
\end{equation}
With $\Gamma_{ P, \delta }$ as in \eqref{eq.obs_angle_Gamma}, we see from \eqref{eq.obs_ext_1d_ass1}--\eqref{eq.obs_ext_1d_ass3}, \eqref{eql.obs_ext_1d_0}, and our choices of $\delta$ and $P$ that
\begin{equation}
\label{eql.obs_ext_1d_1} \Gamma_{ P, \delta } = \mc{D}_P \cap \ell_2 \subseteq \ell_2 \cap \{ \tau_- < t < \tau_+ \} \text{.}
\end{equation}

Finally, since the Dirichlet boundary condition $\phi |_{ \ell_1 \cup \ell_2 } \equiv 0$ implies
\[
| \mc{N} \phi | \lesssim | \partial_x \phi | \text{,}
\]
the observability estimate \eqref{eq.obs_angle_ext} along with \eqref{eql.obs_ext_1d_1} imply \eqref{eq.obs_ext_1d}, as desired.
\end{proof}

\subsubsection{Two-Sided Observation}

Next, we consider the case in which both $\ell_1$ and $\ell_2$ are observed.
(In particular, this corresponds to the problem of imposing controls on both $\ell_1$ and $\ell_2$.)

\begin{theorem} \label{thm.obs_int_1d}
Let $\mc{U}^\ell$ be as in Definition \ref{def.gtc_1d}, and fix $\tau_-, \tau_{+, 1}, \tau_{+, 2} \in \R$.
Consider the setting of Problem \ref{prb.linear_wave}, in the case $n = 1$, and let $\mc{X}$, $V$ be as in \eqref{eq.XV}.
In addition, assume:
\begin{itemize}
\item There exists $T_1 > 0$ such that the forward, leftward null ray emanating from the point $( \tau_-, \lambda_2 ( \tau_- ) ) \in \ell_2$ intersects $\ell_1$ at time $\tau_- + T_1$, that is,
\begin{equation}
\label{eq.obs_int_1d_ass1} \lambda_2 ( \tau_- ) - T_1 = \lambda_1 ( \tau_- + T_1 ) \text{.}
\end{equation}

\item There exists $T_2 > 0$ such that the forward, rightward null ray emanating from the point $( \tau_-, \lambda_1 ( \tau_- ) ) \in \ell_1$ hits $\ell_2$ at the time $\tau_- + T_2$, i.e.,
\begin{equation}
\label{eq.obs_int_1d_ass2} \lambda_1 ( \tau_- ) + T_2 = \lambda_2 ( \tau_- + T_2 ) \text{,}
\end{equation}

\item The following relations hold:
\begin{equation}
\label{eq.obs_int_1d_ass3} \tau_{+, 1} > \tau_- + T_1 \text{,} \qquad \tau_{+, 2} > \tau_- + T_2 \text{.}
\end{equation}
\end{itemize}
Then, for any solution $\phi \in C^2 ( \mc{U}^\ell ) \cap C^1 ( \bar{\mc{U}}^\ell )$ of \eqref{eq.linear_wave} that also satisfies $\phi |_{ \ell_1 \cup \ell_2 } = 0$, we have
\begin{equation}
\label{eq.obs_int_1d} \int_{ \mc{U}^\ell_{ \tau_\pm } } ( | \nabla_{ t, x } \phi |^2 + \phi^2 ) \lesssim_{ \ell_1, \ell_2, \tau_\pm, V, \mc{X} } \sum_{ i = 1 }^2 \int_{ \ell_i \cap \{ \tau_- < t < \tau_{+, i} \} } | \partial_x \phi |^2 \text{.}
\end{equation}
\end{theorem}

\begin{remark}
Again, like in the preceding one-sided setting, the timespan $\max ( T_1, T_2 )$ implied by Theorem \ref{thm.obs_int_1d} exactly matches the sharp value dictated by the geometric control condition.
\end{remark}

\begin{proof}
Let $t_0' \in \R$ be such that
\[
\lambda_2 ( \tau_- ) - ( t_0' - \tau_- ) = \lambda_1 ( \tau_- ) + ( t_0' - \tau_- ) \text{,}
\]
In other words, $t_0'$ is the time at which the forward, leftward null ray from $( \tau_-, \lambda_2 ( \tau_- ) ) \in \ell_2$ and the forward, rightward null ray from $( \tau_-, \lambda_1 ( \tau_- ) ) \in \ell_1$ intersect.
We now set the point $P$ to be slightly above this intersection point of the two null rays described above:
\begin{equation}
\label{eql.obs_int_1d_P} P := ( t_0, x_0 ) := ( t_0' + d, \lambda_2 ( \tau_- ) - ( t_0' - \tau_- ) ) \text{,} \qquad 0 < d \ll 1 \text{.}
\end{equation}
By continuity, choosing $d$ sufficiently small, we have $P \in \mc{U}^\ell$, and:
\begin{itemize}
\item The forward, leftward null ray from $P$ hits $\ell_1$ before time $\tau_{+, 1}$.

\item The forward, rightward null ray from $P$ hits $\ell_2$ before time $\tau_{+, 2}$.
\end{itemize}

We now apply Theorem \ref{thm.obs_angle}, with $\mc{U} = \mc{U}^\ell$ and $P$ as above, and with $0 < \delta \ll 1$.
Similar to the proof of Theorem \ref{thm.obs_ext_1d}, letting $\theta_{ \nu, P }$ be as in the statement of Theorem \ref{thm.obs_angle}, we see that
\begin{equation}
\label{eql.obs_int_1d_0} \cos \theta_{ \nu, P } |_{ \mc{D}_P \cap \ell_1 } \equiv 1 \text{,} \qquad \cos \theta_{ \nu, P } |_{ \mc{D}_P \cap \ell_2 } \equiv 1 \text{,}
\end{equation}
Therefore, we obtain from \eqref{eq.obs_int_1d_ass1}--\eqref{eq.obs_int_1d_ass3}, \eqref{eql.obs_int_1d_0}, and our choices of $\delta$ and $P$ that
\begin{equation}
\label{eql.obs_int_1d_1} \Gamma_{ P, \delta } = \mc{D}_P \cap ( \ell_1 \cup \ell_2 ) \subseteq \bigcup_{ i = 1 }^2 ( \ell_i \cap \{ \tau_- + \varepsilon < t < \tau_{+, i} - \varepsilon \} ) \text{,}
\end{equation}
for some $\varepsilon > 0$, with $\Gamma_{ P, \delta }$ as in \eqref{eq.obs_angle_Gamma}.
Finally, note that the desired observation region
\[
( \ell_1 \cap \{ \tau_- < t < \tau_{+, 1} \} ) \cup ( \ell_2 \cap \{ \tau_- < t < \tau_{+, 2} \} )
\]
is an open subset of $\partial \mc{U}$ containing $\bar{\Gamma}_{ P, \delta }$.
Thus, \eqref{eq.obs_int_1d} follows from \eqref{eq.obs_angle_int} and the above.
\end{proof}

\subsubsection{Linear $\ell_1$ and $\ell_2$}

We now look at special cases of Theorems \ref{thm.obs_ext_1d} and \ref{thm.obs_int_1d}, in which $\ell_1$ and $\ell_2$ are straight lines.
This directly extends results in \cite{cui_jiang_wang:control_wave_fec, sengou:obs_control_wave, sengou:obs_control_wave2, sun_li_lu:control_wave_moving} to general linear waves.
In particular, we explicitly recover the optimal timespans for both the one-sided and two-sided problems.

We begin with the case in which $\ell_1$ and $\ell_2$ are moving apart from each other:

\begin{corollary} \label{thm.obs_lines_1d_plus}
Let $\mc{U}^\ell$ be as in Definition \ref{def.gtc_1d}, and fix $\tau_- > 0$.
Consider the setting of Problem \ref{prb.linear_wave}, in the case $n = 1$, and let $\mc{X}$, $V$ be as in \eqref{eq.XV}.
In addition:
\begin{itemize}
\item Fix $-1 < h_1 < h_2 < 1$, and assume $\ell_1$ and $\ell_2$ satisfy\footnote{Recall $\lambda_1$ and $\lambda_2$ are related to $\ell_1$ and $\ell_2$, respectively, via \eqref{eq.gtc_1d_lambda}.}
\begin{equation}
\label{eq.obs_lines_1d} \lambda_1 ( \tau ) = h_1 \tau \text{,} \qquad \lambda_2 ( \tau ) = h_2 \tau \text{.}
\end{equation}

\item Define the ``optimal timespans"
\begin{equation}
\label{eq.obs_lines_1d_plus} T := \frac{ 2 ( h_2 - h_1 ) \tau_- }{ ( 1 + h_1 ) ( 1 - h_2 ) } \text{,} \qquad T_1 := \frac{ ( h_2 - h_1 ) \tau_- }{ 1 + h_1 } \text{,} \qquad T_2 := \frac{ ( h_2 - h_1 ) \tau_- }{ 1 - h_2 } \text{.}
\end{equation}
\end{itemize}
Then, for any solution $\phi \in C^2 ( \mc{U}^\ell ) \cap C^1 ( \bar{\mc{U}}^\ell )$ of \eqref{eq.linear_wave} that also satisfies $\phi |_{ \ell_1 \cup \ell_2 } = 0$:
\begin{itemize}
\item For any $\tau_+ > \tau_- + T$, the observability inequality \eqref{eq.obs_ext_1d} holds for the above $\phi$ and $\tau_\pm$.

\item For any $\tau_{+, 1} > \tau_- + T_1$ and $\tau_{+, 2} > \tau_- + T_2$, the observability estimate \eqref{eq.obs_int_1d} holds.
\end{itemize}
\end{corollary}

\begin{proof}
This is a consequence of Theorems \ref{thm.obs_ext_1d} and \ref{thm.obs_int_1d}, along with the following observations:
\begin{itemize}
\item The assumptions \eqref{eq.obs_ext_1d_ass1} and \eqref{eq.obs_ext_1d_ass2} hold, with
\[
T_- = \frac{ ( h_2 - h_1 ) \tau_- }{ 1 + h_1 } \text{,} \qquad T_+ = \frac{ ( h_2 - h_1 ) ( 1 + h_2 ) \tau_- }{ ( 1 + h_1 ) ( 1 - h_2 ) } \text{,} \qquad T = T_- + T_+ \text{.}
\]

\item The assumptions \eqref{eq.obs_int_1d_ass1} and \eqref{eq.obs_int_1d_ass2} hold, with
\[
T_1 = \frac{ ( h_2 - h_1 ) \tau_- }{ 1 + h_1 } \text{,} \qquad T_2 = \frac{ ( h_2 - h_1 ) \tau_- }{ 1 - h_2 } \text{.} \qedhere
\]
\end{itemize}
\end{proof}

\begin{remark}
Suppose, in addition to the setting of Corollary \ref{thm.obs_lines_1d_plus}, that $\ell_1$ is a vertical line:
\[
0 = h_1 < h_2 < 1 \text{.}
\]
Then, the optimal times for one-sided and two-sided observability, respectively, reduce to
\[
T = \frac{ 2 h_2 \tau_- }{ 1 - h_2 } \text{,} \qquad \max ( T_1, T_2 ) = \frac{ h_2 \tau_- }{ 1 - h_2 } \text{.}
\]
\end{remark}

For completeness, we also consider the case when $\ell_1$ and $\ell_2$ are moving toward each other:

\begin{corollary} \label{thm.obs_lines_1d_minus}
Let $\mc{U}^\ell$ be as in Definition \ref{def.gtc_1d}, and fix $\tau_- < 0$.
Consider the setting of Problem \ref{prb.linear_wave}, in the case $n = 1$, and let $\mc{X}$, $V$ be as in \eqref{eq.XV}.
In addition:
\begin{itemize}
\item Fix $-1 < h_2 < h_1 < 1$, and assume $\ell_1$ and $\ell_2$ satisfy \eqref{eq.obs_lines_1d}.

\item Define the ``optimal timespans"
\begin{equation}
\label{eq.obs_lines_1d_minus} T := \frac{ 2 ( h_1 - h_2 ) | \tau_- | }{ ( 1 + h_1 ) ( 1 - h_2 ) } \text{,} \qquad T_1 := \frac{ ( h_1 - h_2 ) | \tau_- | }{ 1 + h_1 } \text{,} \qquad T_2 := \frac{ ( h_1 - h_2 ) | \tau_- | }{ 1 - h_2 } \text{.}
\end{equation}
\end{itemize}
Then, for any solution $\phi \in C^2 ( \mc{U}^\ell ) \cap C^1 ( \bar{\mc{U}}^\ell )$ of \eqref{eq.linear_wave} that also satisfies $\phi |_{ \ell_1 \cup \ell_2 } = 0$:
\begin{itemize}
\item For any $\tau_+ > \tau_- + T$, the observability inequality \eqref{eq.obs_ext_1d} holds for the above $\phi$ and $\tau_\pm$.

\item For any $\tau_{+, 1} > \tau_- + T_1$ and $\tau_{+, 2} > \tau_- + T_2$, the observability estimate \eqref{eq.obs_int_1d} holds.
\end{itemize}
\end{corollary}

\begin{proof}
This is completely analogous to the proof of Corollary \ref{thm.obs_lines_1d_plus}.
\end{proof}

\begin{remark}
Again, let us assume, on top of Corollary \ref{thm.obs_lines_1d_minus}, that $\ell_1$ is a vertical line:
\[
-1 < h_2 < h_1 = 0 \text{.}
\]
Then, the optimal times for one-sided and two-sided observability, respectively, reduce to
\[
T = \frac{ 2 | h_2 | | \tau_- | }{ 1 - h_2 } \text{,} \qquad \max ( T_1, T_2 ) = | h_2 | \tau_- \text{.}
\]
\end{remark}

\subsection{Exact Controllability} \label{sec.app_control}

It is well known that observability inequalities are necessary for establishing exact controllability properties; see the discussion in Section \ref{sec.intro_bg}.
Here, we briefly describe this connection in the context of time-dependent domains with moving boundaries.

\subsubsection{Well-Posedness}

The first step in this discussion is to precisely define the relevant spaces and norms.
In the subsequent definitions, we accomplish this in the context of GTCs:

\begin{definition} \label{def.norm_leb}
Let $\mc{M}$ be a submanifold of $( \R^{1+n}, g )$, and suppose its induced metric is either Riemannian or Lorentzian.
We then define the following norms and spaces:
\begin{itemize}
\item Let $C^\infty_0 ( \mc{M} )$ be the space of compactly supported smooth functions on $\mc{M}$.

\item For any $\phi \in C^\infty_0 ( \mc{M} )$, we define
\begin{equation}
\label{eq.norm_cs_L2} \| \phi \|_{ L^2 ( \mc{M} ) }^2 := \int_\mc{M} | \phi |^2 \text{,}
\end{equation}
where the integral is defined with respect to the metric induced by $g$.

\item Let $L^2 ( \mc{M} )$ denote the (Hilbert space) completion of $C^\infty_0 ( \mc{M} )$ with respect to \eqref{eq.norm_cs_L2}.
\end{itemize}
\end{definition}

In order to state the usual well-posedness theorems for wave equations, we will need to define additional functional spaces on cross-sections of GTCs.

\begin{definition} \label{def.norms_cs}
Let $\mc{U}$ be a GTC, and let $\mc{V}$ be a cross-section of $\mc{U}$.
\begin{itemize}
\item For any $\phi \in C^\infty_0 ( \mc{V} )$, we define
\begin{equation}
\label{eq.norm_cs_H1} \| \phi \|_{ H^1 ( \mc{V} ) }^2 := \int_\mc{V} \gamma ( D \phi, D \phi ) + \int_\mc{V} \phi^2 \text{,}
\end{equation}
where $\gamma$ is the (Riemannian) metric on $\mc{V}$ induced by $g$, where the integrals on the right-hand side are with respect to $\gamma$, and where $D \phi$ denotes $\gamma$-gradient of $\phi$.

\item Let $H^1_0 ( \mc{V} )$ denote the (Hilbert space) completion of $C^\infty_0 ( \mc{V} )$ with respect to \eqref{eq.norm_cs_H1}.

\item Let $H^{-1} ( \mc{V} )$ denote the (Hilbert) dual space of $H^1_0 ( \mc{V} )$.
\end{itemize}
\end{definition}

We now state the standard existence and uniqueness results for the linear wave equations of Problem \ref{prb.linear_wave} that will be relevant to Dirichlet boundary controllability (see also \cite{ev:pde, lionj_mage:bvp1}).

\begin{theorem} \label{thm.weak_sol}
Let $\mc{U}$ be a GTC in $\R^{1+n}$, let $Z$ be a generator of $\mc{U}$, let $\mc{N}$ be the outer ($g$-)unit normal of $\mc{U}$, and let $\mc{V}$ be a cross-section of $\mc{U}$.
Also, consider the setting of Problem \ref{prb.linear_wave}.
Then:
\begin{itemize}
\item Given any initial data $( \phi_0, \phi_1 ) \in H^1_0 ( \mc{V} ) \times L^2 ( \mc{V} )$, there exists a unique $\phi \in L^2_\mathrm{loc} ( \mc{U} )$ that solves the wave equation \eqref{eq.linear_wave} and satisfies, in a trace sense,
\begin{equation}
\label{eq.weak_sol_data} ( \phi, Z \phi ) |_{ \mc{V} } = ( \phi_0, \phi_1 ) \text{,} \qquad \phi |_{ \partial \mc{U} } = 0 \text{.}
\end{equation}
Furthermore, for any cross-sections $\mc{V}_\pm$ such that $\mc{V}_- < \mc{V} < \mc{V}_+$,
\begin{align}
\label{eq.weak_sol_energy} \| \phi \|_{ H^1 ( \mc{V}_+ ) }^2 + \| Z \phi \|_{ L^2 ( \mc{V}_+ ) }^2 + \| \mc{N} \phi \|_{ L^2 ( \partial \mc{U} ( \mc{V}, \mc{V}_+ ) ) }^2 &\lesssim_{ Z, \mc{V}, \mc{V}_+ } \| \phi_0 \|_{ H^1 ( \mc{V} ) }^2 + \| \phi_1 \|_{ L^2 ( \mc{V} ) }^2 \text{,} \\
\notag \| \phi \|_{ H^1 ( \mc{V}_- ) }^2 + \| Z \phi \|_{ L^2 ( \mc{V}_- ) }^2 + \| \mc{N} \phi \|_{ L^2 ( \partial \mc{U} ( \mc{V}_-, \mc{V} ) ) }^2 &\lesssim_{ Z, \mc{V}, \mc{V}_- } \| \phi_0 \|_{ H^1 ( \mc{V} ) }^2 + \| \phi_1 \|_{ L^2 ( \mc{V} ) }^2 \text{.}
\end{align}

\item Given any initial data $( \phi_0, \phi_1 ) \in L^2 ( \mc{V} ) \times H^{-1} ( \mc{V} )$ and boundary data $\phi_b \in L^2 ( \partial \mc{U} )$, there is a unique solution $\phi \in L^2_\mathrm{loc} ( \mc{U} )$ of \eqref{eq.linear_wave} that satisfies, in a trace sense,
\begin{equation}
\label{eq.dual_sol_data} ( \phi, Z \phi ) |_{ \mc{V} } = ( \phi_0, \phi_1 ) \text{,} \qquad \phi |_{ \partial \mc{U} } = \phi_b \text{.}
\end{equation}
Furthermore, for any cross-sections $\mc{V}_\pm$ such that $\mc{V}_- < \mc{V} < \mc{V}_+$,
\begin{align}
\label{eq.dual_sol_energy} \| \phi \|_{ L^2 ( \mc{V}_+ ) }^2 + \| Z \phi \|_{ H^{-1} ( \mc{V}_+ ) }^2 &\lesssim_{ Z, \mc{V}, \mc{V}_+ } \| \phi_0 \|_{ L^2 ( \mc{V} ) }^2 + \| \phi_1 \|_{ H^{-1} ( \mc{V} ) }^2 + \| \phi_b \|_{ L^2 ( \partial \mc{U} ( \mc{V}, \mc{V}_+ ) ) }^2 \text{,} \\
\notag \| \phi \|_{ L^2 ( \mc{V}_- ) }^2 + \| Z \phi \|_{ H^{-1} ( \mc{V}_- ) }^2 &\lesssim_{ Z, \mc{V}, \mc{V}_- } \| \phi_0 \|_{ L^2 ( \mc{V} ) }^2 + \| \phi_1 \|_{ H^{-1} ( \mc{V} ) }^2 + \| \phi_b \|_{ L^2 ( \partial \mc{U} ( \mc{V}_-, \mc{V} ) ) }^2 \text{.}
\end{align}
\end{itemize}
\end{theorem}

\begin{remark}
A technical issue inherent to Theorem \ref{thm.weak_sol} (and more generally to the theory of hyperbolic PDEs) is establishing the precise sense that $\phi$ is a solution of Problem \ref{prb.linear_wave}.
\begin{itemize}
\item When $( \phi_0, \phi_1 ) \in H^1_0 ( \mc{V} ) \times L^2 ( \mc{V} )$, the solution $\phi$ can be interpreted as a \emph{weak solution} of \eqref{eq.linear_wave}, via integrations by parts and the Hilbert space structures of the spaces $\mc{H}^1_0 ( \mc{V}_\pm )$.
This is described, for a slightly simpler class of hyperbolic PDEs, in \cite{ev:pde}.

\item When $( \phi_0, \phi_1 ) \in L^2 ( \mc{V} ) \times H^{-1} ( \mc{V} )$, one can make sense of the ``transposition" solution $\phi$ via the above theory of weak solutions and a duality argument; see, for instance, \cite{lionj_mage:bvp1}.

\item More generally, as long as we assume $\mc{X}$ and $\mc{V}$ are smooth, then $\phi$, in all the above, can also be interpreted as a distributional solution of \eqref{eq.linear_wave}.
\end{itemize}
\end{remark}

\subsubsection{Controllability and Observability}

Using the well-posedness results of Theorem \ref{thm.weak_sol}, we can now give a precise statement of the exact controllability problem we will consider:

\begin{definition} \label{def.control}
Let $\mc{U}$ be a GTC in $\R^{1+n}$, let $Z$ be a generator of $\mc{U}$, and fix cross-sections $\mc{V}_\pm$ of $\mc{U}$ with $\mc{V}_- < \mc{V}_+$.
Moreover, we consider the setting of Problem \ref{prb.linear_wave} on $\mc{U}$.
\begin{itemize}
\item The wave equation \eqref{eq.linear_wave} is \emph{exactly} (\emph{Dirichlet boundary}) \emph{controllable} on $\mc{U} ( \mc{V}_-, \mc{V}_+ )$, with control on some open $\Gamma \subseteq \partial \mc{U} ( \mc{V}_-, \mc{V}_+ )$, iff given any initial and final data,
\begin{equation}
\label{eq.control_cauchy_data} ( \phi^\pm_0, \phi^\pm_1 ) \in L^2 ( \mc{V}_\pm ) \times H^{-1} ( \mc{V}_\pm ) \text{,}
\end{equation}
there exists $\phi_b \in L^2 ( \partial \mc{U} )$, supported in $\Gamma$, such that the solution of \eqref{eq.linear_wave} satisfying\footnote{This solution exists and is unique due to the second part of Theorem \ref{thm.weak_sol}.}
\begin{equation}
\label{eq.control_existence} ( \phi, Z \phi ) |_{ \mc{V}_- } = ( \phi^-_0, \phi^-_1 ) \text{,} \qquad \phi |_{ \partial \mc{U} } = \phi_b \text{,}
\end{equation}
also attains the final data
\begin{equation}
\label{eq.control_goal} ( \phi, Z \phi ) |_{ \mc{V}_+ } = ( \phi^+_0, \phi^+_1 ) \text{.}
\end{equation}

\item When \eqref{eq.control_existence} and \eqref{eq.control_goal} hold, we say that $\phi_b$ drives \eqref{eq.linear_wave} from $( \phi^-_0, \phi^-_1 )$ to $( \phi^+_0, \phi^+_1 )$.
\end{itemize}
\end{definition}

The following theorem summarizes, again in the context of GTCs and Problem \ref{prb.linear_wave}, the connection, via the HUM (see \cite{lionj:ctrlstab_hum}), between observability and exact controllability.

\begin{theorem} \label{thm.control_hum}
Assume the setting of Definition \ref{def.control} and Problem \ref{prb.linear_wave}, let $\mc{N}$ denote the outer ($g$-)unit normal to $\mc{U}$, and let $\Gamma$ be an open subset of $\partial \mc{U} ( \mc{V}_-, \mc{V}_+ )$.
In addition, assume that for any $( \psi^-_0, \psi^-_1 ) \in H^1_0 ( \mc{V}_- ) \times L^2 ( \mc{V}_- )$, we have the observability inequality
\begin{equation}
\label{eq.control_hum_obs} \| \psi^-_0 \|_{ H^1 ( \mc{V}_- ) }^2 + \| \psi^-_1 \|_{ L^2 ( \mc{V}_- ) }^2 \lesssim_{ Z, \mc{V}_\pm } \| \mc{N} \psi \|_{ L^2 ( \Gamma ) }^2 \text{,}
\end{equation}
where $\psi$ denotes the solution of the adjoint problem\footnote{This solution exists and is unique due to the first part of Theorem \ref{thm.weak_sol}.} 
\begin{align}
\label{eq.control_hum_adj} [ \Box \psi - \nabla_{ \mc{X} } \psi + ( V - \nabla_\alpha \mc{X}^\alpha ) \psi ] |_{ \mc{U} } &= 0 \text{,} \\
\notag ( \psi, Z \psi ) |_{ \mc{V}_- } &= ( \psi^-_0, \psi^-_1 ) \text{,} \\
\notag \psi |_{ \partial \mc{U} } &\equiv 0 \text{.}
\end{align}
Then, the equation \eqref{eq.linear_wave} is exactly controllable on $\mc{U} ( \mc{V}_-, \mc{V}_+ )$, with control on $\Gamma$.

Furthermore, for each $( \phi^-_0, \phi^-_1 ) \in L^2 ( \mc{V}_- ) \times H^{-1} ( \mc{V}_- )$, there exists a functional
\begin{equation}
\label{eq.control_hum_fct} \mc{J} [ \phi^-_0, \phi^-_1 ]: H^1_0 ( \mc{V}_- ) \times L^2 ( \mc{V}_- ) \rightarrow \R
\end{equation}
such that the following hold:
\begin{itemize}
\item $\mc{J} [ \phi^-_0, \phi^-_1 ]$ has a unique minimizer $( \psi^-_0, \psi^-_1 )$; moreover, the zero extension of $\mc{N} \psi |_\Gamma$ to $\partial \mc{U}$, where $\psi$ solves \eqref{eq.control_hum_adj} with the above $( \psi^-_0, \psi^-_1 )$, drives \eqref{eq.linear_wave} from $( \phi^-_0, \phi^-_1 )$ to $(0, 0)$.

\item If $\phi_b \in L^2 ( \partial \mc{U} )$ is also supported in $\Gamma$ and drives \eqref{eq.linear_wave} from $( \phi^-_0, \phi^-_1 )$ to $(0, 0)$, then
\begin{equation}
\label{eq.control_hum_min} \| \mc{N} \psi \|_{ L^2 ( \Gamma ) } \leq \| \phi_b \|_{ L^2 ( \Gamma ) } \text{.}
\end{equation}
\end{itemize}
\end{theorem}

As is standard, we can now combine Theorem \ref{thm.control_hum} with our main observability results in order to establish exact Dirichlet boundary controllability for Problem \ref{prb.linear_wave} on time-dependent domains.
Below, we demonstrate this for the unified observability estimate of Theorem \ref{thm.obs_combo}.

\begin{corollary} \label{thm.control_combo}
Assume the setting of Theorem \ref{thm.obs_combo}, that is, let $\mc{U} \subseteq \R^{1+n}$ be a GTC, and:
\begin{itemize}
\item Fix $x_0 \in \R^n$ and $\tau_\pm \in \R$ satisfying \eqref{eq.obs_combo_R}.
Also, choose $t_0 \in ( \tau_-, \tau_+ )$ so that \eqref{eq.obs_combo_t0} holds.

\item Consider the setting of Problem \ref{prb.linear_wave}, and let $\mc{X}$, $V$ be as in \eqref{eq.XV}.

\item Let $\mc{N}$ denote the outer ($g$-)unit normal of $\mc{U}$, let $\Gamma_\dagger \subseteq \partial \mc{U}_{ \tau_-, \tau_+ }$ be defined as in \eqref{eq.obs_combo_Gamma}, and let $\mc{Y}_\dagger \subseteq \partial \mc{U}_{ \tau_-, \tau_+ }$ be a neighborhood of $\bar{\Gamma}_\dagger$ in $\partial \mc{U}$.
\end{itemize}
Then, the equation \eqref{eq.linear_wave} is exactly controllable on $\mc{U}_{ \tau_-, \tau_+ }$, with control on $\mc{Y}_\dagger$.

More generally, under the above assumptions, given any cross sections $\mc{V}_\pm$ of $\mc{U}$ that satisfy
\begin{equation}
\label{eq.control_combo_gen} \mc{V}_- \subseteq I^- ( \mc{Y}_\dagger ) \text{,} \qquad \mc{V}_+ \subseteq I^+ ( \mc{Y}_\dagger ) \text{,}
\end{equation}
we have that \eqref{eq.linear_wave} is exactly controllable on $\mc{U} ( \mc{V}_-, \mc{V}_+ )$, with control on $\mc{Y}_\dagger$.
\end{corollary}

\begin{proof}
Consider a solution $\psi$ of the adjoint problem \eqref{eq.control_hum_adj}, with $\mc{V}_- := \mc{U}_{ \tau_- }$ and $Z$ being any generator of $\mc{U}$.
Note the wave equation in \eqref{eq.control_hum_adj} is of the same form as those in Problem \ref{prb.linear_wave}.
Thus, whenever $\psi$ is smooth, we can apply Theorem \ref{thm.obs_combo} to obtain the observability estimate
\begin{equation}
\label{eql.control_combo_1} \| \psi^-_0 \|_{ H^1 ( \mc{U}_{ \tau_- } ) }^2 + \| \psi^-_1 \|_{ L^2 ( \mc{U}_{ \tau_- } ) }^2 \lesssim_{ \mc{U}, ( x_0, t_0 ), \tau_\pm, V, \mc{X}, \mc{Y}_\dagger } \int_{ \mc{Y}_\dagger } | \mc{N} \psi |^2 \text{.}
\end{equation}
Moreover, a standard approximation argument combined with \eqref{eql.control_combo_1} yields the inequality
\begin{equation}
\label{eql.control_combo_2} \| \psi^-_0 \|_{ H^1 ( \mc{U}_{ \tau_- } ) }^2 + \| \psi^-_1 \|_{ L^2 ( \mc{U}_{ \tau_- } ) }^2 \lesssim_{ \mc{U}, x_0, \tau_\pm, V, \mc{X}, \mc{Y}_\dagger } \| \mc{N} \psi \|_{ L^2 ( \mc{Y}_\dagger ) }^2 \text{,}
\end{equation}
for all solutions $\psi$ generated from initial data $( \psi^-_0, \psi^-_1 ) \in H^1_0 ( \mc{U}_{ \tau_- } ) \times L^2 ( \mc{U}_{ \tau_- } )$, that is, solutions generated from the first part of Theorem \ref{thm.weak_sol}.
Applying Theorem \ref{thm.control_hum} along with \eqref{eql.control_combo_2} yields that \eqref{eq.linear_wave} is indeed exactly controllable on $\mc{U}_{ \tau_-, \tau_+ }$, with control on $\mc{Y}_\dagger$.

Finally, for general cross-sections $\mc{V}_\pm$, we combine the observability inequality \eqref{eql.control_combo_1} with the energy estimate \eqref{eq.energy_est} and an approximate argument to obtain \eqref{eq.control_hum_obs} for all solutions $\psi$ arising from initial data $( \psi^-_0, \psi^-_1 ) \in H^1_0 ( \mc{V}_- ) \times L^2 ( \mc{V}_- )$.
By Theorem \ref{thm.control_hum}, we again conclude that \eqref{eq.linear_wave} is exactly controllable on $\mc{U} ( \mc{V}_-, \mc{V}_+ )$, with control on $\mc{Y}_\dagger$.
\end{proof}

\begin{remark}
Note that Theorem \ref{thm.control_hum} and Corollary \ref{thm.control_combo} allow for geometric extensions of controllability results, where the initial and final data can be placed on cross-sections of $\mc{U}$ that are not level sets of $t$.
As before, one still needs to impose the Dirichlet boundary data on a sufficiently large portion of $\partial \mc{U}$ in order to achieve exact controllability.
\end{remark}

Finally, we note that all the previous results in one spatial dimension---Theorems \ref{thm.obs_ext_1d} and \ref{thm.obs_int_1d}, Corollaries \ref{thm.obs_lines_1d_plus} and \ref{thm.obs_lines_1d_minus}---also lead to corresponding exact controllability results.


\appendix

\if\comp1
\newpage

\section{Details and Computations}

In this appendix, we provide---for interested readers' convenience---additional proofs, computations, and details that were omitted in the main sections of this article.

\subsection{Preliminary Computations}

First, a number of computations throughout this appendix will make use of the following null coordinate computations for the warped metric:

\begin{lemma} \label{thm.met_coord_wp}
Let $\varepsilon \in \R$ and $\bar{g}$ be defined as in Definition \ref{def.wp_fct} and \ref{def.wp_met}.
Then:
\begin{itemize}
\item The nonzero components of $\bar{g}$ and $\bar{g}^{-1}$ in the null coordinates $( u, v, \omega )$ are given by
\begin{equation}
\label{eql.met_comp_wp} \bar{g}_{uv} \equiv -2 \text{,} \qquad \bar{g}_{ab} = \bar{\rho}^2 \mathring{\gamma}_{ab} \text{,} \qquad \bar{g}^{uv} \equiv - \frac{1}{2} \text{,} \qquad \bar{g}^{ab} = \bar{\rho}^{-2} \mathring{\gamma}^{ab} \text{.}
\end{equation}

\item The nonzero Christoffel symbols of $\bar{g}$ in the null coordinates $( u, v, \omega )$ are given by
\begin{align}
\label{eql.Gamma_comp_wp} \bar{\Gamma}^u_{ab} = \frac{1}{ 2 \bar{\rho} } ( 1 - 2 \varepsilon u ) \bar{g}_{ab} \text{,} &\qquad \bar{\Gamma}^v_{ab} = - \frac{1}{ 2 \bar{\rho} } ( 1 + 2 \varepsilon v ) \bar{g}_{ab} \text{,} \\
\notag \bar{\Gamma}^a_{ub} = - \frac{1}{ \bar{\rho} }  ( 1 + 2 \varepsilon v ) \delta^a_b \text{,} &\qquad \bar{\Gamma}^a_{vb} = \frac{1}{ \bar{\rho} } ( 1 - 2 \varepsilon u ) \delta^a_b \text{.}
\end{align}
\end{itemize}
\end{lemma}

\subsection{Proof of Proposition \ref{thm.f_deriv_wp}}

To derive \eqref{eq.f_grad_wp}, we use \eqref{eql.met_comp_wp} to obtain that
\[
\bar{\nabla}^\sharp f = \bar{g}^{\alpha \beta} \partial_\alpha f \partial_\beta = \bar{g}^{vu} \partial_v f \partial_u + \bar{g}^{uv} \partial_u f \partial_v = \frac{1}{2} ( u \partial_u + v \partial_v ) \text{,}
\]
and that
\[
\bar{\nabla}^\alpha f \bar{\nabla}_\beta f = 2 \bar{g}^{uv} \partial_u f \partial_v f = - (-v) (-u) = f \text{.}
\]

Next, using \eqref{eql.Gamma_comp_wp}, we obtain
\[
\bar{\nabla}_{uu} f = \partial_u \partial_u f = 0 \text{,} \qquad \bar{\nabla}_{vv} f = \partial_v \partial_v f = 0 \text{,} \qquad \bar{\nabla}_{uv} f = \partial_u \partial_v f = -1 \text{,}
\]
as well as
\[
\bar{\nabla}_{ua} f = - \bar{\Gamma}^b_{u a} \partial_b f = 0 \text{,} \qquad \bar{\nabla}_{va} f = - \bar{\Gamma}^b_{v a} \partial_b f = 0 \text{.}
\]
For the purely spherical components, \eqref{eql.Gamma_comp_wp} also implies
\begin{align*}
\bar{\nabla}_{ab} f &= - \bar{\Gamma}^u_{ab} \partial_u f - \bar{\Gamma}^v_{ab} \partial_v f \\
&= \frac{1}{ 2 \bar{\rho} } [ ( - 1 + 2 \varepsilon u ) ( -v ) + ( 1 + 2 \varepsilon v ) ( -u ) ] \bar{g}_{ab} \\
&= \frac{1}{ 2 \bar{\rho} } ( r + 4 \varepsilon f ) \bar{g}_{ab} \\
&= \left( \frac{1}{2} + \frac{ \varepsilon f }{ \bar{\rho} } \right) \bar{g}_{ab} \text{,}
\end{align*}
which completes the proof of \eqref{eq.f_hess_wp}.

Contracting $\bar{\nabla}^2 f$ and recalling \eqref{eq.f_hess_wp}, we then obtain
\[
\bar{\Box} f = 2 \bar{g}^{uv} \bar{\nabla}_{uv} f + \bar{g}^{ab} \bar{\nabla}_{ab} f = 1 + (n - 1) \left( \frac{1}{2} + \frac{ \varepsilon f }{ \rho } \right) = \frac{n + 1}{2} + \frac{ (n - 1) \varepsilon f }{ \bar{\rho} } \text{,}
\]
Furthermore, recalling the second equation in \eqref{eq.f_grad_wp},
\[
\bar{\nabla}^\alpha f \bar{\nabla}^\beta f \bar{\nabla}_{\alpha \beta} f = \frac{1}{2} \bar{\nabla}^\alpha f \bar{\nabla}_\alpha ( \bar{\nabla}^\beta f \bar{\nabla}_\beta f ) = \frac{1}{2} \bar{\nabla}^\alpha f \bar{\nabla}_\alpha f = \frac{1}{2} f \text{.}
\]
Combining the above results in \eqref{eq.f_box_wp}.

\subsection{Proof of Proposition \ref{thm.f_rho_wp}}

First, \eqref{eq.f_rho_deriv_wp} follows from a direct computation:
\begin{align*}
\partial_u \left( \frac{f}{ \bar{\rho} } \right) &= - \frac{v}{ \bar{\rho} } - \frac{f}{ \bar{\rho}^2 } \partial_u \bar{\rho} = - \frac{ v ( v - u + 2 \varepsilon f ) }{ \bar{\rho}^2 } + \frac{f}{ \bar{\rho}^2 } ( 1 + 2 \varepsilon v ) = - \frac{ v^2 }{ \bar{\rho}^2 } \text{,} \\
\partial_v \left( \frac{f}{ \bar{\rho} } \right) &= - \frac{u}{ \bar{\rho} } - \frac{f}{ \bar{\rho}^2 } \partial_v \bar{\rho} = - \frac{ u ( v - u + 2 \varepsilon f ) }{ \bar{\rho}^2 } - \frac{f}{ \bar{\rho}^2 } ( 1 - 2 \varepsilon u ) = \frac{ u^2 }{ \bar{\rho}^2 } \text{.}
\end{align*}
Taking a second derivative, we then see that
\[
\partial_u \partial_v \left( \frac{f}{ \bar{\rho} } \right) = \frac{ 2 u }{ \bar{\rho}^2 } - \frac{ 2 u^2 }{ \bar{\rho}^3 } \partial_u \bar{\rho} = \frac{ 2 u ( v - u + 2 \varepsilon f ) }{ \bar{\rho}^3 } + \frac{ 2 u^2 ( 1 + 2 \varepsilon v ) }{ \bar{\rho}^3 } = \frac{ 2 u v }{ \bar{\rho}^3 } = - \frac{ 2 f }{ \bar{\rho}^3 } \text{.}
\]
Next, for spherical derivatives, we recall \eqref{eq.rho_wp}, \eqref{eql.Gamma_comp_wp}, and the above to obtain
\begin{align*}
\bar{g}^{ab} \bar{\nabla}_{ab} \left( \frac{f}{ \bar{\rho} } \right) &= - \bar{g}^{ab} \left[ \bar{\Gamma}^u_{ab} \partial_u \left( \frac{f}{ \bar{\rho} } \right) + \bar{\Gamma}^v_{ab} \partial_v \left( \frac{f}{ \bar{\rho} } \right) \right] \\
&= - \frac{ n - 1 }{ 2 \bar{\rho} } \left[ ( 1 - 2 \varepsilon u ) \left( - \frac{ v^2 }{ \bar{\rho}^2 } \right) - ( 1 + 2 \varepsilon v ) \frac{ u^2 }{ \bar{\rho}^2 } \right] \\
&= \frac{ n - 1 }{ 2 \bar{\rho}^3 } ( v^2 + u^2 + 2 \varepsilon r f ) \\
&= \frac{ n - 1 }{ 2 \bar{\rho}^3 } ( \bar{\rho}^2 - 2 f - 2 \varepsilon f \bar{\rho} ) \text{.}
\end{align*}
The identity \eqref{eq.f_rho_box_wp} follows, since
\begin{align*}
\bar{\Box} \left( \frac{f}{ \bar{\rho} } \right) &= - \partial_u \partial_v \left( \frac{f}{ \bar{\rho} } \right) + \bar{g}^{ab} \bar{\nabla}_{ab} \left( \frac{f}{ \bar{\rho} } \right) \\
&= \frac{2 f}{ \bar{\rho}^3 } + \frac{ n - 1 }{ 2 \bar{\rho}^3 } ( \bar{\rho}^2 - 2 f - 2 \varepsilon f \bar{\rho} ) \\
&= \frac{ n - 1 }{ 2 \bar{\rho} } \left( 1 - \frac{ 2 \varepsilon f }{ \bar{\rho} } \right) - \frac{ (n - 3) f }{ \bar{\rho}^3 } \text{.}
\end{align*}

\subsection{Proof of Proposition \ref{thm.TN_wp}}

First, $N$ is normal to the level sets of $f$, since by \eqref{eq.f_grad_wp},
\[
N = f^{ - \frac{1}{2} } \bar{\nabla}^\sharp f \text{.}
\]
That $T$ is tangent to the level sets of $f$ follows from the fact that
\[
T f = \frac{1}{2} f^{ - \frac{1}{2} } [ (-u) (-v) + v (-u) ] = 0 \text{.}
\]

Observe that since $\bar{\nabla}_{u a} f = \bar{\nabla}_{v a} f = 0$ by \eqref{eq.f_hess_wp}, then
\[
\bar{\nabla}_{T a} f = \bar{\nabla}_{N a} f = 0 \text{.}
\]
For the remaining components of $\bar{\nabla}^2 f$, we apply \eqref{eq.f_hess_wp} and \eqref{eq.TN} to obtain
\begin{align*}
\bar{\nabla}_{T T} f &= \frac{1}{4} f^{-1} [ 2 (-u) v \bar{\nabla}_{u v} f ] = - \frac{1}{2} \text{,} \\
\bar{\nabla}_{N N} f &= \frac{1}{4} f^{-1} [ 2 u v \bar{\nabla}_{u v} f ] = \frac{1}{2} \text{,} \\
\bar{\nabla}_{T N} f &= \frac{1}{4} f^{-1} [ (-u) v \bar{\nabla}_{u v} f + u v \bar{\nabla}_{u v} f ] = 0 \text{,}
\end{align*}
which is precisely \eqref{eq.f_TN_wp}.

\subsection{Proof of Proposition \ref{thm.pseudoconvex_wp}}

The components of $\bar{\pi}$ can be directly computed using its definition \eqref{eq.pi_wp} along with the identities \eqref{eq.f_hess_wp} and \eqref{eq.f_TN_wp}.
First, note that
\[
\bar{\pi}_{T N} = \bar{\nabla}_{T N} f = 0 \text{,} \qquad \bar{\pi}_{T a} = \bar{\nabla}_{T a} f = 0 \text{,} \qquad \bar{\pi}_{N a} = \bar{\nabla}_{N a} f = 0 \text{.}
\]
Similarly, for the nontrivial components, we compute
\begin{align*}
\bar{\pi}_{T T} &= \bar{\nabla}_{T T} f + \bar{h} = - \frac{1}{2} + \frac{1}{2} + \frac{ \varepsilon f }{ 2 \bar{\rho} } = \frac{ \varepsilon f }{ 2 \bar{\rho} } \text{,} \\
\bar{\pi}_{N N} &= \bar{\nabla}_{N N} f - \bar{h} = \frac{1}{2} - \frac{1}{2} - \frac{ \varepsilon f }{ 2 \bar{\rho} } = - \frac{ \varepsilon f }{ 2 \bar{\rho} } \text{,} \\
\bar{\pi}_{a b} &= \bar{\nabla}_{a b} f - \bar{h} \cdot \bar{g}_{a b} = \left( \frac{1}{2} + \frac{ \varepsilon f }{ \bar{\rho} } - \frac{1}{2} - \frac{ \varepsilon f }{ 2 \bar{\rho} } \right) \bar{g}_{ab} = \frac{ \varepsilon f }{ 2 \bar{\rho} } \cdot \bar{g}_{ab} \text{.}
\end{align*}
This completes the proof of \eqref{eq.pseudoconvex_wp}.

Finally, for \eqref{eq.w_box_wp}, we simply apply \eqref{eq.f_rho_box_wp}:
\[
\bar{\Box} \bar{w} = \frac{ (n - 2) \varepsilon }{ 2 } \cdot \bar{\Box} \left( \frac{f}{ \bar{\rho} } \right) = - \frac{ (n - 2) \varepsilon }{ 2 } \left[ \frac{ (n - 3) f }{ \bar{\rho}^3 } - \frac{ n - 1 }{ 2 \bar{\rho} } \left( 1 - \frac{ 2 \varepsilon f }{ \bar{\rho} } \right) \right] \text{.}
\]

\subsection{Proof of Proposition \ref{thm.met_conf}}

The identities in \eqref{eq.rf_conf} follow from direct computations.
First,
\[
f \circ \bar{\Phi} = - ( u \circ \bar{\Phi} ) ( v \circ \bar{\Phi} ) = \frac{ - u v }{ ( 1 + \varepsilon u ) ( 1 - \varepsilon v ) } = \xi^{-1} f \text{.}
\]
For the remaining identity, we combine \eqref{eq.uv} and \eqref{eq.conformal} with the above to obtain
\begin{align*}
\rho \circ \bar{\Phi} &= \frac{ v }{ ( 1 - \varepsilon v ) } - \frac{ u }{ ( 1 + \varepsilon u ) } - \frac{ 2 \varepsilon u v }{ ( 1 + \varepsilon u ) ( 1 - \varepsilon v ) } = \xi^{-1} ( v - u ) = \xi^{-1} r \text{.}
\end{align*}

Next, we compute the derivative of $u \circ \bar{\Phi}$ of $v \circ \bar{\Phi}$:
\begin{align*}
\partial_u ( u \circ \bar{\Phi} ) &= \frac{1}{ 1 + \varepsilon u } - \frac{ \varepsilon u }{ ( 1 + \varepsilon u )^2 } = \frac{1}{ ( 1 + \varepsilon u )^2 } \text{,} \\
\partial_v ( v \circ \bar{\Phi} ) &= \frac{1}{ 1 - \varepsilon v } + \frac{ \varepsilon v }{ ( 1 - \varepsilon v )^2 } = \frac{1}{ ( 1 - \varepsilon v )^2 } \text{.}
\end{align*}
In particular, by \eqref{eq.uvf_bound}, both derivatives are positive and uniformly bounded from below.
This implies $u \circ \bar{\Phi}$ and $v \circ \bar{\Phi}$ define invertible functions of $u$ and $v$, respectively, with smooth inverses.

It remains to show the conformal identity \eqref{eq.met_conf}.
For this, we first note that the above derivatives imply that the push-forwards of the null coordinate vector fields satisfy
\[
d \bar{\Phi} ( \partial_u ) = ( 1 + \varepsilon u )^{-2} \partial_u \text{,} \qquad d \bar{\Phi} ( \partial_v ) = ( 1 - \varepsilon v )^{-2} \partial_v \text{,}
\]
and hence the corresponding pullbacks satisfy
\[
\bar{\Phi}_\ast ( d u ) = ( 1 + \varepsilon u )^{-2} du \text{,} \qquad \bar{\Phi}_\ast ( d v ) = ( 1 - \varepsilon v )^{-2} dv \text{.}
\]
(Moreover, $\bar{\Phi}$ by definition leaves the angular components unchanged.)
As a result, by \eqref{eq.rf_conf},
\begin{align*}
\bar{\Phi}_\ast \bar{g} &= \bar{\Phi}_\ast ( - 4 du dv + \bar{\rho}^2 \mathring{\gamma} ) = - \frac{ 4 du dv }{ ( 1 + \varepsilon u ) ( 1 - \varepsilon v ) } + ( \bar{\rho} \circ \bar{\Phi} )^2 \mathring{\gamma} = \xi^{-2} ( - 4 du dv + r^2 \mathring{\gamma} ) \text{,}
\end{align*}
and \eqref{eq.met_conf} now follows from \eqref{eq.mink_met_alt}.

\subsection{Proof of Proposition \ref{thm.conf_wave}}

The main step is to compute the scalar curvature $\bar{\mc{R}}$ with respect to $\bar{g}$.
Since Minkowski spacetime is flat, the standard formula for the change of scalar curvature under conformal transformations (see, for instance, \cite[Appendix D]{wald:gr}) yields that\footnote{The formula found in \cite{wald:gr} only takes into account the conformal factor $\omega$, and not the presence of the diffeomorphism $\bar{\Phi}$.
However, the formula with $\bar{\Phi}$ is similar, as one only has to appropriately map points through $\bar{\Phi}$.}
\begin{align*}
\bar{\mc{R}} \circ \bar{\Phi} &= \xi^2 [ 2 n \cdot \Box ( \log \xi ) - n ( n - 1 ) \cdot g^{ \alpha \beta } \nabla_\alpha ( \log \xi ) \nabla_\beta ( \log \xi ) ] \\
&= 2 n \xi \cdot \Box \xi - n ( n + 1 ) \cdot g^{ \alpha \beta } \nabla_\alpha \xi \nabla_\beta \xi \text{.}
\end{align*}

Recalling Lemma \ref{thm.met_coord_wp} (with $\varepsilon = 0$), we compute
\[
g^{ \alpha \beta } \nabla_\alpha \xi \nabla_\beta \xi = - \partial_u \xi \partial_v \xi = \varepsilon^2 ( 1 - \varepsilon v ) ( 1 + \varepsilon u ) = \varepsilon^2 \xi \text{,}
\]
as well as
\begin{align*}
\Box \xi &= g^{ \alpha \beta } \partial_{ \alpha \beta } \xi - g^{ \alpha \beta } \Gamma_{ \alpha \beta }^\mu \partial_\mu \xi \\
&= - \partial_{ u v } \xi - g^{ a b } \Gamma_{ a b }^u \partial_u \xi - g^{ a b } \Gamma_{ a b }^v \partial_v \xi \\
&= \varepsilon^2 - \frac{ n - 1 }{ 2 r } \cdot \varepsilon ( 1 - \varepsilon v ) - \frac{ n - 1 }{ 2 r } \cdot \varepsilon ( 1 + \varepsilon u ) \\
&= \frac{ ( n + 1 ) \varepsilon^2 }{2} - \frac{ ( n - 1 ) \varepsilon }{ r } \text{.}
\end{align*}
Thus, combining all the above with \eqref{eq.rf_conf} yields
\begin{align*}
\bar{\mc{R}} \circ \bar{\Phi} &= 2 n \xi \left[ \frac{ ( n + 1 ) \varepsilon^2 }{2} - \frac{ ( n - 1 ) \varepsilon }{ r } \right] - n ( n + 1 ) \varepsilon^2 \xi \\
&= - 2 n ( n - 1 ) \varepsilon \cdot \frac{ \xi }{ r } \\
&= - \frac{ 2 n ( n - 1 ) \varepsilon }{ \bar{\rho} \circ \bar{\Phi} } \text{.}
\end{align*}

Finally, recall that the standard formula for how the wave operator changes under conformal transformations (again, see \cite[Appendix D]{wald:gr}, for instance) is given by
\[
\left[ \bar{\Box} - \frac{n - 1}{4 n} ( \bar{\mc{R}} \circ \bar{\Phi} ) \right] \xi^{ \frac{n + 1}{2} - 1 } = \xi^{ \frac{n + 1}{2} + 1 } \Box \text{.}
\]
Substituting our identity for $\bar{\mc{R}}$ into the above results in \eqref{eq.conf_wave}.

\subsection{Proof of Proposition \ref{thm.gtc_1d}}

Clearly, $\partial \mc{U}^\ell$, a disjoint union of two smooth timelike curves $\ell_1$ and $\ell_2$, is a smooth timelike hypersurface of $\R^{1+1}$.
By \eqref{eq.gtc_1d_lambda}, we have, for any $\tau \in \R$, that
\[
\{ y \in \R \mid ( \tau, y ) \in \mc{U} \} = ( \lambda_1 ( \tau ), \lambda_2 ( \tau ) )
\]
is a nonempty bounded open interval in $\R$.
Finally, for the generator, one can choose any smooth vector field $Z$ in $\R^{1+1}$ such that $Z |_{ \ell_1 ( \tau ) } = \ell_1' ( \tau )$ and $Z |_{ \ell_2 ( \tau ) } = \ell_2' ( \tau )$.

Finally, for \eqref{eq.gtc_1d_normal}, we simply note that:
\begin{itemize}
\item On the left curve $\ell_1$, the vector $- ( \lambda_1' ( \tau ), 1 )$ is pointing away from $\mc{U}^\ell$.

\item On the right curve $\ell_2$, the vector $( \lambda_2' ( \tau ), 1 )$ is pointing away from $\mc{U}^\ell$.

\item The above two vectors are normal to $\ell_1$ and $\ell_2$---for any $\tau \in \R$,
\[
\left[ \begin{matrix} - \lambda_1' ( \tau ) & - 1 \end{matrix} \right] \left[ \begin{matrix} -1 & 0 \\ 0 & 1 \end{matrix} \right] \left[ \begin{matrix} 1 \\ \lambda_1' ( \tau ) \end{matrix} \right] = 0 \text{,} \qquad \left[ \begin{matrix} \lambda_2' ( \tau ) & 1 \end{matrix} \right] \left[ \begin{matrix} -1 & 0 \\ 0 & 1 \end{matrix} \right] \left[ \begin{matrix} 1 \\ \lambda_2' ( \tau ) \end{matrix} \right] = 0 \text{.}
\]

\item The above two vectors are ($g$-)unit---for any $\tau \in \R$,
\begin{align*}
\left[ \begin{matrix} - \lambda_1' ( \tau ) & - 1 \end{matrix} \right] \left[ \begin{matrix} -1 & 0 \\ 0 & 1 \end{matrix} \right] \left[ \begin{matrix} - \lambda_1' ( \tau ) \\ -1 \end{matrix} \right] &= 1 - | \lambda_1' ( \tau ) |^2 \text{,} \\
\left[ \begin{matrix} \lambda_2' ( \tau ) & 1 \end{matrix} \right] \left[ \begin{matrix} -1 & 0 \\ 0 & 1 \end{matrix} \right] \left[ \begin{matrix} \lambda_2' ( \tau ) \\ 1 \end{matrix} \right] &= 1 - | \lambda_2' ( \tau ) |^2 \text{.}
\end{align*}
\end{itemize}

\subsection{Proof of Theorem \ref{thm.weak_sol}}

Fix a coordinate system $( \mf{t}, y^1, \dots, y^n )$ on $\bar{\mc{U}}$ such that:
\begin{itemize}
\item $y^1, \dots, y^n$ are constant along the integral curves of $Z$.

\item The $g$-gradient $\nabla^\sharp \mf{t}$ is everywhere timelike and past-directed; in particular, $Z \mf{t} > 0$.

\item $\mc{V}$ is precisely the level set $\mc{U} \cap \{ \mf{t} = 0 \}$.
\end{itemize}
In these new coordinates, \eqref{eq.linear_wave} takes the form of a second-order hyperbolic PDE,
\[
- \partial_{ \mf{t} \mf{t} } u + \sum_{ i, j = 1 }^n \partial_{ y^j } ( A^{ i j } \partial_{ y^i } u ) + \sum_{ i = 1 }^n \partial_{ y^i } ( B^i \partial_{ \mf{t} } u ) + \sum_{ i = 1 }^n X^i \partial_{ y^i } u + Y \partial_{ \mf{t} } u + V u = 0 \text{,}
\]
for some smooth coefficients $A^{i j}$, $B^i$, $X^i$, $Y$, $V$, on a static cylinder of the form $\R \times \Sigma$.

Treating $( \phi_0, \phi_1 )$ as initial data on $\Sigma$ and $\phi_b$ (or zero in the first statement of Theorem \ref{thm.weak_sol}) as boundary data on $\R \times \partial \Sigma$, we can now apply the standard well-posedness theory for linear second-order hyperbolic equations; see \cite{ev:pde, lionj_mage:bvp1}.
More specifically, invoking this well-posedness theory in $H^1_0 \times L^2$ and $L^2 \times H^{-1}$ yields existence in both cases of Theorem \ref{thm.weak_sol}.

The above also yields the energy estimates \eqref{eq.weak_sol_energy} and \eqref{eq.dual_sol_energy}, in the case that $\mc{V}_\pm$ are also level sets of $\mf{t}$.
(The $H^1_0 \times L^2$ case also furnishes a ``hidden regularity", in the form an $L^2$-bound on the Neumann data $\mc{N} \phi$.)
Moreover, \eqref{eq.weak_sol_energy} and \eqref{eq.dual_sol_energy} imply that $\phi \in L^2_\mathrm{loc} ( \mc{U} )$.

Suppose now that $\phi_1$ and $\phi_2$ are two solutions of \eqref{eq.linear_wave} (in either statement of Theorem \ref{thm.weak_sol}), and consider the difference $\phi := \phi_2 - \phi_1$, which also solves \eqref{eq.linear_wave}, lies in $L^2_\mathrm{loc} ( \mc{U} )$, and has zero traces on $\mc{V}$ and $\partial \mc{U}$.
Applying a standard unique continuation argument with respect to any ``time" coordinate $\mf{t}$, as described above, yields that $\phi$ vanishes on $\mc{U}$, establishing uniqueness.

Finally, to establish \eqref{eq.weak_sol_energy} and \eqref{eq.dual_sol_energy} in general, we repeat the above argument for another system of coordinates $( \tilde{\mf{t}}, \tilde{y}^1, \dots, \tilde{y}^n )$ such that $\mc{V}$ and $\mc{V}_\pm$ are level sets of $\tilde{\mf{t}}$.
The above uniqueness argument ensures that the solutions obtained in the two coordinate systems coincide.

\subsection{Proof of Theorem \ref{thm.control_hum}}

The proof is largely an adaptation of standard methods to the GTC setting; for references, see, e.g., \cite{micu_zua:control_pde}.
The first step is to reduce the full problem of exact controllability to the slightly simpler question of null controllability.

\begin{lemma} \label{thm.control_null}
Suppose that given any $( \alpha_0^-, \alpha_1^- ) \in L^2 ( \mc{V}_- ) \times H^{-1} ( \mc{V}_- )$, there exists an $\alpha_b \in L^2 ( \partial \mc{U} )$ that is supported in $\Gamma$ and drives the wave equation \eqref{eq.linear_wave} from $( \alpha_0^-, \alpha_0^- )$ to $( 0, 0 )$.
Then, \eqref{eq.linear_wave} is exactly controllable on $\mc{U} ( \mc{V}_-, \mc{V}_+ )$, with control on $\Gamma$.
\end{lemma}

\begin{proof}
Fix $( \phi^\pm_0, \phi^\pm_1 ) \in L^2 ( \mc{V}_\pm ) \times H^{-1} ( \mc{V}_\pm )$.
Let $\alpha$ be the solution of \eqref{eq.linear_wave} satisfying
\[
( \alpha, Z \alpha ) |_{ \mc{V}_+ } = ( \phi^+_0, \phi^+_1 ) \in L^2 ( \mc{V}_+ ) \times H^{-1} ( \mc{V}_+ ) \text{,} \qquad \alpha |_{ \partial \mc{U} } \equiv 0 \text{,}
\]
which can be obtained from Theorem \ref{thm.weak_sol}.
Moreover, we set
\[
( \alpha^-_0, \alpha^-_1 ) := ( \alpha, Z \alpha ) |_{ \mc{V}_- } \text{.}
\]
By our assumption, there exists $\phi_b \in L^2 ( \partial \mc{U} )$, supported in $\Gamma$, and a solution $\beta$ of \eqref{eq.linear_wave} satisfying
\[
( \beta, Z \beta ) |_{ \mc{V}_- } = ( \phi^-_0 - \alpha^-_0, \phi^-_1 - \alpha^-_1 ) \text{,} \qquad ( \beta, Z \beta ) |_{ \mc{V}_+ } = ( 0, 0 ) \text{,} \qquad \beta |_{ \partial \mc{U} ( \mc{V}_-, \mc{V}_+ ) } = \phi_b \text{.}
\]
Then, the solution $\alpha + \beta$ to \eqref{eq.linear_wave} shows that $\phi_b$ drives \eqref{eq.linear_wave} from $( \phi^-_0, \phi^-_1 )$ to $( \phi^+_0, \phi^+_1 )$.
\end{proof}

As a result, it remains only to prove Theorem \ref{thm.control_hum} in the case $( \phi_0^+, \phi_1^+ ) = ( 0, 0 )$.
Toward this end, the following lemma provides the main identity characterizing null controllability:

\begin{lemma} \label{thm.ibp_nc}
Let $N_-$ denote the future-pointing unit normal to $\mc{V}_-$.
Also, we set
\begin{equation}
\label{eq.ibp_XN} \mc{X}_- := g ( \mc{X}, N_- ) \text{,} \qquad N_- := a_- Z + B_- \text{,} \qquad b_- := \operatorname{div}_{ \mc{V}_- } B_- \text{,}
\end{equation}
where $a_- \in C^\infty ( \mc{V}_- )$, where the vector field $B_-$ is everywhere tangent to $\mc{V}_-$, and where $\operatorname{div}_{ \mc{V}_- }$ denotes the divergence on $\mc{V}_-$ with respect to the metric induced by $g$.
In addition, fix
\begin{equation}
\label{eq.ibp_nc_data} ( \phi^-_0, \phi^-_1 ) \in L^2 ( \mc{V}_- ) \times H^{-1} ( \mc{V}_- ) \text{,} \qquad \phi_b \in L^2 ( \partial \mc{U} ) \text{,}
\end{equation}
and assume $\phi_b$ is supported in $\Gamma$.
Then, $\phi_b$ drives \eqref{eq.linear_wave} from $( \phi^-_0, \phi^-_1 )$ to $(0, 0)$ if and only if for any $( \psi^-_0, \psi^-_1 ) \in H^1_0 ( \mc{V}_- ) \times L^2 ( \mc{V}_- )$, we have the identity
\begin{equation}
\label{eq.ibp_nc} \int_\Gamma \phi_b \mc{N} \psi = \int_{ \mc{V}_- } ( a_- \phi^-_1 + 2 B_- \phi^-_0 + b_- \phi^-_0 + \mc{X}_- \phi^-_0 ) \psi^-_0 - \int_{ \mc{V}_- } a_- \phi^-_0 \cdot \psi^-_1 \text{,}
\end{equation}
where $\psi$ denotes the corresponding solution (see Theorem \ref{thm.weak_sol}) of \eqref{eq.control_hum_adj}.
\end{lemma}

\begin{proof}
For convenience, we also let $N_+$ denote the future-pointing unit normal to $\mc{V}_+$, and we set
\begin{equation}
\label{eql.ibp_XN} \mc{X}_+ := g ( \mc{X}, N_+ ) \text{,} \qquad N_+ := a_+ Z + B_+ \text{,} \qquad b_+ := \operatorname{div}_{ \mc{V}_+ } B_+ \text{,}
\end{equation}
with $a_+ \in C^\infty ( \mc{V}_+ )$, $B_+$ a vector field tangent to $\mc{V}_+$, and $b_+$ the induced divergence of $B_+$ on $\mc{V}_+$.
Furthermore, we let $\phi$ be the solution (see again Theorem \ref{thm.weak_sol}) to \eqref{eq.linear_wave} which also satisfies
\begin{equation}
\label{eql.ibp_nc_data} ( \phi, Z \phi ) |_{ \mc{V}_- } = ( \phi^-_0, \phi^-_1 ) \text{,} \qquad \phi |_{ \partial \mc{U} } = \phi_b \text{.}
\end{equation}

By standard approximation arguments, it suffices to only consider the case in which $\phi$ and $\psi$ are both everywhere smooth.
Recalling \eqref{eq.linear_wave} and integrating by parts, we see that
\begin{align*}
0 &= \int_{ \mc{U} ( \mc{V}_-, \mc{V}_+ ) } ( \Box \phi + \nabla_{ \mc{X} } \phi + V \phi ) \psi \\
&= \int_{ \mc{U} ( \mc{V}_-, \mc{V}_+ ) } [ ( - \nabla^\alpha \phi \nabla_\alpha \psi + \nabla^\alpha ( \mc{X}_\alpha \phi ) \cdot \psi + ( V - \nabla^\alpha \mc{X}_\alpha ) \cdot \phi \psi ] + \int_{ \mc{V}_- } N_- \phi \cdot \psi - \int_{ \mc{V}_+ } N_+ \phi \cdot \psi \\
&= \int_{ \mc{U} ( \mc{V}_-, \mc{V}_+ ) } \phi [ \Box \psi - \nabla_{ \mc{X} } \psi + ( V - \nabla^\alpha \mc{X}_\alpha ) \psi ] + \int_{ \mc{V}_- } [ N_- \phi \cdot \psi - \phi \cdot N_- \psi + g ( \mc{X}, N_- ) \cdot \phi \psi ] \\
&\qquad - \int_{ \mc{V}_+ } [ N_+ \phi \cdot \psi - \phi \cdot N_+ \psi + g ( \mc{X}, N_+ ) \cdot \phi \psi ] - \int_{ \partial \mc{U} ( \mc{V}_-, \mc{V}_+ ) } \phi \mc{N} \psi \text{.}
\end{align*}
In the above, we also used that $\psi$ vanishes on $\partial \mc{U} ( \mc{V}_-, \mc{V}_+ )$.

Recalling \eqref{eq.control_hum_adj} and expanding using \eqref{eq.ibp_XN}, \eqref{eql.ibp_XN}, and \eqref{eql.ibp_nc_data} results in the identity
\begin{align*}
\int_\Gamma \phi_b \mc{N} \psi &= \int_{ \mc{V}_- } [ ( a_- Z \phi + B_- \phi ) \psi - \phi ( a_- Z \psi + B_- \psi ) + \mc{X}_- \phi \psi ] \\
&\qquad - \int_{ \mc{V}_+ } [ ( a_+ Z \phi + B_+ \phi ) \psi - \phi ( a_+ Z \psi + B_+ \psi ) + \mc{X}_+ \phi \psi ] \\
&= \int_{ \mc{V}_- } [ ( a_- \phi^-_1 + B_- \phi^-_0 + \mc{X}_- \phi^-_0 ) \psi^-_0 - \phi^-_0 ( a_- \psi^-_1 + B_- \psi^-_0 ) ] \\
\notag &\qquad - \int_{ \mc{V}_+ } [ ( a_+ Z \phi + B_+ \phi + \mc{X}_+ \phi ) \psi - \phi ( a_+ Z \psi + B_+ \psi ) ] \text{.}
\end{align*}
We can now integrate the quantities $\phi \cdot B_- \psi$ on $\mc{V}_\pm$ by parts.
Recalling that $\psi$ has vanishing trace on $\partial \mc{V}_\pm$, we then see that the above identity becomes
\begin{align}
\label{eql.ibp_id} \int_\Gamma \phi_b \mc{N} \psi &= \int_{ \mc{V}_- } ( a_- \phi^-_1 + 2 B_- \phi^-_0 + b_- \phi^-_0 + \mc{X}_- \phi^-_0 ) \psi^-_0 - \int_{ \mc{V}_- } a_- \phi^-_0 \cdot \psi^-_1 \\
\notag &\qquad - \int_{ \mc{V}_+ } ( a_+ Z \phi + 2 B_+ \phi + b_+ \phi + \mc{X}_+ \phi ) \psi + \int_{ \mc{V}_+ } a_+ \phi Z \psi \text{.}
\end{align}

Consider now the last term on the right-hand side of \eqref{eql.ibp_id}:
\begin{equation}
\label{eql.ibp_ip} \mc{I}_+ := - \int_{ \mc{V}_+ } ( a_+ Z \phi + 2 B_+ \phi + b_+ \phi + \mc{X}_+ \phi ) \psi + \int_{ \mc{V}_+ } a_+ \phi Z \psi \text{.}
\end{equation}
Observe that because of the time reversibility of linear wave equations, both $\psi$ and $Z \psi$ can be freely prescribed on $\mc{V}_+$ by setting appropriate values for $( \psi^-_0, \psi^-_1 )$.
Thus, it follows that $\mc{I}_+ = 0$ for all such $( \psi^-_0, \psi^-_1, \psi )$ if and only if $a_+ Z \phi + B_+ \phi - \mc{X}_+ \phi$ and $\phi$ vanish on $\mc{V}_+$.

Since $a_+ \neq 0$, it follows that $\mc{I}_+ = 0$ vanishes if and only if $( \phi, Z \phi ) |_{ \mc{V}_+ }$ vanishes.
This completes the proof of the lemma, since by \eqref{eql.ibp_id}, the condition $\mc{I}_+ = 0$ is equivalent to \eqref{eq.ibp_nc}, whereas $( \phi, Z \phi ) |_{ \mc{V}_+ } \equiv ( 0, 0 )$ if and only if $\phi_b$ drives \eqref{eq.linear_wave} from $( \phi^-_0, \phi^-_1 )$ to $(0, 0)$.
\end{proof}

Now, given $( \phi^-_0, \phi^-_1 ) \in L^2 ( \mc{V}_- ) \times H^{-1} ( \mc{V}_- )$, we define the functional $\mc{J} [ \phi^-_0, \phi^-_1 ]$ from \eqref{eq.control_hum_fct} by
\begin{align}
\label{eql.control_hum_J} \mc{J} [ \phi^-_0, \phi^-_1 ] ( \psi^-_0, \psi^-_1 ) &:= \frac{1}{2} \int_\Gamma | \mc{N} \psi |^2 - \int_{ \mc{V}_- } ( a_- \phi^-_1 + 2 B_- \phi^-_0 + b_- \phi^-_0 + \mc{X}_- \phi^-_0 ) \psi^-_0 \\
\notag &\qquad + \int_{ \mc{V}_- } a_- \phi^-_0 \cdot \psi^-_1 \text{,}
\end{align}
where $\psi$ is defined from $( \psi^-_0, \psi^-_1 )$ via \eqref{eq.control_hum_adj}.
We now connect $\mc{J} [ \phi^-_0, \phi^-_1 ]$ to null controllability:

\begin{lemma} \label{thm.control_J}
Suppose $( \psi^-_0, \psi^-_1 ) \in H^1_0 ( \mc{V}_- ) \times L^2 ( \mc{V}_- )$ is a minimizer of $\mc{J} [ \phi^-_0, \phi^-_1 ]$, and let $\psi$ denote the solution of \eqref{eq.control_hum_adj}, with initial data given by $( \psi^-_0, \psi^-_1 )$.
Then, $\mc{N} \psi |_\Gamma$ lies in $L^2 ( \Gamma )$, and its zero extension to $\partial \mc{U}$ drives \eqref{eq.linear_wave} from $( \phi^-_0, \phi^-_1 )$ to $(0, 0)$.
\end{lemma}

\begin{proof}
That $\mc{N} \psi |_\Gamma \in L^2 ( \Gamma )$ is an immediate consequence of \eqref{eq.weak_sol_energy}, thus we need only show the null control property.
Let $( \alpha^-_0, \alpha^-_1 ) \in H^1_0 ( \mc{V}_- ) \times L^2 ( \mc{V}_- )$, and let $\alpha$ be the corresponding solution to \eqref{eq.control_hum_adj} from initial data $( \alpha^-_0, \alpha^-_1 )$.
Since $( \psi^-_0, \psi^-_1 )$ minimizes $\mc{J} [ \phi^-_0, \phi^-_1 ]$, then for any $h \in \R$,
\begin{align*}
0 &= \lim_{ h \rightarrow 0 } \frac{1}{h} \{ \mc{J} [ \phi^-_0, \phi^-_1 ] ( \psi^-_0 + h \alpha^-_0 , \psi^-_1 + h \alpha^-_0 ) - \mc{J} [ \phi^-_0, \phi^-_1 ] ( \psi^-_0, \psi^-_1 ) \} \\
&= \lim_{ h \rightarrow 0 } \frac{1}{2 h} \int_\Gamma [ | \mc{N} ( \psi + h \alpha ) |^2 - | \mc{N} \psi |^2 ] - \int_{ \mc{V}_- } ( a_- \phi^-_1 + 2 B_- \phi^-_0 + b_- \phi^-_0 + \mc{X}_- \phi^-_0 ) \alpha^-_0 \\
&\qquad + \int_{ \mc{V}_- } a_- \phi^-_0 \cdot \alpha^-_1 \\
&= \int_\Gamma \mc{N} \psi \mc{N} \alpha - \int_{ \mc{V}_- } ( a_- \phi^-_1 + 2 B_- \phi^-_0 + b_- \phi^-_0 + \mc{X}_- \phi^-_0 ) \alpha^-_0 + \int_{ \mc{V}_- } a_- \phi^-_0 \cdot \alpha^-_1 \text{.}
\end{align*}
Since the above holds for any $( \alpha^-_0, \alpha^-_1 )$, then Lemma \ref{thm.ibp_nc} completes the proof.
\end{proof}

Therefore, we have reduced the problem of controllability to finding a minimizer of $\mc{J} [ \phi^-_0, \phi^-_1 ]$:

\begin{lemma} \label{thm.control_J_min}
$\mc{J} [ \phi^-_0, \phi^-_1 ]$ has a unique minimizer $( \psi^-_0, \psi^-_1 ) \in H^1_0 ( \mc{V}_- ) \times L^2 ( \mc{V}_- )$.
\end{lemma}

\begin{proof}
By the estimates \eqref{eq.weak_sol_energy} and the linearity of \eqref{eq.linear_wave}, we obtain that $\mc{J} [ \phi^-_0, \phi^-_1 ]$ is continuous.
Furthermore, the observability estimate \eqref{eq.control_hum_obs} yields constants $c, C > 0$ such that
\begin{align*}
\mc{J} [ \phi^-_0, \phi^-_1 ] ( \alpha^-_0, \alpha^-_1 ) &\geq c [ \| \alpha^-_0 \|_{ H^1 ( \mc{V}_- ) }^2 + \| \alpha^-_1 \|_{ L^2 ( \mc{V}_- ) }^2 ] \\
&\qquad - C [ \| \phi^-_0 \|_{ L^2 ( \mc{V}_- ) } + \| \phi^-_1 \|_{ H^{-1} ( \mc{V}_- ) } ] [ \| \alpha^-_0 \|_{ H^1 ( \mc{V}_- ) } + \| \alpha^-_1 \|_{ L^2 ( \mc{V}_- ) } ] \text{,}
\end{align*}
for all $( \alpha^-_0, \alpha^-_1 ) \in H^1_0 ( \mc{V}_- ) \times L^2 ( \mc{V}_- )$, which establishes that $\mc{J} [ \phi^-_0, \phi^-_1 ]$ is coercive.

We next claim that $\mc{J} [ \phi^-_0, \phi^-_1 ]$ is strictly convex.
Indeed, given any
\[
\lambda \in (0, 1) \text{,} \qquad ( \alpha^-_0, \alpha^-_1 ), ( \beta^-_0, \beta^-_1 ) \in H^1_0 ( \mc{V}_- ) \times L^2 ( \mc{V}_- ) \text{,} \qquad ( \alpha^-_0, \alpha^-_1 ) \neq ( \beta^-_0, \beta^-_1 ) \text{,}
\]
and letting $\alpha$, $\beta$ be the solutions of \eqref{eq.control_hum_adj}, with data $( \alpha^-_0, \alpha^-_1 )$ and $( \beta^-_0, \beta^-_1 )$, respectively, we have
\begin{align*}
&\lambda \mc{J} [ \phi^-_0, \phi^-_1 ] ( \alpha^-_0, \alpha^-_1 ) + ( 1 - \lambda ) \mc{J} [ \phi^-_0, \phi^-_1 ] ( \beta^-_0, \beta^-_1 ) \\
&\qquad - \mc{J} [ \phi^-_0, \phi^-_1 ] ( \lambda \alpha^-_0 + ( 1 - \lambda ) \beta^-_0, \lambda \alpha^-_1 + ( 1 - \lambda ) \beta^-_1 ) \\
&\quad = \frac{1}{2} \int_\Gamma \{ \lambda | \mc{N} \alpha |^2 + ( 1 - \lambda ) | \mc{N} \beta |^2 - | \mc{N} [ \lambda \alpha + ( 1 - \lambda ) \beta ] |^2 \} \\
&\quad = \frac{ \lambda ( 1 - \lambda ) }{ 2 } \int_\Gamma | \mc{N} ( \alpha - \beta ) |^2 \text{.}
\end{align*}
Applying the observability inequality \eqref{eq.control_hum_obs} to $\alpha - \beta$ and $( \alpha^-_0 - \beta^-_0, \alpha^-_1 - \beta^-_1 )$, we see that the right-hand side of the above is strictly positive, hence $\mc{J} [ \phi^-_0, \phi^-_1 ]$ is strictly convex.

Via the direct method in the calculus of variations (see, for instance, \cite{daco:direct_calcvar}), the above considerations then imply that $\mc{J} [ \phi^-_0, \phi^-_1 ]$ indeed has a unique minimizer.
\end{proof}

Let $( \psi^-_0, \psi^-_1 )$ now denote the minimizer from Lemma \ref{thm.control_J_min}, and let $\psi$ denote the corresponding solution to \eqref{eq.control_hum_adj}.
By Lemma \ref{thm.control_J}, the zero extension of $\mc{N} \psi |_\Gamma$ drives \eqref{eq.linear_wave} from $( \phi^-_0, \phi^-_1 )$ to $(0, 0)$, and hence Lemma \ref{thm.control_null} implies exact controllability for \eqref{eq.linear_wave} on $\mc{U} ( \mc{V}_-, \mc{V}_+ )$, with control on $\Gamma$.

Finally, suppose $\phi_b \in L^2 ( \partial \mc{U} )$ also is supported in $\Gamma$ and drives \eqref{eq.linear_wave} from $( \phi^-_0, \phi^-_1 )$ to $(0, 0)$.
Applying Lemma \ref{thm.ibp_nc} twice, first to the control $\mc{N} \psi |_\Gamma$ and then to $\phi_b$, we obtain that
\begin{align*}
\int_\Gamma | \mc{N} \psi |^2 &= \int_{ \mc{V}_- } ( a_- \phi^-_1 + 2 B_- \phi^-_0 + b_- \phi^-_0 + \mc{X}_- \phi^-_0 ) \psi^-_0 - \int_{ \mc{V}_- } a_- \phi^-_0 \cdot \psi^-_1 \\
&= \int_\Gamma \phi_b \mc{N} \psi \\
&\leq \| \phi_b \|_{ L^2 ( \Gamma ) } \| \mc{N} \psi \|_{ L^2 ( \Gamma ) } \text{,}
\end{align*}
and the inequality \eqref{eq.control_hum_min} follows immediately.

\fi

\raggedright
\raggedbottom
\bibliographystyle{amsplain}
\bibliography{articles,books,misc}

\end{document}